\documentclass[12pt]{amsart}
\usepackage{amsfonts, amsthm, amssymb, amsmath, stmaryrd}
\usepackage{mathrsfs,array}
\usepackage{xy}
\usepackage{hyperref}
\usepackage{verbatim}
\input xy
\xyoption{all}

\newenvironment{altenumerate}
   {\begin{list}
      {\textup{(\theenumi)} }
      {\usecounter{enumi}
       \setlength{\labelwidth}{0pt}
       \setlength{\labelsep}{2pt}
       \setlength{\leftmargin}{0pt}
       \setlength{\itemsep}{\the\smallskipamount}
       \renewcommand{\theenumi}{\roman{enumi}}
      }}
   {\end{list}}

\def\from{\colon}

\newcommand{\PD}{\mathrm{PD}}

\DeclareMathOperator{\im}{im}
\def\C{\mathbb{C}}
\def\B{\mathbb{B}}
\def\F{\mathbb{F}}
\def\G{\mathbf{G}}

\def\Q{\mathbb{Q}}
\def\R{\mathbb{R}}

\def\Z{\mathbb{Z}}
\def\A{\mathbb{A}}
\def\FF{\mathbb{F}}
\def\OO{\mathcal{O}}

\def\MM{\mathbf{M}}

\def\M{\mathcal{M}}
\def\op{\mathrm{op}}
\def\ad{\mathrm{ad}}
\def\wa{\mathrm{wa}}
\def\a{\mathrm{a}}
\def\top{\mathrm{top}}
\def\naive{\mathrm{naive}}
\def\isog{\mathrm{isog}}

\def\cont{\mathrm{cont}}
\def\int{\mathrm{int}}

\def\gothm{\mathfrak{m}}

\def\isom{\cong}
\def\Ga{\mathbf{G}_{\text{a}}}
\def\Gm{\mathbf{G}_{\text{m}}}
\newcommand{\powerseries}[1]{\llbracket #1 \rrbracket}
\newcommand{\laurentseries}[1]{(\!(#1)\!)}
\newcommand{\injects}{\hookrightarrow}

\DeclareMathOperator{\qlog}{qlog}
\DeclareMathOperator{\Gal}{Gal}
\DeclareMathOperator{\Sym}{Sym}

\DeclareMathOperator{\End}{End}
\DeclareMathOperator{\Aut}{Aut}
\DeclareMathOperator{\Lie}{Lie}

\newcommand{\cris}{\mathrm{cris}}
\DeclareMathOperator{\Spec}{Spec}
\DeclareMathOperator{\Proj}{Proj}
\DeclareMathOperator{\Spa}{Spa}

\DeclareMathOperator{\Div}{Div}
\DeclareMathOperator{\Hom}{Hom}

\DeclareMathOperator{\Spf}{Spf}
\DeclareMathOperator{\GL}{GL}

\newcommand{\LT}{\mathrm{LT}}

\newcommand{\Dr}{\mathrm{Dr}}

\DeclareMathOperator{\Nilp}{Nilp}
\DeclareMathOperator{\Aff}{Aff}
\DeclareMathOperator{\CAff}{CAff}
\DeclareMathOperator{\Sets}{Sets}
\DeclareMathOperator{\Adic}{Adic}
\DeclareMathOperator{\Def}{Def}

\newcommand{\Flag}{\mathscr{F}\!\ell}
\newcommand{\GM}{\mathrm{GM}}
\newcommand{\HT}{\mathrm{HT}}

\newcommand{\et}{\text{\'et}}
\newcommand{\abs}[1]{\left\lvert #1 \right\rvert}
\newcommand{\class}[1]{\left< #1 \right>}
\newcommand{\set}[1]{\left\{ #1 \right\}}

\newcommand{\tatealgebra}{\class}

\numberwithin{equation}{subsection}
\newtheorem{Theorem}{Theorem}[subsection]
\numberwithin{Theorem}{subsection}
\newtheorem{lemma}[Theorem]{Lemma}
\newtheorem{Cor}[Theorem]{Corollary}
\newtheorem{prop}[Theorem]{Proposition}
\newtheorem{defprop}[Theorem]{Definition/Proposition}
\newtheorem{conj}[Theorem]{Conjecture}
\newtheorem*{thma}{Theorem A}
\newtheorem*{thmb}{Theorem B}
\newtheorem*{thmc}{Theorem C}
\newtheorem*{thmd}{Theorem D}
\newtheorem*{thme}{Theorem E}

\newtheorem{defn}[Theorem]{Definition}

\theoremstyle{definition}
\newtheorem{rmk}[Theorem]{Remark}
\newtheorem{exmp}[Theorem]{Example}

\begin{document}
\title{Moduli of $p$-divisible groups}
\author{Peter Scholze and Jared Weinstein}
\begin{abstract}
We prove several results about moduli spaces of $p$-divisible groups such as Rapoport--Zink spaces. Our main goal is to prove that Rapoport--Zink spaces at infinite level carry a natural structure as a perfectoid space, and to give a description purely in terms of $p$-adic Hodge theory of these spaces. This allows us to formulate and prove duality isomorphisms between basic Rapoport--Zink spaces at infinite level in general. Moreover, we identify the image of the period morphism, reproving results of Faltings, \cite{FaltingsPeriodDomains}. For this, we give a general classification of $p$-divisible groups over $\OO_C$, where $C$ is an algebraically closed complete extension of $\Q_p$, in the spirit of Riemann's classification of complex abelian varieties. Another key ingredient is a full faithfulness result for the Dieudonn\'e module functor for $p$-divisible groups over semiperfect rings (i.e. rings on which the Frobenius is surjective).
\end{abstract}

\maketitle
\tableofcontents
\pagebreak

\section{Introduction}  In this article we prove several results about $p$-divisible groups and their moduli.  Our work touches upon various themes known to arise in the context of $p$-divisible groups, including the crystalline Dieudonn\'e module (\cite{Messing}, \cite{BerthelotBreenMessing}), the rigid-analytic moduli spaces of Rapoport-Zink and their associated period maps (\cite{RZ}), and the more recent work of Fargues and Fontaine on the fundamental curve of $p$-adic Hodge theory (\cite{FarguesFontaine}).   The theory of perfectoid spaces (\cite{Sch}) arises repeatedly in our constructions and arguments, and in particular it provides the appropriate context for studying Rapoport-Zink spaces at infinite level.

Our first result concerns the full faithfulness of the crystalline Dieudonn\'e module functor.  Given a ring $R$ in which $p$ is nilpotent, and a $p$-divisible group $G$ over $R$, let $\MM(G)$ denote its (covariant) Dieudonn\'e crystal, as defined in \cite{Messing}.  The question of the full faithfulness of $\MM$ has a long history (see \cite{deJongMessing} for a survey), and is known to have an affirmative answer in a number of situations, for instance when $R$ is an excellent ring in characteristic $p$ which is a locally complete intersection ring (loc. cit.).

In contrast, our result applies to rings in characteristic $p$ which are typically non-noetherian.  We say a ring $R$ in characteristic $p$ is {\em semiperfect} if the Frobenius map $\Phi\from R\to R$ is surjective.  Such a ring is called {\em f-semiperfect} if the topological ring $\varprojlim_{\Phi}R$ has a finitely generated ideal of definition.  For instance, the quotient of a perfect ring by a finitely generated ideal is f-semiperfect. An important example is $\OO_C/p$, where $C$ is any algebraically closed extension of $\Q_p$.

\begin{thma} Let $R$ be an f-semiperfect ring.  Then the Dieudonn\'e module functor on $p$-divisible groups up to isogeny is fully faithful. If moreover $R$ is the quotient $R=S/J$ of a perfect ring $S$ by a regular ideal $J\subset S$, then the Dieudonn\'e module functor on $p$-divisible groups is fully faithful.
\end{thma}

\begin{rmk} In this case, a Dieudonn\'e module can be made very explicit: In general, if $R$ is a semiperfect ring, we have the Fontaine rings $A_{\cris}(R)$ and $B_{\cris}^+(R)$, and the functor $\MM\mapsto M=\MM(A_{\cris}(R))$ (resp. $\MM\mapsto M=\MM(A_\cris(R))[p^{-1}]$) induces an equivalence of categories between the category of Dieudonn\'e crystals over $R$ (resp., up to isogeny) and the category of finite projective $A_\cris(R)$ (resp. $B_{\cris}^+(R)$)-modules equipped with Frobenius and Verschiebung maps.
\end{rmk}

We remark that the case of perfect fields is classical, and that the case of perfect rings is an unpublished result due to Gabber, relying on a result of Berthelot for the case of perfect valuation rings, \cite{BerthelotValuationRings}. Recently, this was reproved by Lau, \cite{LauTruncatedDisplays}. In all of these results, one knows even that the Dieudonn\'e module functor is essentially surjective, and one also has a description of finite locally free group schemes.

Our method of proof is entirely different. It handles first the case of morphisms $\Q_p/\Z_p\to \mu_{p^\infty}$ by an explicit computation involving integrality of the Artin-Hasse exponential, and reduces to this case by a trick we learned from a paper of de Jong and Messing, \cite{deJongMessing}. We note that this reduction step only works for morphisms of $p$-divisible groups (and not for finite locally free group schemes), and even if one is only interested in the result for perfect rings, this reduction step will introduce rather arbitrary f-semiperfect rings. In this reduction step, a certain technical result is necessary, for which we give a proof that relies on Faltings's almost purity theorem in the form given in \cite{Sch}. As explained below, Theorem A has direct consequences in $p$-adic Hodge theory, so that this state of affairs may be reasonable.

Of particular interest are the morphisms $\Q_p/\Z_p\to G$.  Observe that if $S$ is any $R$-algebra, then the $\Q_p$-vector space $\Hom_S(\Q_p/\Z_p,G)[1/p]$ is identified with the inverse limit $\tilde{G}(S)=\varprojlim G(S)$ (taken with respect to multiplication by $p$).  Then $\tilde{G}$ is a functor from $R$-algebras to $\Q_p$-vector spaces which we call the {\em universal cover} of $G$.  An important observation is that if $S\to R$ is a surjection with nilpotent kernel, and if $G_S$ is a lift of $G$ to $S$, then $\tilde{G}_S(S)$ is canonically isomorphic to $\tilde{G}(R)$ (and so does not depend on the choice of $G_S$).  In other words, $\tilde{G}$ is a crystal on the infinitesimal site of $R$.

Using this construction, we can explain how Theorem A is closely related to $p$-adic Hodge theory. In fact, take $R=\OO_C/p$, where $C$ is an algebraically closed complete extension of $\Q_p$, and let $G$ be any $p$-divisible group over $\OO_C$, with reduction $G_0$ to $\OO_C/p$. Let $M(G)$ be the finite projective $B_\cris^+$-module given by evaluating the Dieudonn\'e crystal of $G$ on $A_\cris = A_\cris(\OO_C/p)$, and inverting $p$. Then, using Theorem A for morphisms $\Q_p/\Z_p\to G_0$ up to isogeny, one finds that
\[
M(G)^{\varphi = p} = \Hom(\Q_p/\Z_p,G_0)[p^{-1}] = \tilde{G}(\OO_C/p) = \tilde{G}(\OO_C)\ .
\]
On the other hand, there is an exact sequence
\begin{equation}\label{TGW}
0\to T(G)[p^{-1}]\to \tilde{G}(\OO_C)\to \Lie G\otimes C\to 0\ ,
\end{equation}
where the latter map is the logarithm map, and $T(G)$ denotes the Tate module of $G$. This translates into the exact sequence
\[
0\to T(G)[p^{-1}]\to M(G)^{\varphi = p}\to \Lie G\otimes C\to 0
\]
relating the \'etale and crystalline homology of $G$. We remark that Theorem A also handles the relative case of this result, and also the integral version.

Our second main result is a classification of $p$-divisible groups over $\OO_C$, where $C$ is still an algebraically closed complete extension of $\Q_p$. We recall that one has a Hodge-Tate sequence
\[
0\to \Lie G\otimes C(1)\to T(G)\otimes C\to (\Lie G^\vee)^\vee\otimes C\to 0\ ,
\]
where $(1)$ denotes a Tate twist.

\begin{thmb} There is an equivalence of categories between the category of $p$-divisible groups over $\OO_C$ and the category of free $\Z_p$-modules $T$ of finite rank together with a $C$-subvectorspace $W$ of $T\otimes C(-1)$.
\end{thmb}

The equivalence carries $G$ to the pair $(T,W)$, where $T$ is the Tate module of $G$ and $W=\Lie G\otimes C$. In particular, the result is a classification in terms of linear algebra, instead of $\sigma$-linear algebra as usually, and (related to that) is in terms of its \'etale cohomology, instead in terms of its crystalline cohomology. We note that Theorem B can be regarded as a $p$-adic analogue of Riemann's classification of complex abelian varieties, which are classified by their singular homology together with the Hodge filtration.

Results in the direction of Theorem B have previously been obtained by Fargues in \cite{FarguesPDivGroups}. In particular, the fully faithfulness part is proved already there. We prove essential surjectivity first in the case that $C$ is spherically complete, with surjective norm map $C\to \R_{\geq 0}$. In that case, the argument is very direct; the key idea is to look first at the rigid-analytic generic fibre of the (connected, say) $p$-divisible group (considered as a formal scheme), which is an open ball and can be constructed directly from $(T,W)$. For the converse, one has to see that a certain rigid-analytic variety is an open ball. In general, one can show that it is an increasing union of closed balls; under the assumptions on $C$, one can conclude that it is an open ball. Afterwards, we make a (somewhat indirect) descent argument involving Rapoport-Zink spaces (and Theorem C) to deduce the general case.

It will be crucial to relate the exact sequence in Eq. \eqref{TGW} to vector bundles on the Fargues-Fontaine curve $X$, which is the object investigated in \cite{FarguesFontaine}.  This is defined as $X=\Proj P$, where $P$ is the graded $\Q_p$-algebra
\[ P=\bigoplus_{d\geq 0} (B^+_{\cris})^{\varphi=p^d}, \]
 and $B^+_{\cris}=B^+_{\cris}(\OO_C/p)$.  There is a special point $\infty\in X$ corresponding to the homomorphism $\Theta\from B^+_{\cris}\to C$;  we write $i_\infty\from \set{\infty}\to X$ for the inclusion.  For every isocrystal $M$ over $\bar{\FF}_p$, there is a corresponding vector bundle on $\OO_C$, given by the vector bundle associated to $\bigoplus_{d\geq 0} (M\otimes B^+_{\cris})^{\varphi=p^d}$, and Fargues and Fontaine show that all the vector bundles on $X$ arise in this way.  In particular there is a bijection $H\mapsto \mathscr{E}(H)$ between isogeny classes of $p$-divisible groups over $\bar{\FF}_p$ and vector bundles on $X$ whose slopes are between $0$ and $1$.  Note that $i_{\infty}^*\mathscr{E}(H)=M(H)\otimes C$.

Using this classification of vector bundles, we show in Theorem \ref{FactsFarguesFontaine} that every $p$-divisible group $G$ over $\OO_C$ is {\em isotrivial} in the sense that  there exists a $p$-divisible group $H$ over $\bar{\FF}_p$ and a quasi-isogeny
\[ \rho \from H\otimes_{\bar{\FF}_p}\OO_C/p\to G\otimes_{\OO_C} \OO_C/p. \]
Thus $(G,\rho)$ is a deformation of $H$ in the sense of Rapoport-Zink. Appealing to Theorem A, we can identify $\tilde{G}(\OO_C)\isom \tilde{H}(\OO_C/p)\isom (M(H)\otimes B_{\cris}^+)^{\varphi=p}$ with the space of global sections of $\mathscr{E}=\mathscr{E}(H)$.  The exact sequence in Eq. \eqref{TGW} appears when one takes global sections in the exact sequence of coherent sheaves on $X$,
\begin{equation}
\label{FEW}
0\to \mathscr{F}\to\mathscr{E}\to i_{\infty\ast}(\Lie G\otimes C) \to 0,
\end{equation}
where $\mathscr{F}=T\otimes_{\Z_p}\OO_X$, see Proposition \ref{PDivGroupGivesModVectBund}.  To summarize the situation, a $p$-divisible group $H$ gives a vector bundle $\mathscr{E}$ on $X$, while a deformation of $H$ to $\OO_C$ gives a {\em modification} of $\mathscr{E}$ (in the sense that $\mathscr{F}$ and $\mathscr{E}$ are isomorphic away from $\infty$).

We should note that vector bundles over $X$ are equivalent to $\varphi$-modules over the Robba ring, and that our use of $X$ could be replaced by Kedlaya's theory of $\varphi$-modules over the Robba ring, \cite{KedlayaRobbaRing}.

We now turn to Rapoport-Zink spaces and their associated period maps, as in \cite{RZ}.  Let $H$ be a $p$-divisible group of dimension $d$ and height $h$ over a perfect field $k$ of characteristic $p$.  Let $\mathcal{M}$ be the associated Rapoport-Zink space.  This is a formal scheme parametrizing deformations $(G,\rho)$ of $H$, see Definition \ref{deformation}.

Passing to the generic fibre, we get an adic space $\mathcal{M}_{\eta}^{\ad}$.  Note that \cite{RZ} uses rigid spaces, but we work with adic spaces, as it is important for us to consider non-classical points (in particular to talk about the image of the period morphism later).\footnote{As in our previous work, we prefer adic spaces over Berkovich spaces.} We remark that if $(R,R^+)$ is a complete affinoid $(W(k)[1/p],W(k))$-algebra, then a morphism $\Spa(R,R^+)\to\mathcal{M}_{\eta}^{\ad}$ corresponds to a covering of $\Spa(R,R^+)$ by open affinoids $\Spa(R_i,R_i^+)$, together with a deformation $(G_i,\rho_i)$ of $H$ to an open bounded subring of each $R_i^+$.  We will simply refer to such a datum as a pair $(G,\rho)$.  If $(K,K^+)$ is an affinoid field, then $\mathcal{M}_{\eta}^{\ad}(K,K^+)$ is the set of deformations of $H$ to $K^+$.

There is a (Grothendieck-Messing) {\em period morphism} $\pi\from \mathcal{M}_{\eta}^{\ad}\to \Flag$, where $\mathcal{M}_{\eta}^{\ad}$ is the adic generic fibre of $\mathcal{M}$ and $\Flag$ is the Grassmannian variety of $d$-dimensional quotients of $M(H)\otimes\Q_p$ (considered as an adic space).  If $x$ is a $(C,\OO_C)$-point of $\mathcal{M}_{\eta}^{\ad}$ corresponding to a deformation $(G,\rho)$ of $H$ to $\OO_C$, then $\pi(x)\in\Flag(C,\OO_C)$ corresponds to the quotient $M(H)\otimes C\to W$ given by Grothendieck-Messing theory.

Our third main theorem provides a description of the image of $\pi$. A $(C,\OO_C)$-valued point $x$ of $\Flag$ corresponds to a $d$-dimensional quotient $M(H)\otimes C\to W$.  Let $\mathscr{E}=\mathscr{E}(H)$ be the vector bundle over $X$ associated to $H$, so that $i_{\infty}^*\mathscr{E}=M(H)\otimes C$.  The map $i_{\infty}^\ast\mathscr{E}\to W$ induces a map $\mathscr{E}\to i_{\infty\ast}i_\infty^\ast\mathscr{E}\to i_{\infty\ast}W$.  Let $\mathscr{F}$ be the kernel, so that we have a modification of $\mathscr{E}$ as in Eq. \eqref{FEW}.

\begin{thmc} The point $x$ is in the image of $\pi$ if and only if $\mathscr{F}\isom\OO_X^h$.
\end{thmc}

In order to prove this theorem, it is enough to consider the case where $C$ is spherically complete and the norm map $C\to \R_{\geq 0}$ is surjective. In that case, Theorem B is available to construct the desired $p$-divisible group $G$. The construction of the isogeny $\rho$ over $\OO_C/p$ relies on Theorem A. On the other hand, Theorem C can be used to deduce Theorem B for general $C$, by constructing the desired $p$-divisible group within some Rapoport-Zink space.

Let us make some historical remarks on Theorem C. In their book \cite{RZ}, Rapoport and Zink constructed the period morphism $\pi$, and showed that the image is contained in the weakly admissible set $\Flag^\wa\subset \Flag$, which is explicitly described as the complement of lower-dimensional subvarieties. This generalizes previous known cases, notably the case of Drinfeld's upper half-space, \cite{DrinfeldUpperHalfSpace}, and the Gross-Hopkins period map from Lubin-Tate space to projective space, \cite{GrossHopkins}. Rapoport and Zink conjectured that there is a natural rigid-analytic variety $\Flag^\prime$ with an \'etale and bijective map $\Flag^\prime\to \Flag^\wa$, over which the period morphism factors, and over which there exists a $\Q_p$-local system (coming from the rational Tate module of the universal $p$-divisible group). In the language of adic spaces, $\Flag^\prime$ should be an open subset of $\Flag^\wa$, with the same classical points. As $\pi$ is \'etale and \'etale maps are open (in the adic world), one can simply define $\Flag^\prime$ to be the image of $\pi$ (as we shall do in the discussion to follow); but then, one does not yet know what its classical points are, and the description is very inexplicit.

On the level of classical points, a theorem of Colmez and Fontaine, \cite{ColmezFontaine}, shows that weakly admissible implies admissible (meaning that any weakly admissible filtered isocrystal comes from a crystalline representation), so that for a classical point of $\Flag^\wa$, one gets an associated crystalline representation with Hodge-Tate weights only $0$ and $1$. On the other hand, a theorem of Breuil, \cite{Breuil}, and Kisin, \cite{Kisin}, shows that all crystalline representations with Hodge-Tate weights only $0$ and $1$ come from $p$-divisible groups. We note that their results rely on deep techniques in integral $p$-adic Hodge theory. Taken these results together, one finds that indeed the classical points of $\Flag^\prime$ agree with the classical points of $\Flag^\wa$.

Then Hartl, after handling the case of equal characteristic, \cite{HartlEqualChar}, used Kedlaya's theory of $\varphi$-modules over the Robba ring to construct an open subset $\Flag^\a\subset \Flag^\wa$, called the admissible locus, cf. \cite{HartlAdmLocus}.\footnote{He uses Berkovich spaces there, but one gets a corresponding adic space as well, by using that the Berkovich space is the maximal hausdorff quotient of the adic space, for example.} Roughly, it is the locus where the universal filtered $\varphi$-module over $\Flag$ comes from a $\Q_p$-local system (at least punctually). Hartl conjectured that $\Flag^\prime = \Flag^\a$. This was proved by Faltings, \cite{FaltingsPeriodDomains}, by first showing that the universal filtered $\varphi$-module over $\Flag^\a$ comes from a $\Q_p$-local system globally, and then using this local system to spread the $p$-divisible groups constructed from the results of Breuil and Kisin at classical points to the whole admissible locus.

We note that our approach is entirely different, in that we bypass both the results of Breuil and Kisin (i.e., integral $p$-adic Hodge theory), and the need to construct a $\Q_p$-local system first (i.e., relative $p$-adic Hodge theory). Instead, we produce the desired $p$-divisible groups directly from Theorem B. As a consequence, we get a new proof of the result of Breuil and Kisin, using only rational $p$-adic Hodge theory.

By adding level structures, one constructs a tower of covers $\mathcal{M}_n$ of $\mathcal{M}^{\ad}_{\eta}$.   The $(R,R^+)$-points of $\mathcal{M}_n$ are triples $(G,\rho,\alpha)$, where $(G,\rho)$ is a deformation of $H$ and $\alpha\from (\Z/p^n\Z)^h\to G[p^n]^{\ad}_{\eta}(R,R^+)$ is a homomorphism which induces an isomorphism at every point $x=\Spa(K,K^+)$ of $\Spa(R,R^+)$.

It is then quite natural to define a Rapoport-Zink moduli problem at infinite level.  Let $\mathcal{M}_\infty$ be the functor on complete affinoid $(W(k)[1/p],W(k))$-algebras which assigns to $(R,R^+)$ the set of triples $(G,\rho,\alpha)$, where $(G,\rho)\in \mathcal{M}_{\eta}^{\ad}(R,R^+)$, and $\alpha\from \Z_p^h\to T(G)^{\ad}_{\eta}(R,R^+)$ is a homomorphism which induces an isomorphism at every point $x=\Spa(K,K^+)$ of $\Spa(R,R^+)$.

Our fourth main theorem shows that $\mathcal{M}_\infty$ is representable by an adic space which has an alternative description independent of deformations of $H$. Recall that the universal cover $\tilde{H}$ of $H$ is a crystal of $\Q_p$-vector spaces on the infinitesimal site of $k$.  In the following we consider $\tilde{H}$ as a formal scheme over $\Spf W(k)$, which represents the functor $R\mapsto \varprojlim H'(R)$ (where $H'$ is any lift of $H$ to $W(k)$).  Then $\tilde{H}^{\ad}_{\eta}$ is a $\Q_p$-vector space object in the category of adic spaces over $\Spa(W(k)[1/p],W(k))$. Note that any point $(G,\rho,\alpha)$ of $\mathcal{M}_\infty$ gives a map
\[
\Z_p^h\to T(G)^\ad_\eta\subset \tilde{G}^\ad_\eta\cong \tilde{H}^\ad_\eta\ ,
\]
where the first map is $\alpha$, and the isomorphism is induced from $\rho$. In other words, all $p$-divisible groups $G$ parametrized by $\mathcal{M}$ have the same universal cover $\tilde{H}$, and are of the form $\tilde{H} / \Lambda$, where $\Lambda\subset \tilde{H}$ is a $\Z_p$-lattice. The following theorem answers the question for which $\Lambda\subset \tilde{H}$ one can form the quotient $\tilde{H} / \Lambda$. The description uses the quasi-logarithm map $\qlog\from\tilde{H}^{\ad}_{\eta}\to M(H)\otimes\Ga$. We remark that for any deformation $(G,\rho)$ of $H$, we have a commutative diagram
\[\xymatrix{
\tilde{H}^\ad_\eta\ar^{\qlog}[rr]\ar[d]&&M(H)\otimes\Ga\ar[d]\\
G^\ad_\eta\ar^\log[rr]&&\Lie G\otimes\Ga
}\]
In particular, $\Lambda = \ker(\tilde{H}^\ad_\eta\to G^\ad_\eta)$ maps to the kernel of $M(H)\to \Lie G$, which explains the condition on the rank in the following theorem.

\begin{thmd} The moduli problem $\mathcal{M}_{\infty}$ is representable by an adic space, and is isomorphic to the functor which assigns to a complete affinoid $(W(k)[1/p],W(k))$-algebra $(R,R^+)$ the set of $h$-tuples $s_1,\dots,s_h\in \tilde{H}^{\ad}_{\eta}(R,R^+)$ for which the following conditions are satisfied.
\begin{altenumerate}
\item[{\rm (i)}] The cokernel $W$ of the map
\[
R^h\xrightarrow{(\qlog(s_1),\ldots,\qlog(s_h))}M(H)\otimes R
\]
is a projective $R$-module of rank $d$.
\item[{\rm (ii)}] For all geometric points $x=\Spa(C,\OO_C)\to \Spa(R,R^+)$, the sequence
\[
0\to \Q_p^h\xrightarrow{(s_1,\ldots,s_h)} \tilde{H}^\ad_\eta(C,\OO_C)\to W\otimes_R C\to 0
\]
is exact.
\end{altenumerate}

Moreover, $\mathcal{M}_\infty\subset (\tilde{H}^\ad_\eta)^h$ is a locally closed subspace, the space $\mathcal{M}_{\infty}$ is preperfectoid, and $\mathcal{M}_{\infty} \sim \varprojlim_n \mathcal{M}_n$.
\end{thmd}

The fact that we can describe $\mathcal{M}_\infty$ explicitly independently of $p$-divisible groups relies on Theorem C.\footnote{We note that using Theorem A, one can make explicit what $\tilde{H}^\ad_\eta(R,R^+)$ is in terms of Dieudonn\'e theory, at least when $(R,R^+)$ is a perfectoid affinoid $(K,\OO_K)$-algebra for some perfectoid field $K$. This gives a description of $\mathcal{M}_\infty$ purely in terms of $p$-adic Hodge theory.} The last part of the theorem says more explicitly that for any perfectoid field extension $K$ of $W(k)[p^{-1}]$, the base-change $\mathcal{M}_{\infty,K}$ of $\mathcal{M}_\infty$ to $\Spa(K,\OO_K)$ has a natural completion $\hat{\mathcal{M}}_{\infty,K}$, which is a perfectoid space over $K$, and one has an equivalence of \'etale topoi
\[
\hat{\mathcal{M}}_{\infty,K,\et}\sim \varprojlim_n \mathcal{M}_{n,K,\et}
\]
(at least after base-change to a quasicompact open subset of some $\mathcal{M}_{n,K}$, so as to work with qcqs topoi, where projective limits are well-defined). We remark that taking inverse limits in the category of adic spaces is not canonical, corresponding to the phenomenon that on a direct limit $\varinjlim A_i$ of Banach algebras $A_i$, one can put inequivalent norms (which are equivalent when restricted to an individual $A_i$, but inequivalent in the direct limit). The space $\mathcal{M}_\infty$ has a nice moduli description, but it is not clear to which norm on $\varinjlim A_i$ it corresponds. On the other hand, $\hat{\mathcal{M}}_{\infty,K}$ is characterized as giving $\varinjlim A_i$ the norm making $\varinjlim A_i^\circ$ the norm-$\leq 1$-subalgebra, where $A_i^\circ\subset A_i$ is the set of power-bounded elements. (This corresponds to the weakest possible topology on $\varinjlim A_i$.) In \cite{FaltingsPeriodDomains}, Faltings essentially constructs $\mathcal{M}_\infty$ (in the world of Berkovich spaces), and notes that understanding $\hat{\mathcal{M}}_{\infty,K}$ would be harder. In our context, it is the theory of perfectoid spaces that makes it possible to control this change of topology, and show that one will get a perfectoid (and thus well-behaved) space, without having to construct explicit integral models of all Rapoport-Zink spaces at finite level.

We should also remark that we do not know whether the structure presheaf of $\mathcal{M}_\infty$ is a structure sheaf, and thus this may not be an adic space in Huber's sense. This is one reason that we generalize his theory slightly so as to include `spaces' where the structure presheaf is not a sheaf; this is possible by a Yoneda-type definition.

From here it is not difficult to introduce the notion of an EL Rapoport-Zink space at infinite level.  We define a (rational) EL-datum to be a quadruple $\mathcal{D}=(B,V,H,\mu)$, where $B/\Q_p$ is a semisimple $\Q_p$-algebra, $V$ is a finite $B$-module, $H$ is a $p$-divisible group over $k$ up to isogeny, and $\mu$ is a cocharacter of $\mathbf{G}=\GL_B(V)$, subject to the constraints described in \S\ref{ELstructures}.  We obtain a functor $\mathcal{M}_{\mathcal{D},\infty}$ which satisfies an analogue of Theorem D. In particular $\mathcal{M}_{\mathcal{D},\infty}$ is an adic space. It carries an action of $\check{\mathbf{G}}(\Q_p)$, where $\check{\mathbf{G}}=\Aut_B(H)$ (automorphisms up to isogeny).

Our final main theorem is a duality isomorphism between EL Rapoport-Zink spaces at infinite level which generalizes the isomorphism between the Lubin-Tate and Drinfeld towers described in \cite{FarguesGenestierLafforgue}. There is a simple duality $\mathcal{D}\mapsto \check{\mathcal{D}}$ among basic Rapoport-Zink data, which reverses the roles of $\mathbf{G}$ and $\check{\mathbf{G}}$. Also, we note that in general, $\mathcal{M}_{\mathcal{D},\infty}$ comes with a Hodge-Tate period map
\[
\pi_\HT\from \mathcal{M}_{\mathcal{D},\infty}\to \Flag_\HT
\]
to a certain flag variety $\Flag_\HT$. Analogues of $\pi_\HT$ exist whenever one has a family of $p$-divisible groups with trivialized Tate module, e.g. for infinite-level Shimura varieties.

\begin{thme} There is a natural $\mathbf{G}(\Q_p)\times \check{\mathbf{G}}(\Q_p)$-equivariant isomorphism $\mathcal{M}_{\mathcal{D},\infty}\isom \mathcal{M}_{\check{\mathcal{D}},\infty}$ which exchanges the Grothendieck-Messing and Hodge-Tate period maps.
\end{thme}

In light of the linear algebra descriptions of $\mathcal{M}_{\mathcal{D},\infty}$ and $\mathcal{M}_{\check{\mathcal{D}},\infty}$ of Theorem D, Theorem E is reduced to a direct linear algebra verification. Roughly, in the picture of modifications of vector bundles
\[
0\to \mathscr{F}\to \mathscr{E}\to i_{\infty\ast} W\to 0\ ,
\]
one applies $\mathscr{H}om(-,\mathscr{E})$ to get an $\End(\mathscr{E})$-equivariant modification
\[
0\to \mathscr{H}om(\mathscr{E},\mathscr{E})\to \mathscr{H}om(\mathscr{F},\mathscr{E})\to i_{\infty\ast} \check{W}\to 0\ .
\]
In particular, the duality makes \'etale homology (i.e., $\mathscr{F}$) into crystalline homology $\check{\mathscr{E}} = \mathscr{H}om(\mathscr{F},\mathscr{E})$, and conversely. As an application, we show (Theorem \ref{splittingtheorem}) that any $\GL_n(\Q_p)$-equivariant \'etale cover of Drinfeld's upper half-space $\Omega\subset\mathbf{P}^{n-1}$ must factor through the Drinfeld space $\mathcal{M}_m^{\text{Dr}}$ for some $m\geq 0$.

Again, let us make some historical remarks on Theorem E. A conjecture in this direction appeared first in the book \cite{RZ} of Rapoport and Zink. However, they are being somewhat vague about the nature of this isomorphism, as they were using the language of rigid-analytic geometry, whereas the infinite-level spaces $\mathcal{M}_{\mathcal{D},\infty}$ are not of finite type any more, and thus not rigid-analytic varieties. Until now, a good framework to work with those infinite-level spaces was missing, and is now provided by the theory of perfectoid spaces. The case of Lubin-Tate and Drinfeld tower was first established by Faltings, \cite{FaltingsLubinTateDrinfeld}, and then worked out in detail by Fargues (in mixed characteristic) and Genestier and Lafforgue (in equal characteristic), \cite{FarguesGenestierLafforgue}. The case of Lubin-Tate and Drinfeld tower is easier, as in this case Theorem C (i.e., the image of the period morphism) was known for a long time, by work of Drinfeld, \cite{DrinfeldUpperHalfSpace}, and Gross-Hopkins, \cite{GrossHopkins}. In fact, in those cases, one has $\Flag^\wa = \Flag^\a$, which is not true in general.

We note that Fargues, Genestier and Lafforgue worked with explicit integral models (as formal schemes of infinite type make good sense) to formalize this isomorphism. This necessitates a detailed understanding of the integral structure of all objects involved, and in particular of integral $p$-adic Hodge theory (in fact, often $p\neq 2$ is assumed for this reason). By contrast, our arguments are entirely on the generic fibre, and use only rational $p$-adic Hodge theory.

Theorem E was proved earlier by Faltings in \cite{FaltingsPeriodDomains}, after proving Theorem C.

We note that the description of $\mathcal{M}_\infty$ in terms of $p$-adic Hodge theory (in particular, independently of $p$-divisible groups) makes it possible to generalize our methods to give a definition of an analogue of Rapoport-Zink spaces for non-PEL cases; this will be the subject of future work of the first author.

Finally, let us summarize the different sections. In Section 2, we review the theory of adic spaces in the form that we will need it. In particular, we give a Yoneda-style definition of adic spaces which is more general than Huber's, and contains Huber's adic spaces as a full subcategory (that we will refer to as honest adic spaces). Moreover, we recall the notion of perfectoid spaces, and review some facts about inverse limit constructions in the category of adic spaces.

In Section 3, we recall some facts about $p$-divisible groups. In general, we like to view them as formal schemes (instead of inductive limits), thus making their generic fibres (disjoint unions of) open unit balls (instead of a set of points). Most importantly, we define the universal cover $\tilde{G}$ of a $p$-divisible group $G$, and show that it is a crystal on the infinitesimal site. Moreover, we show that there is a natural morphism $\tilde{G}\to EG$ of crystals on the nilpotent PD site, where $EG$ is the universal vector extension. This makes it possible to make explicit some maps between Dieudonn\'e modules.

In Section 4, we prove Theorem A, in Section 5, we prove Theorem B, and in Section 6, we prove Theorem C and D, using the respective ideas explained above. Finally, in Section 7, we prove Theorem E and the result on equivariant covers of Drinfeld's upper half-space.

{\bf Acknowledgments.} Our work in this direction started after we gave (incidentally!) successive talks on perfectoid spaces, and the Lubin-Tate tower at infinite level, respectively, at a conference in Princeton in March 2011, where some explicit formulas made it strikingly clear that this is a perfectoid space -- we are deeply thankful to the organizers of that conference for making this happen. The authors thank Laurent Fargues for many discussions, and in particular for sending us his preprints \cite{FarguesPDivGroups} and \cite{FarguesFontaineDurham}, the first of which led to a considerable simplification in our proof of Theorem B. Moreover, they thank Bhargav Bhatt, Przemyslaw Chojecki, Pierre Colmez, Brian Conrad, Ofer Gabber, Urs Hartl, Hadi Hedayatzadeh, Eugen Hellmann, Johan de Jong, Kiran Kedlaya, Eike Lau, William Messing, Timo Richarz, Michael Rapoport and Richard Taylor for helpful discussions and feedback. The work on this project was carried out during various stays of the authors at the Fields Institute, Bonn, Oxford, Harvard, Boston University, and Stanford, and the authors want to thank these institutions for their hospitality. In particular, they thank the Clay Institute for its support for travels to Bonn, Oxford, and Harvard, where crucial parts of this work were done. This work was done while Peter Scholze was a Clay Research Fellow.

\section{Adic spaces}

\subsection{Adic spaces}  We will be relying heavily on Huber's theory of adic spaces~\cite{Hub94}, so we include a summary of that theory here. Also, we will slightly generalize his notion of adic spaces.

\begin{defn}  An {\em adic ring} is a topological ring $A$ for which there exists an ideal $I$ such that $I^n$ gives a basis of neighborhoods of $0\in A$.  Such an ideal $I$ is an {\em ideal of definition} for $A$.

An {\em f-adic ring} is a topological ring $A$ admitting an open adic subring $A_0\subset A$ with finitely generated ideal of definition. Then any open adic subring of $A$ is called a ring of definition, and all rings of definition of $A$ admit a finitely generated ideal of definition.

A subset $S\subset A$ of a topological ring is {\em bounded} if for every neighborhood $U$ of $0$ in $A$, there exists another neighborhood $V$ of $0$ with $VS\subset U$.  An element $f\in A$ is {\em power bounded} if $\set{f^n}_{n\geq 1}$ is bounded. Write $A^\circ$ for the subring of power-bounded elements in $A$. Also, we denote by $A^{\circ\circ}\subset A^\circ$ the ideal of topologically nilpotent elements.

Finally, an {\em affinoid ring} is a pair $(A,A^+)$, with $A$ an $f$-adic ring and $A^+\subset A$ an open subring which is integrally closed and contained in $A^\circ$. Morphisms between affinoid rings $(A,A^+)$ and $(B,B^+)$ are continuous ring homomorphisms $A\to B$ sending $A^+$ into $B^+$.
\end{defn}

For instance, if a ring $A$ is given the discrete topology then $A$ is both adic ($0$ serves as an ideal of definition) and f-adic, and $(A,A)$ is an affinoid ring.

For a less trivial example, let $K$ be a complete nonarchimedean field with ring of integers $\OO_K$ and maximal ideal $\mathfrak{m}_K$. Then $\OO_K$ is an adic ring and $(K,\OO_K)$ is an affinoid ring.  The integral Tate algebra $A^+=\OO_K\tatealgebra{X_i}_{i\in I}$ in any number of variables is an adic ring. If $\varpi\in \gothm_K$, then $\varpi A^+$ serves as an ideal of definition for $A^+$, so that $A^+$ is $f$-adic. The Tate algebra $A=K\tatealgebra{X_i}_{i\in I}=A^+[\varpi^{-1}]$ is an $f$-adic ring because it contains $A^+$ as an open subring.  In this case we have $A^\circ=A^+$, so that $(A,A^+)$ is an affinoid ring. It should be noted that $(A^+,A^+)$ is an affinoid ring as well.

We note that it is not always the case that $A^+\subset A$ is bounded. Examples arise by taking nonreduced rings appearing in rigid-analytic geometry, such as $A=K[X]/X^2$, in which case one can choose $A^+=A^\circ = \OO_K + KX$.

\begin{defn} Let $A$ be a topological ring.  A {\em continuous valuation} of $A$ is a multiplicative map $\abs{\cdot}$ from $A$ to $\Gamma\cup \set{0}$, where $\Gamma$ is a linearly ordered abelian group (written multiplicatively) such that $\abs{0}=0$, $\abs{1}=1$, $\abs{x+y}\leq \max(\abs{x},\abs{y})$, and for all $\gamma\in \Gamma$, the set $\set{x\in A:\;\abs{x}<\gamma}$ is open.
\end{defn}

Two valuations $\abs{\cdot}_i\from A\to\Gamma_i\cup\set{0}$ ($i=1,2$) are {\em equivalent} if there are subgroups $\Gamma_i^\prime\subset \Gamma_i$ containing the image of $\abs{\cdot}_i$ for $i=1,2$ and an isomorphism of ordered groups $\iota\from\Gamma_1^\prime\to\Gamma_2^\prime$ for which $\iota\left(\abs{f}_1\right)=\abs{f}_2$, all $f\in A$.

\begin{defn}
Let $(A,A^+)$ be an affinoid ring.  The {\em adic spectrum} $X=\Spa_\top(A,A^+)$ is the set of equivalence classes of continuous valuations on $A$ which satisfy $\abs{A^+}\leq 1$. For $x\in X$ and $f\in A$, we write $\abs{f(x)}$ for the image of $f$ under the continuous valuation $\abs{\;}$ corresponding to $x$.

Let $f_1,\dots,f_n\in A$ be elements such that $(f_1,\dots,f_n)A$ is open in $A$, and let $g\in A$.  Define the subset $X_{f_1,\dots,f_n;g}$ by
\[ X_{f_1,\dots,f_n;g}=\set{x\in X\biggm\vert \abs{f_i(x)}\leq \abs{g(x)}\neq 0,\;i=1,\dots,n} \]
Finite intersections of such subsets are called {\em rational}. We give $X$ the topology generated by its rational subsets.
\end{defn}

As an example, if $K$ is a nonarchimedean field, then $\Spa_\top(K,\OO_K)$ consists of a single equivalence class of continuous valuations, namely the valuation $\eta$ which defines the topology on $K$. The space $\Spa_\top(\OO_K,\OO_K)$ is a trait containing $\eta$ and a special point $s$ defined by
\[ \abs{f(s)}=\begin{cases}
1, & f\not\in\mathfrak{m}_K \\
0, & f\in\mathfrak{m}_K.
\end{cases}
\]

For any affinoid ring $(A,A^+)$, Huber defines a presheaf of complete topological rings $\OO_X$ on $X=\Spa_\top(A,A^+)$, together with a subpresheaf $\OO_X^+$. They have the following universal property.

\begin{defn} Let $(A,A^+)$ be an affinoid ring, $X=\Spa_\top(A,A^+)$, and $U\subset X$ be a rational subset. Then there is a complete affinoid ring $(\OO_X(U),\OO_X^+(U))$ with a map $(A,A^+)\to (\OO_X(U),\OO_X^+(U))$ such that
\[
\Spa_\top(\OO_X(U),\OO_X^+(U))\to \Spa_\top(A,A^+)
\]
factors through $U$, and such that for any complete affinoid ring $(B,B^+)$ with a map $(A,A^+)\to (B,B^+)$ for which $\Spa_\top(B,B^+)\to \Spa_\top(A,A^+)$ factors through $U$, the map $(A,A^+)\to (B,B^+)$ extends to a map
\[
(\OO_X(U),\OO_X^+(U))\to (B,B^+)
\]
in a unique way.
\end{defn}

One checks that $\Spa_\top(\OO_X(U),\OO_X^+(U))\cong U$, and that the valuations on $(A,A^+)$ associated to $x\in U$ extend to $(\OO_X(U),\OO_X^+(U))$. In particular, they extend further to $\OO_{X,x}$. Also,
\[
\OO_X^+(U) = \set{f\in \OO_X(U)\mid \forall x\in U : \abs{f(x)}\leq 1 }\ .
\]

The presheaf $\OO_X$ is in general not a sheaf, but it is a sheaf in the important cases where $A$ has a noetherian ring of definition, is strongly noetherian, or is perfectoid. Let us say that $(A,A^+)$ is sheafy if $\OO_X$ is a sheaf. Recall the category $(V)$ whose objects are triples $(X,\OO_X,\set{v_x:\;x\in X})$ consisting of a topological space $X$, a sheaf of topological rings $\OO_X$, and a choice of continuous valuations $v_x$ on $\OO_{X,x}$ for any $x\in X$. In general, we can associate to $(A,A^+)$ the triple $\Spa_\naive(A,A^+)=(X,\OO_X^\sharp,\set{v_x:\;x\in X})$ consisting of the topological space $X=\Spa_\top(A,A^+)$, the sheafification $\OO_X^\sharp$ of $\OO_X$, and a continuous valuation $v_x$ on $\OO_{X,x}^\sharp = \OO_{X,x}$; this induces a functor $\Aff^\op\rightarrow (V)$, where $\Aff$ is the category of affinoid rings. It factors over the category $\CAff^\op$, opposite of the category of complete affinoid rings. Recall that an {\em adic space (in the sense of Huber)} is an object of $(V)$ that is locally isomorphic to $\Spa_\naive(A,A^+)$ for some sheafy affinoid ring $(A,A^+)$. We will use the term {\em honest adic space} to refer to such spaces, and denote their category as $\Adic^h$.

We will now define a more general notion of adic spaces. The definition is entirely analogous to the definition of a scheme as a sheaf for the Zariski topology on the category of rings which admits an open cover by representable sheaves. The only difference is that in our situation, the site is not subcanonical, i.e. representable functors are not in general sheaves.

\begin{defn} Consider the category $\CAff^\op$, opposite of the category of complete affinoid rings $(A,A^+)$. We give it the structure of a site by declaring a cover of $(A,A^+)$ to be a family of morphisms $(A,A^+)\rightarrow (A_i,A_i^+)$, such that $(A_i,A_i^+) = (\OO_X(U_i),\OO_X^+(U_i))$ for a covering of $X=\Spa_\top(A,A^+)$ by rational subsets $U_i\subset X$.

Let $(\CAff^\op)^\sim$ be the associated topos, consisting of covariant functors $\CAff\to \Sets$ satisfying the sheaf property for the given coverings. Any affinoid ring $(A,A^+)$ (not necessarily complete) gives rise to the presheaf
\[
(B,B^+)\mapsto \Hom((A,A^+),(B,B^+))\ ,
\]
and we denote by $\Spa(A,A^+)$ the associated sheaf. Note that passing to the completion gives $\Spa(A,A^+) = \Spa(\hat{A},\hat{A}^+)$. We call $\Spa(A,A^+)$ an affinoid adic space. We say that $\mathcal{F}\to \Spa(A,A^+)$ is a rational subset if there is some rational subset $U\subset X=\Spa_\top(A,A^+)$ such that $\mathcal{F} = \Spa(\OO_X(U),\OO_X^+(U))$. We say that $\mathcal{F}\to \Spa(A,A^+)$ is an open immersion if there is an open subset $U\subset X$ such that
\[
\mathcal{F} = \varinjlim_{V\subset U, V\ \mathrm{rational}} \Spa(\OO_X(V),\OO_X^+(V))\ .
\]
If $f: \mathcal{F}\to \mathcal{G}$ is any morphism in $(\CAff^\op)^\sim$, then we say that $f$ is an open immersion if for all $(A,A^+)$ and all morphisms $\Spa(A,A^+)\to \mathcal{G}$, the fibre product
\[
\mathcal{F}\times_{\mathcal{G}} \Spa(A,A^+)\to \Spa(A,A^+)
\]
is an open immersion. Note that open immersions are injective, and we will often simply say that $\mathcal{F}\subset \mathcal{G}$ is open. Finally, an adic space is by definition a functor $\mathcal{F}\in (\CAff^\op)^\sim$ such that
\[
\mathcal{F} = \varinjlim_{\Spa(A,A^+)\subset \mathcal{F}\ \mathrm{open}} \Spa(A,A^+)\ .
\]
The category of adic spaces is denoted $\Adic$.
\end{defn}

We list some first properties.

\begin{prop}\begin{altenumerate}
\item[{\rm (i)}] Let $(A,A^+)$, $(B,B^+)$ be affinoid rings, and assume that $(A,A^+)$ is sheafy and complete. Then
\[
\Hom(\Spa(A,A^+),\Spa(B,B^+)) = \Hom((B,B^+),(A,A^+))\ .
\]
In particular, the functor $(A,A^+)\mapsto \Spa(A,A^+)$ is fully faithful on the full subcategory of sheafy complete affinoid rings.
\item[{\rm (ii)}] The functor $(A,A^+)\mapsto \Spa_\naive(A,A^+)$ factors over $(A,A^+)\mapsto \Spa(A,A^+)$.
\item[{\rm (iii)}] The functor $\Spa(A,A^+)\mapsto \Spa_\naive(A,A^+)$ coming from part (ii) extends to a functor $X\mapsto X_\naive\from \Adic\to (V)$. In particular, any adic space $X$ has an associated topological space $|X|$, given by the first component of $X_\naive$.
\item[{\rm (iv)}] The full subcategory of adic spaces that are covered by $\Spa(A,A^+)$ with $(A,A^+)$ sheafy is equivalent to $\Adic^h$ under $X\mapsto X_\naive$. Under this identification, we consider $\Adic^h$ as a full subcategory of $\Adic$.
\item[{\rm (v)}] Let $X$ be an honest adic space and let $(B,B^+)$ be an affinoid ring. Then
\[
\Hom(X,\Spa(B,B^+)) = \Hom((B,B^+),(\OO_X(X),\OO_X^+(X)))\ .
\]
Here, the latter is the set of continuous ring homomorphisms $B\to \OO_X(X)$ which map $B^+$ into $\OO_X^+(X)$.
\item[{\rm (vi)}] Fibre products exist in $\Adic$. Locally, they are given by
\[
\Spa(A,A^+)\times_{\Spa(B,B^+)} \Spa(C,C^+) = \Spa(D,D^+)\ ,
\]
where $D=A\otimes_B C$ with the topology making the image of $A_0\otimes_{B_0} C_0$ a ring of definition with ideal of definition the image of $I\otimes J$, where $A_0\subset A$, $B_0\subset B$ and $C_0\subset C$ are rings of definition and $I\subset A_0$ and $J\subset C_0$ are ideals of definition. Moreover, $D^+\subset D$ is the integral closure of the image of $A^+\otimes_{B^+} C^+$ in $D$.
\end{altenumerate}
\end{prop}

\begin{proof} Easy and left to the reader.
\end{proof}

A basic example of an affinoid adic space is the closed adic disc $\Spa(A,A^+)$, where $A=K\tatealgebra{X}$ and $A^+=\OO_K\tatealgebra{X}$.  See~\cite{Sch}, Example 2.20 for a discussion of the five species of points of $\Spa_\top(A,A^+)$. An example of a nonaffinoid adic space is the open adic disc, equal to the open subset of $\Spa(\OO_K\powerseries{X},\OO_K\powerseries{X})$ defined by the condition $\abs{\varpi}\neq 0$, where $\varpi\in \gothm_K$.  A covering of the open disc by affinoid adic spaces over $K$ is given by the collection $\Spa(K\tatealgebra{X,X^n/\varpi},\OO_K\tatealgebra{X,X^n/\varpi})$ for $n\geq 1$.

In the following, we will often write $\Spa(A,A^+)$ for either of $\Spa(A,A^+)$ or $\Spa_\top(A,A^+)$. As by definition, their open subsets agree, we hope that this will not result in any confusion. However, we will never talk about $\Spa_\naive(A,A^+)$ in the following.

\subsection{Formal schemes and their generic fibres}

In this subsection, we recall the definition of formal schemes that we will use, and relate them to the category of adic spaces. Fix a complete nonarchimedean field $K$ with ring of integers $\OO=\OO_K$ and fix some $\varpi\in \OO$, $|\varpi|<1$. Consider the category $\Nilp_{\OO}$ of $\OO$-algebras $R$ on which $\varpi$ is nilpotent. Its opposite $\Nilp_{\OO}^\op$ has the structure of a site, using Zariski covers. We get the associated topos $(\Nilp_{\OO}^\op)^\sim$. Any adic $\OO$-algebra $A$ with ideal of definition $I$ containing $\varpi$ gives rise to the sheaf
\[
(\Spf A)(R) = \varinjlim_n \Hom(A/I^n,R)
\]
on $\Nilp_{\OO}^\op$. As before, one defines open embeddings in $(\Nilp_{\OO}^\op)^\sim$, and hence formal schemes over $\OO$ as those sheaves on $\Nilp_{\OO}^\op$ which admit an open cover by $\Spf A$ for $A$ as above, cf. \cite{RZ}, first page of Chapter 2.

\begin{prop} The functor $\Spf A\mapsto \Spa(A,A)$ extends to a fully faithful functor $\mathfrak{M}\mapsto \mathfrak{M}^\ad$ from formal schemes over $\OO$ which locally admit a finitely generated ideal of definition to adic spaces over $\Spa(\OO,\OO)$.
\end{prop}

\begin{proof} Let us explain how one can reconstruct $\Spf A$ from $\Spa(A,A)$. First, note that $\Spf A\subset \Spa(A,A)$ is the subset of valuations with value group $\{0,1\}$, recovering the topological space. To recover the structure sheaf, it is enough to show that the global sections of the sheafification $\OO_{\Spa(A,A)}^\sharp$ of $\OO_{\Spa(A,A)}$ are given by $A$ itself.

Note that there is a continuous map $\Spa(A,A)\to \Spf A:x\mapsto \bar{x}$, given by sending any continuous valuation $x$ on $A$ to the valuation
\[
f\mapsto |f(\bar{x})| = \left\{\begin{array}{cl}0& |f(x)|<1\\ 1 & |f(x)|=1\ .\end{array}\right .
\]
Moreover, every $x\in \Spa(A,A)$ specializes to $\bar{x}\in \Spf(A)\subset \Spa(A,A)$. It follows that any covering of $\Spa(A,A)$ is refined by a covering coming as pullback via the specialization from a cover of $\Spf A$, so we only have to check the sheaf property for such coverings. But this reduces to the sheaf property of the structure sheaf of $\Spf A$.
\end{proof}

Note that if $A$ is an adic $\OO$-algebra with ideal of definition $I$ containing $\varpi$, and $\mathcal{F}\in (\Nilp_{\OO}^\op)^\sim$, then
\[
\Hom(\Spf A,\mathcal{F}) = \varprojlim_n \mathcal{F}(A/I^n)\ .
\]
We abbreviate the left-hand side to $\mathcal{F}(A)$, thereby extending the functor $\mathcal{F}$ to the larger category of adic $\OO$-algebras with ideal of definition containing $\varpi$.

The following proposition gives a moduli description of the generic fibre
\[
\mathfrak{M}^\ad_\eta = \mathfrak{M}^\ad\times_{\Spa(\OO,\OO)} \Spa(K,\OO)\ .
\]

\begin{prop}\label{PropGenericFibre} Let $(R,R^+)$ be a complete affinoid $(K,\OO)$-algebra.
\begin{altenumerate}
\item[{\rm (i)}] The ring $R^+$ is the filtered union of its open and bounded $\OO$-subalgebras $R_0\subset R^+$.
\item[{\rm (ii)}] The functor $\mathfrak{M}^\ad_\eta\from \CAff_{(K,\OO)}\to \Sets$ is the sheafification of
\[
(R,R^+)\mapsto \varinjlim_{R_0\subset R^+} \mathfrak{M}(R_0) = \varinjlim_{R_0\subset R^+} \varprojlim_n \mathfrak{M}(R_0/\varpi^n)\ .
\]
\end{altenumerate}
\end{prop}

\begin{proof}\begin{altenumerate}
\item[{\rm (i)}] First, there is some open and bounded subring $R_0\subset R^+$. Next, we check that if $R_1,R_2\subset R^+$ are bounded subrings, then so is the image of $R_1\otimes R_2\rightarrow R^+$. Indeed, $R_i\subset \varpi^{-n_i} R_0$ for certain $n_i\in \Z$, and then the image of $R_1\otimes R_2$ is contained in $\varpi^{-n_1-n_2} R_0$. In particular, there are open and bounded $\OO$-subalgebras, and the set of such is filtered. Finally, if $R_0\subset R^+$ is an open and bounded $\OO$-subalgebra and $x\in R^+$, then $x$ is power-bounded, so that all powers of $x$ are contained in $\varpi^{-n} R_0$ for some $n$, and hence $R_0[x]\subset \varpi^{-n} R_0$ is still open and bounded.
\item[{\rm (ii)}] It suffices to check in the case $\mathfrak{M} = \Spf(A,A)$, where $A$ is some adic $\OO$-algebra with finitely generated ideal $I$ containing $\varpi$. Then it follows from
\[
\Hom((A,A),(R,R^+)) = \Hom(A,R^+) = \varinjlim_{R_0\subset R^+} \Hom(A,R_0) = \varinjlim_{R_0\subset R^+} \varprojlim_n \Hom(A,R_0/\varpi^n)\ .
\]
Indeed, any continuous map $A\to R^+$ is also continuous for the $\varpi$-adic topology on $A$; then $A[\frac 1{\varpi}]\to R$ is a continuous map of $K$-Banach spaces. In particular, the image of $A$ in $R^+$ is bounded, and hence contained in some open and bounded $\OO$-subalgebra $R_0$. As $R_0$ is $\varpi$-adically complete, one gets the identification
\[
\Hom(A,R_0) = \varprojlim_n \Hom(A,R_0/\varpi^n)\ .
\]
\end{altenumerate}
\end{proof}

In particular, we can define a functor $\mathcal{F}\mapsto \mathcal{F}^\ad_\eta\from (\Nilp_{\OO}^\op)^\sim\to (\CAff_{(K,\OO)}^\op)^\sim$ taking a moduli problem over $\OO$-algebras in which $\varpi$ is nilpotent to a moduli problem on complete affinoid $(K,\OO)$-algebras, mapping $\mathcal{F}$ to the sheafification of
\[
(R,R^+)\mapsto \varinjlim_{R_0\subset R^+} \varprojlim_n \mathcal{F}(R_0/p^n)\ ,
\]
such that whenever $\mathcal{F}$ is representable by a formal scheme $\mathfrak{M}$ with locally finitely generated ideal of definition, then $\mathcal{F}^\ad_\eta$ is representable by $\mathfrak{M}^\ad_\eta$.

Similarly, this discussion extends to stacks. As an example of this, consider the stack of $p$-divisible groups over $\Nilp_{\Z_p}^\op$, sending any ring $R$ on which $p$ is nilpotent to the groupoid of $p$-divisible groups over $R$, with isomorphisms as morphisms. Note that if $A$ is an adic ring with ideal of definition containing $p$, then giving a compatible system of $p$-divisible groups over $A/I^n$ for all $n\geq 1$ is equivalent to giving a $p$-divisible group over $A$. Using the previous construction, we get a stack on the category of adic spaces over $\Spa(\Q_p,\Z_p)$, which we still call the stack of $p$-divisible groups. Explicitly, if $X$ is an adic space over $\Spa(\Q_p,\Z_p)$, then giving a $p$-divisible group over $X$ amounts to giving a cover of $X$ by open subsets $U_i=\Spa(R_i,R_i^+)\subset X$, open and bounded $\Z_p$-subalgebras $R_{i,0}\subset R_i^+$, and $p$-divisible groups over $R_{i,0}$, satisfying an obvious compatibility condition.

\subsection{Perfectoid spaces}

Here we present the basic definitions of the theory of perfectoid spaces, cf.~\cite{Sch}.

\begin{defn}\label{perfectoidfield} A {\em perfectoid field} $K$ is a complete nonarchimedean field such that its valuation is non-discrete and the $p$-th power map $\OO_K/p\to\OO_K/p$ is surjective.
\end{defn}

Examples include $\C_p$, as well as the $p$-adic completions of the fields $\Q_p(p^{1/p^\infty})$ and $\Q_p(\zeta_{p^\infty})$.  The $t$-adic completion of $\mathbf{F}_p\laurentseries{t}(t^{1/p^\infty})$ is a perfectoid field of characteristic $p$.

\begin{defn} Let $K$ be a perfectoid field. A {\em perfectoid $K$-algebra} is a complete Banach $K$-algebra $A$ satisfying the following properties:
\begin{altenumerate}
\item[{\rm (i)}] The subalgebra $A^\circ$ of power-bounded elements is bounded in $A$, and
\item[{\rm (ii)}] The Frobenius map $A^\circ/p\to A^\circ/p$ is surjective.
\end{altenumerate}
A {\em perfectoid affinoid $(K,\OO_K)$-algebra} is an affinoid $(K,\OO_K)$-algebra $(A,A^+)$ such that $A$ is a perfectoid $K$-algebra.
\end{defn}

In the paper \cite{Sch}, it is proved that if $(A,A^+)$ is a perfectoid affinoid $(K,\OO_K)$-algebra, then $(A,A^+)$ is sheafy, i.e. the structure presheaf on $\Spa(A,A^+)$ is a sheaf. Moreover, if $U\subset X$ is a rational subset, then the pair $(\OO_X(U),\OO_X^+(U))$ is a perfectoid affinoid $(K,\OO_K)$-algebra again.

\begin{defn} A {\em perfectoid space} over $\Spa(K,\OO_K)$ is an adic space locally isomorphic to $\Spa(A,A^+)$ for a perfectoid affinoid $(K,\OO_K)$-algebra $(A,A^+)$. The category of perfectoid spaces over $\Spa(K,\OO_K)$ is a full subcategory of the category of adic spaces over $\Spa(K,\OO_K)$.
\end{defn}

We note that all perfectoid spaces are honest adic spaces. A basic example of a perfectoid space is $X=\Spa(A,A^+)$, where $A^+=\OO_K\tatealgebra{X^{1/p^\infty}}$ is the $\varpi$-adic completion of $\varinjlim \OO_K[X^{1/p^m}]$ and $A=A^+[1/\varpi]$.   Here $A$ is given the unique Banach $K$-algebra structure for which $A^+$ is the $\OO_K$-subalgebra of elements of norm $\leq 1$.

The spaces that we shall ultimately be interested in will turn out to be not quite perfectoid, but instead satisfy a slightly weaker notion.

\begin{defn} Let $X$ be an adic space over $\Spa(K,\OO_K)$. Then $X$ is preperfectoid if there is a cover of $X$ by open affinoid $U_i = \Spa(A_i,A_i^+)\subset X$ such that $(\hat{A}_i,\hat{A}_i^+)$ is a perfectoid affinoid $(K,\OO_K)$-algebra, where we take the completion with respect to the topology on $A_i$ giving $A_i^+$ the $p$-adic topology.
\end{defn}

We call $(\hat{A}_i,\hat{A}_i^+)$ the strong completion of $(A_i,A_i^+)$.

\begin{rmk} As an example, the nonreduced rigid-analytic point $X=\Spa(K[X]/X^2,\OO_K + KX)$ is preperfectoid. Note that in this example, $(K[X]/X^2, \OO_K + KX)$ is already complete for its natural topology. However, artificially enforcing the $p$-adic topology on $\OO_K + KX$ makes the completion simply $(K,\OO_K)$, which is certainly a perfectoid affinoid $(K,\OO_K)$-algebra.
\end{rmk}

\begin{prop} Let $X$ be a preperfectoid adic space over $\Spa(K,\OO_K)$. Then there exists a perfectoid space $\hat{X}$ over $\Spa(K,\OO_K)$ with the same underlying topological space, and such that for any open subset $U=\Spa(A,A^+)\subset X$ for which $(\hat{A},\hat{A}^+)$ is perfectoid, the corresponding open subset of $\hat{X}$ is given by $\hat{U} = \Spa(\hat{A},\hat{A}^+)$. The space $\hat{X}$ is universal for morphisms from perfectoid spaces over $\Spa(K,\OO_K)$ to $X$.
\end{prop}

\begin{proof} Easy and left to the reader.
\end{proof}

We call $\hat{X}$ the strong completion of $X$. We have the following behaviour with respect to passage to locally closed subspaces.

\begin{prop}\label{LocClosedPreperfectoidK} Let $X$ be a preperfectoid adic space over $\Spa(K,\OO_K)$, and let $Y\subset X$ be a locally closed subspace. Then $Y$ is preperfectoid.
\end{prop}

\begin{rmk} Again, this is all one can expect as $Y$ may be nonreduced.
\end{rmk}

\begin{proof} The statement is clear for open embeddings, so let us assume that $Y\subset X$ is a closed subspace. We may assume that $X=\Spa(A,A^+)$ is affinoid with $(\hat{A},\hat{A}^+)$ perfectoid. Then $Y=\Spa(B,B^+)$ is given by a closed ideal $I$, so that $B=A/I$ and $B^+$ the integral closure of the image of $A^+$ in $B$. We claim that $\hat{B}$ is a perfectoid $K$-algebra, which implies the result. Note that $Y$ is the filtered intersection of its rational neighborhoods $U_i\subset X$; indeed, for any $f_1,\ldots,f_k\in I$, one can consider the subsets where $|f_1|,\ldots,|f_k|\leq p^{-n}$, and then $Y$ is their intersection. Then
\[
\hat{B}^+ = \widehat{\varinjlim_{Y\subset U_i} \OO_X^+(U_i)}\ .
\]
Indeed, any $f\in I$ becomes infinitely $p$-divisible in the direct limit, and thus gets killed in the completion. But each $\OO_X(U_i)$ is perfectoid, hence so is $\hat{B}$.
\end{proof}

Finally, we will need one more notion.

\begin{defn} \label{defpreperfectoid} Let $X$ be an adic space over $\Spa(\Q_p,\Z_p)$. Then $X$ is preperfectoid if for any perfectoid field $K$ of characteristic $0$, the base-change $X_K = X\times_{\Spa(\Q_p,\Z_p)} \Spa(K,\OO_K)$ is preperfectoid.
\end{defn}

\begin{prop} Let $K$ be a perfectoid field, and let $X$ be an adic space over $\Spa(K,\OO_K)$. Then $X$ is preperfectoid as an adic space over $\Spa(K,\OO_K)$ if $X$ is preperfectoid as an adic space over $\Spa(\Q_p,\Z_p)$.
\end{prop}

\begin{proof} By assumption, $X_K$ is preperfectoid over $\Spa(K,\OO_K)$. But there is a closed embedding $X\to X_K$, which shows that $X$ is preperfectoid over $\Spa(K,\OO_K)$ by Proposition \ref{LocClosedPreperfectoidK}.
\end{proof}

\begin{prop}\label{LocClosedPreperfectoid} Let $X$ be a preperfectoid adic space over $\Spa(\Q_p,\Z_p)$, and let $Y\subset X$ be a locally closed subspace. Then $Y$ is preperfectoid.
\end{prop}

\begin{proof} This follows directly from Proposition \ref{LocClosedPreperfectoid}.
\end{proof}

\begin{exmp} \label{ExmpBall}  Let $A=\Z_p\powerseries{X_1^{1/p^\infty},\dots,X_n^{1/p^\infty}}$ be the completion of the $\Z_p$-algebra $\varinjlim \Z_p[X_1^{1/p^m},\dots,X_n^{1/p^m}]$ with respect to the $I=(p,X_1,\dots,X_n)$-adic topology.

\begin{prop}\label{ballfunctornilp} Let $R$ be an adic $\Z_p$-algebra with ideal of definition containing $p$. Then
\[
(\Spf A)(R) = (\varprojlim_{x\mapsto x^p} R^{\circ\circ})^n = (\varprojlim_{x\mapsto x^p} R^{\circ\circ}/p)^n = (R^{\flat\circ\circ})^n\ ,
\]
where $R^\flat = \varprojlim_{x\mapsto x^p} R/p$, equipped with the inverse limit topology.
\end{prop}

\begin{proof} Easy and left to the reader.
\end{proof}

Let $X = (\Spf A)^\ad_\eta$ be the generic fibre of $\Spa(A,A)\to\Spa(\Z_p,\Z_p)$. Then $X$ is an adic space over $\Spa(\Q_p,\Z_p)$, equal to the union of affinoids $\Spa(A_m,A_m^+)$ defined by
\[
|X_1|^m,\ldots,|X_n|^m\leq |p|\ ,\ m\geq 1\ .
\]
It is easy to see that $X$ is a preperfectoid space over $\Spa(\Q_p,\Z_p)$.

\begin{lemma}\label{ballfunctor} The adic space $X$ over $\Spa(\Q_p,\Z_p)$ is the sheaf associated to
\[
(R,R^+)\mapsto \varinjlim_{R_0\subset R^+} (R_0^{\flat\circ\circ})^n\ .
\]
\end{lemma}

\begin{proof} This follows directly from Proposition \ref{PropGenericFibre}.
\end{proof}

\end{exmp}

\subsection{Inverse limits in the category of adic spaces}

Inverse limits rarely exist in the category of adic spaces, even when the transition maps are affine. The problem is that if $(A_i,A_i^+)$ is a direct system of affinoid rings, then the direct limit topology of $\varinjlim A_i$ will not be f-adic, and there is no canonical choice for the topology. (Another problem arises because of non-honest adic spaces.) We will use the following definition.

\begin{defn} \label{limitdef} Let $X_i$ be a filtered inverse system of adic spaces with quasicompact and quasiseparated transition maps, let $X$ be an adic space, and let $f_i\from X\to X_i$ be a compatible family of morphisms.  We write $X\sim\varprojlim X_i$ if the map of underlying topological spaces $\abs{X}\to\varprojlim\abs{X_i}$ is a homeomorphism, and if there is an open cover of $X$ by affinoid $\Spa(A,A^+)\subset X$, such that the map
\[
\varinjlim_{\Spa(A_i,A_i^+)\subset X_i} A_i\rightarrow A
\]
has dense image, where the direct limit runs over all open affinoid
\[
\Spa(A_i,A_i^+)\subset X_i
\]
over which $\Spa(A,A^+)\subset X\rightarrow X_i$ factors.
\end{defn}

If the inverse limit consists of a single space $X_0$, then we write $X\sim X_0$ instead of $X\sim\varprojlim X_0$.

\begin{prop}\label{ConstrInvLimits} Let $(A_i,A_i^+)$ be a direct system of complete affinoid rings. Assume that there are rings of definition $A_{i,0}\subset A_i$ compatible for all $i$, and finitely generated ideals of definition $I_i\subset A_{i,0}$ such that $I_j = I_i A_{j,0}\subset A_{j,0}$ for $j\geq i$. Let $(A,A^+)$ be the direct limit of the $(A_i,A_i^+)$, equipped with the topology making $A_0 = \varinjlim_i A_{i,0}$ an open adic subring with ideal of definition $I = \varinjlim_i I_i$. Then
\[
\Spa(A,A^+)\sim \varprojlim \Spa(A_i,A_i^+)\ .
\]
\end{prop}

\begin{proof} It is clear that $(A,A^+)$ with the given topology is an affinoid ring. Moreover, giving a valuation on $(A,A^+)$ is equivalent to giving a compatible system of valuations on all $(A_i,A_i^+)$. The continuity condition is given by asking that $|I^n|\to 0$ as $n\to \infty$, checking that also continuous valuations correspond, thus
\[
|\Spa(A,A^+)|\cong \varprojlim |\Spa(A_i,A_i^+)|\ .
\]
The condition on rings is satisfied by definition, finishing the proof.
\end{proof}

\begin{prop}\label{BaseChangeInvLimit} In the situation of Definition \ref{limitdef}, let $Y_i\to X_i$ be any map of adic spaces, and let $Y_j=Y_i\times_{X_i} X_j\to X_j$ for $j\geq i$ be the pullback, as well as $Y=Y_i\times_{X_i} X\to X$. Then $Y\sim\varprojlim_{j\geq i} Y_j$.
\end{prop}

\begin{proof} See Remark 2.4.3 in \cite{HuberBook}.
\end{proof}

\begin{prop} Let $X$ be a preperfectoid space over $\Spa(K,\OO_K)$, where $K$ is a perfectoid field. Let $\hat{X}$ be the strong completion of $X$. Then $\hat{X}\sim X$.
\end{prop}

\begin{proof} Immediate.
\end{proof}

Of course, this shows that one may have $X\sim X_0$, while $X\neq X_0$. Let us however record the following result.

\begin{prop} Let $K$ be a perfectoid field, let $X_i$ be an inverse system of adic spaces over $\Spa(K,\OO_K)$ with qcqs transition maps, and assume that there is a perfectoid space $X$ over $\Spa(K,\OO_K)$ such that $X\sim \varprojlim X_i$. Then for any perfectoid affinoid $(K,\OO_K)$-algebra $(B,B^+)$, we have
\[
X(B,B^+) = \varprojlim X_i(B,B^+)\ .
\]
In particular, if $Y$ is a perfectoid space over $\Spa(K,\OO_K)$ with a compatible system of maps $Y\to X_i$, then $Y$ factors over $X$ uniquely, making $X$ unique up to unique isomorphism.
\end{prop}

\begin{proof} We may assume that $X=\Spa(A,A^+)$ is affinoid such that $(A,A^+)$ is a perfectoid affinoid $(K,\OO_K)$-algebra. Then we have $X(B,B^+)=\Hom((A,A^+),(B,B^+))$. As
\[
\varinjlim_{\im(X)\subset \Spa(A_i,A_i^+)\subset X_i} A_i\rightarrow A
\]
has dense image, it follows that the map
\[
X(B,B^+) \to \varprojlim X_i(B,B^+)
\]
is injective. Conversely, given a compatible system of maps $\Spa(B,B^+)\to X_i$, it follows that $\Spa(B,B^+)\to X_i$ factors over the image of $X$ in $X_i$. We get a map
\[
\varinjlim_{\im(X)\subset \Spa(A_i,A_i^+)\subset X_i} A_i^+\to B^+\ .
\]
Passing to $p$-adic completions, the left-hand side becomes equal to $A^+$, as $(A,A^+)$ is strongly complete, and the left-hand side is dense in $A^+$. This gives the desired map $A^+\to B^+$, i.e. $(A,A^+)\to (B,B^+)$.
\end{proof}

It seems reasonable to expect that there is a good definition of an \'{e}tale site for any adic space, and that the following result is true.

\begin{conj} Assume that $X_i$ is an inverse system of quasicompact and quasiseparated adic spaces, and that $X\sim \varprojlim X_i$. Then the \'{e}tale topos of $X$ is equivalent to the projective limit of the \'{e}tale topoi of the $X_i$ (considered as a fibred topos in the obvious way).
\end{conj}

We will only need the cases where all $X_i$ and $X$ are either strongly noetherian adic spaces, or perfectoid spaces. In that case, there is a definition of their \'{e}tale topoi, and we have the following result.

\begin{Theorem}[\cite{Sch}, Thm. 7.17] \label{topoi} Let $K$ be a perfectoid field, and assume that all $X_i$ are strongly noetherian adic spaces (in particular quasicompact and quasiseparated) over $\Spa(K,\OO_K)$, and that $X$ is a perfectoid space over $\Spa(K,\OO_K)$, $X\sim \varprojlim X_i$. Then the \'etale topos of $X$ is the projective limit of the \'etale topoi of the $X_i$.
\end{Theorem}

We remark that the notion $\sim$ introduced here is stronger than the notion $\sim^\prime$ introduced in \cite{Sch}, which only talks about the residue fields instead of the sections of the structure sheaf in an open neighborhood. As explained in \cite[Section 2.4]{HuberBook}, one should expect that the weaker notion $\sim^\prime$ is sufficient in the conjecture if all $X_i$ are analytic adic spaces, as is the case in the previous theorem.

\section{Preparations on $p$-divisible groups}

In this section, we work over a ring $R$ which is $p$-torsion, and consider $p$-divisible groups $G$ over $R$. Let $\Nilp_R^\op$ denote the category opposite of the category of $R$-algebras on which $p$ is nilpotent. Of course, the last condition is vacuous, but later we will pass to $p$-adically complete $\Z_p$-algebras $R$.

\subsection{The universal cover of a $p$-divisible group}

We will consider $G$ as the sheaf on $\Nilp_R^\op$, sending any $R$-algebra $S$ to $\varinjlim G[p^n](S)$.

\begin{lemma} Assume that $G$ is isogenous to an extension of an \'{e}tale $p$-divisible group by a connected $p$-divisible group. Then the functor $G$ is representable by a formal scheme (still denoted $G$), which locally admits a finitely generated ideal of definition.
\end{lemma}

\begin{proof} It is clear that isogenies induce relatively representable morphisms, which do not change the ideal of definition, so we may assume that $G$ is an extension of an \'{e}tale by a connected $p$-divisible group. Then one reduces to the connected case, in which case $G$ is a formal Lie group, giving the result.
\end{proof}

One can also consider the completion $\hat{G}$ of $G$, which we define as the functor on $\Nilp_R^\op$ sending any $R$-algebra $S$ on which $p$ is nilpotent to the subset of those $x\in G(S)$ for which there is some nilpotent ideal $I\subset S$ such that $\bar{x}\in G(S/I)$ is the zero section. We recall that this is always representable.

\begin{lemma} The functor $\hat{G}$ is a formal Lie variety, and in particular representable by an affine formal scheme with finitely generated ideal of definition. If $G$ is connected, then $G=\hat{G}$. Moreover, if $\Lie G$ is free, then
\[
\hat{G} = \Spf R\powerseries{X_1,\ldots,X_d}\ .
\]
\end{lemma}

\begin{proof} Cf. \cite{Messing}.
\end{proof}

Consider the sheaf
\[
\tilde{G}(S) = \varprojlim_{p\from G\to G} G(S)\ ,
\]
on $\Nilp_R^\op$, which we call the universal cover of $G$. Clearly, $\tilde{G}$ is a sheaf of $\Q_p$-vector spaces.

\begin{prop}\label{UnivCoverRepr}\begin{altenumerate}
\item[{\rm (i)}] Let $\rho\from G_1\to G_2$ be an isogeny. Then the induced morphism $\tilde{\rho}\from \tilde{G}_1\to \tilde{G}_2$ is an isomorphism.
\item[{\rm (ii)}] Let $S\to R$ be a surjection with nilpotent kernel. Then the categories of $p$-divisible groups over $R$ and $S$ up to isogeny are equivalent. In particular, for any $p$-divisible group $G$ over $R$, the universal cover $\tilde{G}$ lifts canonically to $\tilde{G}_S$ over $S$. Moreover, $\tilde{G}_S(S) = \tilde{G}(R)$.
\item[{\rm (iii)}] Assume that $G$ is isogenous to an extension of an \'{e}tale by a connected $p$-divisible group. Then $\tilde{G}$ is representable by a formal scheme (still denoted $\tilde{G}$), which locally admits a finitely generated ideal of definition. If $R$ is perfect of characteristic $p$, $G$ is connected and $\Lie G$ is free of dimension $d$, then
\[
\tilde{G}\cong \Spf R\powerseries{X_1^{1/p^\infty},\ldots,X_d^{1/p^\infty}}\ .
\]
\end{altenumerate}
\end{prop}

\begin{rmk} Part (ii) says that $\tilde{G}$ can be considered as a crystal on the infinitesimal site of $R$.
\end{rmk}

\begin{proof} Part (i) is clear, as multiplication with $p$-powers induces an isomorphism on $\tilde{G}$. In part (ii), only the last statement requires proof; but note that
\[
\tilde{G}(R) = \Hom_R(\Q_p/\Z_p,G)[p^{-1}]\ ,
\]
so the result follows from the equivalence of categories of $p$-divisible groups up to isogeny, in the special case of homomorphisms with source $\Q_p/\Z_p$.

For part (iii), the first part follows easily from the result for $G$. Now assume that $R$ is perfect, and $G$ is connected. Let $G^{(p^n)}=G\otimes_{R,\Phi^n} R$ for $n\in \Z$, where $\Phi\from R\to R$ is the Frobenius.  Write $F\from G^{(p^n)}\to G^{(p^{n+1})}$ for the Frobenius isogeny, and $V\from G^{(p^{n+1})}\to G^{(p^n)}$ for the Verschiebung, so that $VF$ (resp., $FV$) is multiplication by $p$ in $G^{(p^n)}$ (resp., $G^{(p^{n+1})}$).

There is a natural transformation of functors $\mathcal{V}\from \varprojlim_p G\to \varprojlim_F G^{(p^{-n})}$, given by
\[
\xymatrix{
G \ar[d]_{=} & G \ar[d]_{V} \ar[l]_{p} & G \ar[d]_{V^2} \ar[l]_{p} & \cdots \ar[l] \\
G & G^{(p^{-1})} \ar[l]^{F} & G^{(p^{-2})} \ar[l]^{F} &\cdots \ar[l]
}
\]
We claim that $\mathcal{V}$ is an isomorphism.  Since $G$ is connected, there exists $m\geq 1$ such that $F^m=pu$ for an isogeny $u\from G^{(p^{-m})}\to G$.  Use the same symbol $u$ to denote the base change $G^{(p^{-mn})}\to G^{(p^{-m(n-1)})}$ for any $n\in\Z$, so that $u^n$ is an isogeny $G^{(p^{-mn})}\to G$.   Write $\mathcal{U}$ for the natural transformation $\varprojlim_{F^m} G^{(p^{-mn})}\to \varprojlim_p G$ induced by the $u^n$.  One checks that $\mathcal{U}\circ \mathcal{V}$ and $\mathcal{V}\circ \mathcal{U}$ are both the identity map, giving the claim.

Now it is clear that $\varprojlim_F G^{(p^{-n})}$ is represented by the completion of
\[
\varinjlim_{X_i\mapsto X_i^p} R\powerseries{X_1,\ldots,X_d}\ ,
\]
as claimed.
\end{proof}

A similar argument is presented in the proof of Prop. 8.4 of~\cite{FarguesFontaine}.

\begin{Cor}\label{tG-G0} Assume that there is a perfect ring $A$, a map $A\to R/p$, a connected $p$-divisible group $H$ over $A$ with free Lie algebra of dimension $d$, and a quasi-isogeny
\[
\rho\from H\otimes_A R/p\to G\otimes_R R/p\ .
\]
Then $\tilde{G}$ is represented by the formal scheme
\[
\tilde{G} = \Spf R\powerseries{X_1^{1/p^\infty},\ldots,X_d^{1/p^\infty}}\ .
\]
Moreover, $\rho$ induces a natural identification $\tilde{G} = \tilde{H}_R$, where the latter denotes the evaluation of the crystal $\tilde{H}\otimes_A R/p$ on the infinitesimal site of $R/p$ on the thickening $R\to R/p$.
\end{Cor}

Again, one can consider a variant with better representability properties, namely $\tilde{G}\times_G \hat{G}$.

\begin{lemma}\label{UnivCoverAlmRepr} The sheaf $\tilde{G}\times_G \hat{G}$ is representable by an affine formal scheme $\Spf S$ with finitely generated ideal of definition $I\subset S$ given by the preimage of the ideal defining the unit section in $\hat{G}$. Moreover, $S$, $I$ and $S/I$ are flat over $R$, and if $R$ is of characteristic $p$, then $S$ is relatively perfect over $R$, i.e. the relative Frobenius morphism
\[
\Phi_{S/R}\from S\otimes_{R,\Phi} R\to S
\]
is an isomorphism.
\end{lemma}

\begin{proof} As $G\buildrel p\over\to G$ is relatively representable (and finite locally free), the map $G\times_G \hat{G}\to \hat{G}$ is relatively representable, and thus $G\times_G \hat{G}$ is representable, where the transition map $G\to G$ is multiplication by a $p$-power. Moreover, the ideal of definition does not change: One can take the preimage of the ideal defining the unit section in $\hat{G}$. Upon passing to the inverse limit (i.e., the direct limit on the level of rings), we get representability of $\tilde{G}\times_G \hat{G}$: The flatness assertions are easy to check, as $\hat{G}$ is a formal Lie variety and the transition morphisms are finite and locally free. To check that $S$ is relatively perfect over $R$, one has to check that the pullback of $\hat{G}^{(p)}$ along $F: G\to G^{(p)}$ is $\hat{G}$, which is obvious.
\end{proof}

\begin{rmk} Note that $\tilde{G}\times_G \hat{G}$ is still a crystal on the infinitesimal site, as checking whether a section of $G$ on a thickening $S\to R$ of $R$ lies in $\hat{G}(S)$ can be checked after restricting to $R$, by definition of $\hat{G}$.
\end{rmk}

\subsection{The universal vector extension, and the universal cover}

Let $EG$ be the universal vector extension of $G$:
\[
0 \to V\to EG\to G\to 0\ .
\]
Define a map
\begin{equation}\label{alphaG}
 s_G\from \tilde{G}(R)\to EG(R)
\end{equation}
as follows.  Represent an element $x\in \tilde{G}(R)$ by a sequence $x_n\in G(R)$.  Lift $x_n$ to an element $y_n\in EG(R)$, and form the limit $\lim_{n\to \infty} p^ny_n$.  The limit exists because the differences $p^{m+n}y_{m+n}-p^ny_n$ lie in $p^nV(R)$.  A different choice of lifts changes the value of $y$ by an element of the form $\lim_{n\to\infty} p^nv_n$ with $v_n\in V(R)$, but this limit is $0$.  Set $s_G(x)=y$.  It is easy to see that the formation of $s_G$ is functorial in $G$.

We note that if $G=\Q_p/\Z_p$, then the universal vector extension of $G$ is
\[ 0\to\mathbf{G}_a\to \frac{\mathbf{G}_a\oplus\Q_p}{\Z_p} \to \Q_p/\Z_p\to 0,\]
and the morphism $s_G\from\tilde{G}\to EG$ is exactly the inclusion of $\Q_p$ into $(R\oplus\Q_p)/\Z_p$.

\begin{lemma} The morphism $s_G$ extends functorially to a morphism of crystals $\tilde{G}\to EG$ on the nilpotent crystalline site (i.e., thickenings with a nilpotent PD structure).
\end{lemma}

\begin{proof} Let $S\to R$ be a nilpotent PD thickening. As one can always lift $p$-divisible groups from $R$ to $S$, it is enough to prove the following result.

\begin{lemma} Let $G$, $H$ be two $p$-divisible groups over $S$, with reductions $G_0$, $H_0$ to $R$. Let $f_0\from G_0\to H_0$ be any morphism. Then the diagram
\[
\xymatrix{
\tilde{G}(S) \ar[r]^{s_G} \ar[d]_{\tilde{f}_0}  & EG(S) \ar[d]^{Ef_0} \\
\tilde{H}(S) \ar[r]_{s_H} & EH(S)
}
\]
commutes.
\end{lemma}

\begin{proof} Some multiple $p^n f_0$ of $f_0$ lifts to a morphism $f\from G\to H$ over $S$. It is clear that the diagram
\[
\xymatrix{
\tilde{G}(S) \ar[r]^{s_G} \ar[d]_{\tilde{f}}  & EG(S) \ar[d]^{Ef} \\
\tilde{H}(S) \ar[r]_{s_H} & EH(S)
}
\]
commutes, as everything is defined over $S$. Let $g\in \Hom(\tilde{G}(S),EH(S))$ denote the difference of the two morphisms in the diagram of the lemma. Then we just proved that $p^n g=0$. On the other hand, $\tilde{G}(S)$ is a $\Q_p$-vector space. It follows that $g=0$.
\end{proof}
\end{proof}

Now assume that $S$ is a topological $\Z_p$-algebra, equipped with a surjection $S\to R$, whose kernel is topologically nilpotent and has a PD structure. We do not require the PD structure to be (topologically) nilpotent. For example, we allow the surjection $\Z_2\to \F_2$.

In the following, we write $\MM(G)$ for the crystal $\Lie EG$. Note that by the theory of \cite{BerthelotBreenMessing}, it can be made into a crystal on the crystalline site, not just on the nilpotent crystalline site.

\begin{defn} The quasi-logarithm map is defined as the map
\[
\qlog_G\from \tilde{G}(S)\to \MM(G)(S)[p^{-1}]
\]
given as the composite
\[
\tilde{G}(S)\buildrel{s_G}\over\longrightarrow EG(S)\buildrel{\log_{EG}}\over\longrightarrow \MM(G)(S)[p^{-1}]\ .
\]
\end{defn}

\begin{rmk} In this definition, we want to evaluate $EG$ on $S$. This may not be possible, as the thickening $S\to R$ is not required to have a (topologically) nilpotent PD structure. The easiest way around this is the following. Consider $R^\prime = S/I^2$, where $I=\ker(S\to R)$. Then $S\to R^\prime$ is a topologically nilpotent PD thickening. Up to isogeny, working over $R$ and $R^\prime$ is equivalent. As the quasi-logarithm only involves objects up to isogeny, we can use the definition for $S\to R^\prime$ to get the result for $S\to R$.
\end{rmk}

\begin{lemma}\label{QLogvsLog} Let $G^\prime$ be a lift of $G$ to $S$. The diagram
\[\xymatrix{
\tilde{G}(S)\ar[rr]^{\qlog_G}\ar[d] && \MM(G)(S)[p^{-1}]\ar[d]\\
G^\prime(S)\ar[rr]^{\log_{G^\prime}} && (\Lie G^\prime)[p^{-1}]
}\]
commutes.
\end{lemma}

\begin{rmk} This lemma says that on the universal cover $\tilde{G}(S)$, which is independent of the lift, one has quasi-logarithm map (which is also independent of the lift), and which specializes to the logarithm map of any chosen lift $G^\prime$ upon projection to the corresponding quotient $\MM(G)(S)\to \Lie G^\prime$.
\end{rmk}

\begin{proof} Reduce to the case of a topologically nilpotent PD thickening as in the previous remark. Then the maps from $\tilde{G}(S)=\tilde{G}^\prime(S)$ both factor over $s_{G^\prime}\from \tilde{G}^\prime(S)\to EG^\prime(S)$. The remaining diagram is just the functoriality of $\log$.
\end{proof}

\subsection{The Tate module of a $p$-divisible group}

Consider the sheaf
\[
T(G)(S) = \varprojlim_n G[p^n](S)\ ,
\]
on $\Nilp_R^\op$, which we call the (integral) Tate module of $G$. Clearly, $T(G)$ is a sheaf of $\Z_p$-modules.

\begin{prop} The sheaf $T(G)$ is representable by an affine scheme, flat over $\Spec R$.
\end{prop}

\begin{proof} Note that $T(G)$ is the closed subfunctor of $\tilde{G}\times_G \hat{G}$ given by pullback from the zero section in $\hat{G}$. This means that in the notation of Lemma \ref{UnivCoverAlmRepr} that $T(G) = \Spec (S/I)$, giving the result.
\end{proof}

From now on, we will assume that $R$ is an adic $\Z_p$-algebra, with a finitely generated ideal $I$ of definition containing $p$, and take a $p$-divisible group $G$ over $R$. The previous discussion applies to all rings $R/I^n$, and we will consider the objects obtained by passage to the inverse limit in the following.

We have the adic space $T(G)^\ad$, and its generic fibre $T(G)^\ad_\eta$ over $\Spa(\Q_p,\Z_p)$, which is an adic space over $\Spa(R,R)_\eta$.

\begin{prop}\label{TateInvLimit} We have
\[
T(G)^\ad_\eta\sim \varprojlim_n G[p^n]^\ad_\eta\ .
\]
\end{prop}

\begin{proof} By Proposition \ref{BaseChangeInvLimit}, is enough to prove that $T(G)^\ad\sim \varprojlim_n G[p^n]^\ad$. However, if $G[p^n] = \Spf R_n$, then $T(G) = \Spf R_\infty$, where $R_\infty$ is the $p$-adic completion of $\varinjlim R_n$. Thus the result follows from Proposition \ref{ConstrInvLimits}.
\end{proof}

\begin{prop} Multiplication by $p$ induces an open immersion $T(G)^\ad_\eta\to T(G)^\ad_\eta$. In particular,
\[
V(G)^\ad_\eta = \varinjlim_p T(G)^\ad_\eta
\]
is an adic space over $\Spa(R,R)_\eta$, which we call the rational Tate module of $G$.
\end{prop}

\begin{rmk} Although $V(G)^\ad_\eta$ does not come as the generic fibre of some object $V(G)^\ad$ in general, we still leave a subscript $_\eta$ in the notation to indicate where the object is living.
\end{rmk}

\begin{proof} We note that there is a pullback diagram
\[\xymatrix{
T(G)^\ad\ar[r]^p\ar[d] & T(G)^\ad\ar[d]\\
\set{e}\ar[r] & G[p]^\ad\\
}\]
in the category of adic spaces, where $\set{e}\subset G[p]^\ad$ denotes the identity section. On the generic fibre, $\set{e}\subset G[p]^\ad_\eta$ is an open immersion, giving the claim.
\end{proof}

Although $G$ is not in general representable, one can still define $G^\ad_\eta$ as a functor.

\begin{lemma} The subfunctor
\[
G^\ad_\eta[p^\infty] = \varinjlim G^\ad_\eta[p^n] = \varinjlim G[p^n]^\ad_\eta
\]
is representable by an adic space over $\Spa(R,R)_\eta$.

There is an exact sequence of functors
\[
0\to T(G)^\ad_\eta\to V(G)^\ad_\eta\to G^\ad_\eta[p^\infty]\ .
\]
\end{lemma}

\begin{proof} The identity $G^\ad_\eta[p^n] = G[p^n]^\ad_\eta$ is easy to check. However, $G[p^n]$ is representable, hence so is $G[p^n]^\ad_\eta$. The transition maps are open immersions, so we get the result. The exact sequence of functors is obvious.
\end{proof}

\subsection{The logarithm}

The Lie algebra $\Lie G$ is a locally free $R$-module. We can associate to it an adic space. We give the construction for any locally free $R$-module.

\begin{prop} Let $M$ be a locally free $R$-module. The sheafification of the functor on $\CAff_{\Spa(R,R)_\eta}$ sending $(S,S^+)$ to $M\otimes_R S$ is representable by an adic space over $\Spa(R,R)_\eta$, which we denote $M\otimes \Ga$.
\end{prop}

\begin{proof} One immediately reduces to the case where $M$ is free, and hence where $M=R$. Now, we can assume that $R=\Z_p$, as the general case follows from this by base-change. We have to show that the sheafification of the functor $(S,S^+)\mapsto S$ on complete affinoid $(\Q_p,\Z_p)$-algebras is representable by an adic space $\Ga$ over $\Q_p$. Note that as $S=S^+[p^{-1}]$, it can be written as the direct limit of the functor $(S,S^+)\mapsto S^+$ along the multiplication by $p$ map. The latter functor is representable by $\Spa(\Q_p\tatealgebra{X},\Z_p\tatealgebra{X})$, hence $\Ga = \A^1$ is representable by the increasing union of the balls $\Spa(\Q_p\tatealgebra{p^m X},\Z_p\tatealgebra{p^m X})$, $m\geq 0$.
\end{proof}

The objects introduced so far sit in an exact sequence.

\begin{prop}\label{logG}\begin{altenumerate}
\item[{\rm (i)}] There is a natural $\Z_p$-linear logarithm map
\[
\log_G\from G^\ad_\eta\to \Lie G\otimes \Ga
\]
of sheaves over $\CAff_{\Spa(R,R)_\eta}$.
\item[{\rm (ii)}] The functor $G^\ad_\eta$ is representable by an adic space over $\Spa(R,R)_\eta$.
\item[{\rm (iii)}] The sequence
\[
0\to G^\ad_\eta[p^\infty]\to G^\ad_\eta\to \Lie G\otimes \Ga
\]
of sheaves of $\Z_p$-modules over $\CAff_{\Spa(R,R)_\eta}$ is exact.
\item[{\rm (iv)}] The functor $\tilde{G}^\ad_\eta$ is representable by an adic space over $\Spa(R,R)_\eta$.
\item[{\rm (v)}] The sequence
\[
0\to V(G)^\ad_\eta\to \tilde{G}^\ad_\eta\to \Lie G\otimes \Ga
\]
of sheaves of $\Q_p$-vector spaces over $\CAff_{\Spa(R,R)_\eta}$ is exact.
\end{altenumerate}
\end{prop}

\begin{proof} Functorially in $(S,S^+)$, we have to define a $\Z_p$-linear map
\[
\log_G\from G^\ad_\eta(S,S^+)\to (\Lie G\otimes \Ga)(S)\ .
\]
Recall that $G^\ad_\eta$ is the sheaf associated to
\[
(S,S^+)\mapsto \varinjlim_{S_0\subset S^+} G(S_0)\ ,
\]
and $\Lie G\otimes \Ga$ is the sheaf associated to $(S,S^+)\mapsto \Lie G\otimes_R S$. Changing notation slightly, we see that it is enough to describe a functorial map
\[
\log_G\from G(R)\to \Lie G[p^{-1}]
\]
for any $p$-adically complete and separated flat $\Z_p$-algebra $R$, and $p$-divisible group $G$ over $R$. Then Lemma 2.2.5 of \cite{Messing} gives a $\Z_p$-linear bijection
\[
\log_G\from \ker(G(R)\to G(R/p^2))\buildrel\cong\over\to p^2 \Lie G\ .
\]
As multiplication by $p$ is topologically nilpotent on $G(R)$, any section $x\in G(R)$ admits some integer $n\geq 0$ such that $[p^n]_G(x)\in \ker(G(R)\to G(R/p^2))$. This allows one to extend the morphism to
\[
\log_G\from G(R)\to \Lie G[p^{-1}]\ ,
\]
as desired. It also shows that the kernel is precisely $G(R)[p^\infty]$, which proves part (iii). For part (ii), note that we have proved that there is an open subset $U$ of the identity of $G^\ad_\eta$ on which $\log_G$ is an isomorphism onto its image; it follows that this open subset is representable. But as multiplication by $p$ is relatively representable (and finite locally free), it follows that the preimages of $U$ under multiplication by $p$-powers are representable. Now one finishes the proof by noting multiplication by $p$ is topologically nilpotent.

Similarly, in part (iv), the representability of $\tilde{G}^\ad_\eta$ follows from the representability of $G^\ad_\eta$ and the relative representability of $\tilde{G}\to G$ (as an inverse of finite locally free morphisms $G\buildrel{p^k}\over\longrightarrow G$).

For part (v), the only nontrivial part is to show exactness in the middle. But if $x\in \tilde{G}^\ad_\eta(S,S^+)$ maps to $0$ in $(\Lie G\otimes \Ga)(S,S^+)$, then the sequence in part (iii) shows that after multiplication by a power of $p$, $x$ maps to $0$ in $G^\ad_\eta(S,S^+)$. By definition of $T(G)$, this implies that $x$ lies in $T(G)^\ad_\eta(S,S^+)$, as desired.
\end{proof}

\subsection{Explicit Dieudonn\'{e} theory}

Later, we will need to make certain maps between Dieudonn\'{e} modules explicit. This will be contained in the following discussion. Again, let $R$ be a ring on which $p$ is nilpotent, and let $S$ be a topological ring equipped with a surjection $S\to R$ whose kernel is topologically nilpotent and has a PD structure. Let $G$ be a $p$-divisible group over $R$. Note that
\[
\Hom_R(\Q_p/\Z_p,G) = TG(R)\ .
\]
In particular, Dieudonn\'{e} theory gives a map
\[
TG(R) = \Hom_R(\Q_p/\Z_p,G)\to \Hom(\MM(\Q_p/\Z_p)(S),\MM(G)(S))\to \MM(G)(S)\ ,
\]
by evaluation at the natural element $1\in \MM(\Q_p/\Z_p)(S)$, in turn giving a map
\[
\tilde{G}(R) = TG(R)[p^{-1}]\to \MM(G)(S)[p^{-1}]\ .
\]

\begin{lemma}\label{ExplicitDieudonne} The map $\tilde{G}(R)\to \MM(G)(S)[p^{-1}]$ agrees with $\qlog\from \tilde{G}(R)\to \MM(G)(S)[p^{-1}]$.
\end{lemma}

\begin{proof} First note that we are allowed to work up to isogeny; in particular, we can again assume that $S\to R$ has a topologically nilpotent PD structure. Now note that both morphisms are functorial in the group $G$, and any element $t\in \tilde{G}(R)=\Hom_R(\Q_p/\Z_p,G)[p^{-1}]$ comes via functoriality from the identity morphism of $\Q_p/\Z_p$. It follows that it is enough to prove the statement for the identity morphism of $G=\Q_p/\Z_p$. In that case, the statement is easily verified from the definition of the natural basis element $1\in \MM(\Q_p/\Z_p)(S)$.
\end{proof}

\section{Dieudonn\'{e} theory over semiperfect rings}

\subsection{Statement of result}

Let $R$ be a ring of characteristic $p$.

\begin{defn}  A ring $R$ in characteristic $p$ is {\em semiperfect} if the Frobenius map $\Phi\from R\to R$ is surjective. A map $f\from R\to S$ of semiperfect rings is said to be an isogeny if its kernel $I\subset R$ satisfies $\Phi^n(I)=0$ for some $n$.
\end{defn}

In particular, an isogeny $f\from R\to S$ factors over $\Phi^n\from R\to R$ for some $n$, giving a map $g\from S\to R$, which is also an isogeny. One checks that the notion of being isogenous gives an equivalence relation on the set of semiperfect rings.

\begin{defprop}\label{DefFSemiperfect} Let $R$ be a semiperfect ring. Let $R^\flat = \varprojlim_\Phi R$ denote its (inverse) perfection, equipped with the inverse limit topology. Then the following statements are equivalent.
\begin{altenumerate}
\item[{\rm (i)}] The ring $R^\flat$ is f-adic.
\item[{\rm (ii)}] The ring $R$ is isogeneous to a semiperfect ring $S$ which is the quotient of a perfect ring $T$ by a finitely generated ideal $J\subset T$.
\end{altenumerate}
In this case, we say that $R$ is f-semiperfect. Moreover, in that case $R$ is isogenous to a ring $S$ which is the quotient of a perfect ring $T$ by an ideal $J\subset T$ for which $\Phi(J) = J^p$.
\end{defprop}

\begin{proof} Let $J=\ker(R^\flat\to R)$. Then $\Phi^n(J)$ is a basis of open neighborhoods of $0$ in $R^\flat$. Assume that $R^\flat$ is f-adic, and let $I\subset R^\flat$ be a finitely generated ideal of definition. Then $\Phi^n(J)\subset I$ for some $n$, hence $J\subset \Phi^{-n}(I)$, and we can take $S=R^\flat/\Phi^{-n}(I)$.

Conversely, it is enough to show that if $J$ is finitely generated, then $R^\flat$ is f-adic. It is enough to show that $J$ is an ideal of definition. But $\Phi^n(J)$ and $J^n$ are cofinal if $J$ is finitely generated, giving the result.

For the last statement, note that if $J\subset R^\flat$ is finitely generated, then $J^\prime = \bigcup_n \Phi^{-n}(J^{p^n})\subset R^\flat$ is an ideal of definition of $R^\flat$ with $\Phi(J^\prime) = (J^\prime)^p$, and $S=R^\flat/J^\prime$ has the desired property.
\end{proof}

Also recall the following result of Fontaine.

\begin{prop} Let $R$ be a semiperfect ring. Then there is a universal $p$-adically complete PD thickening $A_\cris(R)$ of $R$. The construction of $A_\cris(R)$ is functorial in $R$; in particular, there is a natural Frobenius $\varphi$ on $A_\cris(R)$.
\end{prop}

\begin{proof} Let us just recall the construction: $A_\cris(R)$ is the $p$-adic completion of the PD hull of the surjection $W(R^\flat)\to R$.
\end{proof}

For semiperfect rings $R$, the Dieudonn\'{e} module functor is given by evaluating the Dieudonn\'{e} crystal on $A_\cris(R)\to R$. One gets a functor from $p$-divisible groups $G$ over $R$ to finite projective $A_\cris(R)$-modules $\MM = \MM(G)(A_\cris(R))$ equipped with maps
\[
F\from \MM\otimes_{A_\cris(R),\varphi} A_\cris(R)\to \MM\ ,\ V\from \MM\to \MM\otimes_{A_\cris(R),\varphi} A_\cris(R)\ ,\ FV = FV = p\ .
\]
Passing to the isogeny category, we get a functor from $p$-divisible groups over $R$ up to isogeny to finite projective $B_\cris^+(R)=A_\cris(R)[p^{-1}]$-modules equipped with $F$ and $V$ as before. We can now state the main theorem of this section.

\begin{Theorem}\label{fullyfaithful} Let $R$ be an f-semiperfect ring. Then the Dieudonn\'{e} module functor on $p$-divisible groups up to isogeny is fully faithful.
\end{Theorem}

The following proposition says that this statement is invariant under isogenies of semiperfect rings.

\begin{prop}\label{IsogenyReduction} Let $f\from R\to S$ be an isogeny of semiperfect rings. Then the reduction functor from $p$-divisible groups over $R$ up to isogeny to $p$-divisible groups over $S$ up to isogeny is an equivalence. The same result holds for Dieudonn\'e modules.
\end{prop}

\begin{proof} One reduces to checking both statements for $\Phi\from R\to R$. Write $R^{(1)}$ for the ring $R$, considered as an $R$-algebra via $\Phi\from R\to R$. We have the base extension functor $G\mapsto G^{(1)} = G\otimes_R R^{(1)}$. On the other hand, $R^{(1)}\cong R$ as rings, so one may consider any $p$-divisible group $H$ over $R^{(1)}$ as a $p$-divisible group over $R$. The composite of the two functors is isogenous to the identity, using $F\from G\to G^{(1)}$, and $V\from G^{(1)}\to G$, respectively. The same argument works for Dieudonn\'e modules.
\end{proof}

\begin{rmk} Before going on, let us observe several pathologies. The first is that torsion may appear in $A_\cris(R)$. In fact, one may adapt the example in \cite{BerthelotOgusInvent} to the present situation. Let $p\neq 2$, and consider
\[
R = \F_p[X^{1/p^\infty},Y^{1/p^\infty}]/(X^2,XY,Y^2)\ .
\]
Then $R^\flat = \F_p\powerseries{X^{1/p^\infty},Y^{1/p^\infty}}$, and we get canonical Teichm\"uller lits $X,Y\in W(R^\flat)$. Let $\gamma_p(x)$ denote the $p$-th divided power of $x$. Then we can consider the element
\[
\tau = \gamma_p(X^2) \gamma_p(Y^2) - \gamma_p(XY)^2\ .
\]
It is easy to see that $p\tau = 0$, and also $\varphi(\tau)=0$. However, one can check that $\tau\neq 0$. For this, look at the thickening
\[\begin{aligned}
S = \F_p[X^{1/p^\infty},Y^{1/p^\infty},U,V] / (X^{2+\epsilon},X^2 Y^{\epsilon},X^{1+\epsilon}Y, &X Y^{1+\epsilon}, X^{\epsilon} Y^2, Y^{2+\epsilon}, \\
&X^\epsilon U, Y^\epsilon U, X^\epsilon V, Y^\epsilon V, U^2, V^2)
\end{aligned}\]
of $R$. Here $\epsilon$ runs through $1/p^n$, for all $n\geq 0$. Note that the kernel $K$ of $S\to R$ has an $\F_p$-basis $(X^2,XY,Y^2,U,V,UV)$. One can define divided powers on $K$ in a unique way extending
\[
\gamma_p(aX^2 + bXY + cY^2 + dU + eV + fUV) = a^p U + c^p V\ .
\]
One gets an induced map $A_\cris(R)\to S$ sending $\tau$ to $UV\neq 0$.

Note that $\tau$ can be interpreted as a morphism of Dieudonn\'{e} crystals from $\Q_p/\Z_p$ to $\mu_{p^\infty}$, which cannot come from a morphism of $p$-divisible groups.

In this first example, $R$ is f-semiperfect. Another pathology occurs for semiperfect rings which are not f-semiperfect. Consider the ring
\[
R = \F_p[X_1^{1/p^\infty},X_2^{1/p^\infty},\ldots] / (X_1,X_2,\ldots,\forall n : X_1^{1/p^n} = (X_2X_3)^{1/p^n} = (X_4X_5X_6)^{1/p^n} = \ldots)\ .
\]
In this example,
\[
R^\flat = \F_p\powerseries{X_1^{1/p^\infty},X_2^{1/p^\infty},\ldots} / (\forall n : X_1^{1/p^n} = (X_2X_3)^{1/p^n} = (X_4X_5X_6)^{1/p^n} = \ldots)\ .
\]
This ring is complete for the topology making the ideals $\Phi^k(J) = (X_1^{p^k},X_2^{p^k},\ldots)$, $k\geq 0$, a basis of open neighborhoods of $0$. Then the element $[X_1]^p$ lies in the kernel of $W(R^\flat)\to A_\cris(R)$. Indeed, it is enough to show that it is divisible by $p^k$ for all $k$. But we can write $X_1$ as a product of $k$ elements of $J$, $X_1 = Y_1\cdots Y_k$. (For example, if $k=3$, take $Y_1 = X_4$, $Y_2 = X_5$, $Y_3 = X_6$.) Then
\[
[X_1]^p = [Y_1]^p \cdots [Y_k]^p = (p!)^k \gamma_p(Y_1)\cdots \gamma_p(Y_k)\ ,
\]
as claimed.
\end{rmk}

However, for f-semiperfect rings, this second pathology cannot occur.

\begin{lemma}\label{WtoAcrisInjective} Let $R$ be an f-semiperfect ring. Then the canonical map $W(R^\flat)\to A_\cris(R)$ is injective.
\end{lemma}

\begin{proof} For the proof, we will construct certain explicit PD thickenings. Let $J=\ker(R^\flat\to R)$; we may assume that $\Phi(J) = J^p$. This implies that the subset
\[
[J] := \{\sum_{i\geq 0} [r_i] p^i\in W(R^\flat)\mid r_i\in J\}\subset W(R^\flat)
\]
is an ideal. Moreover, $W(R^\flat)$ is complete for the $[J]$-adic topology. Let $W_\PD\subset W(R^\flat)[p^{-1}]$ denote the subring generated by all divided powers of elements of $[J]$. It is the quotient of the PD hull of $W(R^\flat)\to R$ by its $p$-torsion. Now consider
\[
W_{\PD,n} = W_\PD / (W_\PD\cap \varphi^n([J][p^{-1}]))\ .
\]
Then elements of $W_{\PD,n}$ can be written uniquely as a sum $\sum_{i\in \Z} [r_i] p^i$, where $r_i\in R^\flat / \Phi^n(J)$, and zero for $i$ sufficiently negative. It is easy to see that the PD structure on $W_\PD(R^\flat)$ passes uniquely to the quotient $W_{\PD,n}$. Moreover, $W_{\PD,n}$ is $p$-adically complete. One gets an induced map $A_\cris(R)\to W_{\PD,n}$. The kernel of the composite $W\to A_\cris(R)\to W_{\PD,n}$ is exactly $\Phi^n([J])$. As $W(R^\flat)$ is complete for the $[J]$-adic topology, this gives the result.
\end{proof}

We will also need the following result.

\begin{lemma}\label{ExPhi1} Let $R$ be any semiperfect ring, and let $I_\cris(R)\subset A_\cris(R)$ be the kernel of $A_\cris(R)\to R$. There is a unique $\varphi$-linear map
\[
\varphi^1\from I_\cris(R)\to A_\cris(R)
\]
such that $p\varphi^1 = \varphi$, $\varphi^1(p)=1$ and
\[
\varphi^1(\gamma_n([x])) = \frac{(np)!}{p\cdot n!} \gamma_{np}([x])
\]
for all $x\in J$ and $n\geq 1$.
\end{lemma}

\begin{rmk} As the proof will show, one could also simply say there is a unique way to define `$\varphi/p$' on $I_\cris(R)$ in a way functorial in $R$. Also, the lemma implies that the surjection $A_\cris(R)\to R$ naturally has the structure of a frame in the sense of \cite{LauFrames}. It is probably true that any $p$-divisible group over $R$ gives rise to a window over this frame, but this does not seem to be a direct consequence of the general theory.
\end{rmk}

\begin{proof} As $I_\cris(R)$ is generated by $p$ and $\gamma_n([x])$ for $x\in J$, $n\geq 1$, it is clear that the map is uniquely determined. We have to see that such a map exists. We argue by reduction to the universal case. To avoid technicalities, we argue modulo $p^k$ for any $k\geq 1$ and show that there is a unique map
\[
I_{\cris,n}(R)\to A_{\cris,n-1}(R)
\]
with the properties described, where $A_{\cris,n}(R) = A_\cris(R)/p^n$ is the universal PD thickening of $p^n$-torsion. The requirements uniquely pin down
\[
\varphi^1(\gamma_k(a))
\]
for any $a=\sum_{i=0}^{n-1} [r_i] p^i\in W_{n-1}(R^\flat)$, $r_0\in J$, $k\geq 1$. For this, use induction in $k$ to break up the sum and reduce to a single summand. In that case, use the defining properties. Now, the elements $\gamma_k(a)$ generate $I_{\cris,n}(R)$, and are subject to certain universal relations. We have to check that $\varphi^1$ preserves these relations. For this, we have to check that certain relations, each involving finitely many $\varphi^1(\gamma_{k_i}(a_i))$, are satisfied. The coefficients of the $a_i$ give rise to a map
\[
R^\prime = \F_p[X_1^{1/p^\infty},\ldots,X_s^{1/p^\infty},Y_1^{1/p^\infty},\ldots,Y_t^{1/p^\infty}]/(X_1,\ldots,X_s)\to R
\]
sending $X_i$ to the $0$-th coefficients of the $a_i$, and the $Y_j$ to the other coefficients of $a_i$; so in particular, there are elements $a_i^\prime\in W(R^{\prime \flat})$ mapping to $a_i$. We claim that it is enough to check the result over $R^\prime$. If it is true there, the desired relations among the $\varphi^1(\gamma_{k_i}(a_i^\prime))$ follow from the relations defining $A_{\cris,n-1}(R^\prime)$. Each defining relation of $A_{\cris,n-1}(R^\prime)$ gives a defining relation for $A_{\cris,n-1}(R)$ by specialization, and thus the desired relation among the $\varphi^1(\gamma_{k_i}(a_i))$ follows by specialization.

Now we claim that in the universal case of $R^\prime$, $A_\cris(R^\prime)$ is $p$-torsion free, which obviously implies existence of $\varphi^1$. Indeed, by decomposition into a product, one reduces immediately to the cases where $R=\F_p[X^{1/p^\infty}]$ or $R=\F_p[X^{1/p^\infty}]/X$. In the first case, $A_\cris(R) = W(R)$ is clearly $p$-torsion free. In the second case, one has the exact sequence
\[
0\to (X-T)W(R^\flat)\tatealgebra{T}\to W(R^\flat)\tatealgebra{T}\to A_\cris(R)\to 0\ ,
\]
where $\tatealgebra{T}$ denotes the $p$-adically completed free PD polynomial algebra in one variable $T$, and $X\in W(R^\flat)$ denotes the canonical Teichm\"uller lift. Assume that
\[
\sum_{n\geq 0} a_n \gamma_n(T)\in W(R^\flat)\tatealgebra{T}
\]
with $a_n\in W(R^\flat)$, converging to $0$, defines a $p$-torsion element in $A_\cris(R)$. Then there exists
\[
\sum_{n\geq 0} b_n \gamma_n(T)\in W(R^\flat)\tatealgebra{T}
\]
with
\[
(X-T)\sum_{n\geq 0} b_n \gamma_n(T) = p \sum_{n\geq 0} a_n \gamma_n(T)\ .
\]
This means $pa_0 = Xb_0$, $pa_1 = Xb_1 - Tb_0$, and in general $pa_n = Xb_n - nTb_{n-1}$. Using that $X$ is not torsion in $R^\flat$, one checks inductively that $b_n$ is divisible by $p$, which then shows that $\sum_{n\geq 0} a_n \gamma_n(T)$ maps to $0$ in $A_\cris(R)$, as desired.
\end{proof}

\begin{rmk} The final computation showing that $A_\cris(R)$ is $p$-torsion free if $J=\ker(R^\flat\to R)$ is principal and generated by a non-torsion element, is the same as the classical computation showing that Fontaine's ring $A_\cris = A_\cris(\OO_C/p)$ is $p$-torsion free (and generalizes it), where $C/\Q_p$ is an algebraically closed complete extension.

More generally, we have the following proposition. We will not need the following results in the rest of the paper, and include them only for completeness.

\begin{prop} Let $R$ be an f-semiperfect ring such that $R=S/J$, where $S$ is a perfect ring, and $J=(s_1,\ldots,s_n)$ is an ideal generated by a regular sequence $s_1,\ldots,s_n\in S$. Then $A_\cris(R)$ is $p$-torsion free. Moreover, $\Phi\from R\to R$ admits the structure of a PD thickening.
\end{prop}

\begin{proof} Easy and left to the reader. We note that the PD structure on $\Phi\from R\to R$ depends on the choice of the regular sequence $(s_1,\ldots,s_n)$.
\end{proof}

\begin{Cor}\label{Intfullyfaithful} Let $R$ be an f-semiperfect ring such that $R=S/J$, where $S$ is a perfect ring, and $J$ is an ideal generated by a regular sequence. Assume Theorem \ref{fullyfaithful}. Then the Dieudonn\'e module functor on $p$-divisible groups over $R$ is fully faithful.
\end{Cor}

\begin{proof} We follow the proof of Lemma 11 in \cite{FaltingsPeriodDomains}. Let $G, H$ be two $p$-divisible groups over $R$. We have to check that
\[
\Hom(G,H)\to \Hom(\MM(G),\MM(H))
\]
is an isomorphism. As $A_\cris(R)$ is $p$-torsion free, both groups are flat over $\Z_p$. As the kernel and cokernel are $p$-torsion, we see that the kernel is trivial; we have to show that the cokernel is trivial as well. For this, we have to check that if $f\in \Hom(G,H)$ such that $\MM(f)$ is divisible by $p$, then $f$ is divisible by $p$. For this, evaluate $\MM(f)$ at the PD thickening $\Phi\from R\to R$. By assumption, this evaluation is $0$, as $R$ is $p$-torsion. Fixing lifts of $G^\prime$ and $H^\prime$ to $\Phi\from R\to R$, we find that $f$ lifts to $f^\prime\from G^\prime\to H^\prime$, and that the map induced by $f^\prime$ on Lie algebras of the universal vector extension is trivial. We claim that $f^\prime$ factors as a composite
\[
G^\prime\buildrel F\over \to G^{\prime (1)}\to H^{\prime (1)}\buildrel V\over \to H^\prime\ ,
\]
where $G^{\prime (1)}, H^{\prime (1)}$ are the pullbacks of $G^\prime, H^\prime$ via $\Phi\from R\to R$. By duality, it suffices to check factorization over $G^{\prime (1)}$. But $G^\prime\to H^\prime$ is trivial on Lie algebras, which means that the map becomes trivial on the kernel of Frobenius, i.e. the desired factorization. But then we get a map $g^{(1)}\from G^{\prime (1)}\to H^{\prime (1)}$, which is the same thing as a map $g\from G\to H$ over $R$. One checks directly that $pg=f$, as desired.
\end{proof}
\end{rmk}

\subsection{The case $\Q_p/\Z_p\to \mu_{p^\infty}$}

In this section, we will prove the full faithfulness results up to isogeny for the case of homomorphisms from $\Q_p/\Z_p$ to $\mu_{p^\infty}$. We note that the Dieudonn\'{e} module $\MM(\Q_p/\Z_p)$ (resp., $\MM(\mu_{p^\infty})$) is a free $A_\cris(R)$-module of rank $1$, with $F$ acting on a basis element as $p$ (resp., as $1$).

\begin{lemma} There is a canonical identification
\[
\Hom_R(\Q_p/\Z_p,\mu_{p^\infty}) = 1 + J\subset R^\flat\ ,
\]
under which the canonical map
\[
\Hom_R(\Q_p/\Z_p,\mu_{p^\infty})\to \Hom(\MM(\Q_p/\Z_p),\MM(\mu_{p^\infty}))[p^{-1}] = B_\cris^+(R)^{\varphi=p}
\]
is identified with the map
\[
1 + J\to B_\cris^+(R)^{\varphi = p} : r\mapsto \log([r])\ .
\]
\end{lemma}

\begin{proof} This follows from Lemma \ref{ExplicitDieudonne} above. For the first part, note that an element of $\Hom_R(\Q_p/\Z_p,\mu_{p^\infty})$ is given by a sequence $r_0,r_1,\ldots\in R$ such that $r_i=r_{i+1}^p$, and $r_0=1$. This gives rise to an element of $r\in R^\flat$ such that $r$ maps to $1$ in $R$, i.e. $r\in 1+J$.
\end{proof}

Let $I=\ker(W(R^\flat)\to R)$. The following lemma implies that $1 + J\to B_\cris^+(R)^{\varphi = p}$ is injective.

\begin{lemma} \label{gammainjective} Let $w\in 1+I$. If $\log w=0$, then $w=\pm 1$. If $p\neq 2$, then in fact $w=1$.
\end{lemma}

\begin{proof} Assume first $p\neq 2$. Note that $w^p\in 1+pA$. We have $\log w^p=0$, so that
\[
0 = \log w^p=(w^p-1)\left(1-\frac{w^p-1}{2}+\frac{(w^p-1)^2}{3}-\dots\right)\ .
\]
The factor on the right lies in $1+pA\in A^\times$, whence $w^p=1$ (in $A_\cris(R)$, hence in $W(R^\flat)$ by Lemma \ref{WtoAcrisInjective}). But as $R^\flat$ is perfect, there are no nontrivial $p$-th roots in $W(R^\flat)$, hence $w=1$.

If $p=2$, argue similarly with $w^4\in 1+4A$, showing that $w^4=1$. For $p=2$, the only nontrivial $4$-th roots of unity in $W(R^\flat)$ are $\pm 1$, giving the result.
\end{proof}

In the following, we use Proposition \ref{IsogenyReduction} and \ref{DefFSemiperfect} to replace $R$ by an f-semiperfect ring for which $J=\ker(R^\flat\to R)$ satisfies $\Phi(J) = J^p$. For $k\geq 0$, let $I^{(k)}\subset W(R^\flat)$ be the ideal
\[
I^{(k)} = [J^{p^{k}}]W(R^\flat) + p[J^{p^{k-1}}]W(R^\flat) + \ldots + p^k[J]W(R^\flat) + p^{k+1}W(R^\flat)\ ,
\]
so that $I^{(0)}=I$.  In terms of the presentation of Witt vectors as sequences, we have
\[
I^{(k)}=\{(a_0,a_1,\ldots)\mid a_i\in J^{p^{k}}\mathrm{\ for\ } i=0,1,\ldots,k\}\ .
\]
Thus $I^{(k)}$ is the kernel of the homomorphism from $W(R^\flat)$ onto $W_{k+1}(R^{\flat}/J^{p^{k}})$, where $W_i$ is the functor of truncated Witt vectors of length $i$. Since $W(R^\flat)=\varprojlim W_k(R^\flat)$, and $R^\flat=\varprojlim R^\flat/J^{p^k}$, the following lemma is immediate.

\begin{lemma} \label{Wtopology} The ring $W(R^\flat)$ is complete with respect to the linear topology induced by the ideals $I^{(k)}$, i.e., $W(R^\flat)\isom \varprojlim W(R^\flat)/I^{(k)}$. Moreover, for any $i\geq 0$, $(1+I)^{p^i}\subset 1+I^{(i)}$.
\end{lemma}

Recall the quotient $A_\cris(R)\to W(R^\flat) / [J]$ of $A_\cris(R)$ from the proof of Lemma \ref{WtoAcrisInjective}. We let $N\subset A_\cris(R)$ be the kernel. Note that $N$ is generated by $\gamma_n([x])$, for $x\in J$ and $n\geq 1$.

\begin{lemma}\label{pN0} The map $\varphi^1\from I_\cris(R)\to A_\cris(R)$ preserves $N\subset I_\cris(R)$. The restriction $\varphi^1|_N$ is topologically nilpotent.
\end{lemma}

\begin{proof} Recall that $N$ is generated as $A_\cris(R)$-module by $z=\gamma_n([x])$, where $x\in J$ and $n\geq 1$. We compute from the definition of $\varphi^1$,
\[
\varphi^1(z) = \varphi^1(\gamma_n([x])) = \frac{(np)!}{p\cdot n!} \gamma_{np}([x])\in N\ .
\]
Moreover, this is divisible by $(n!)^{p-1}$, from which one easily deduces that $\varphi^1$ is topologically nilpotent.
\end{proof}

Take any $a\in B_\cris^+(R)^{\varphi=p}$. Multiplying $a$ by a power of $p$, we may assume that $a\in A_\cris(R)^{\varphi = p}$, or even that
\[
a\in I_\cris(R)^{\varphi^1 = 1}\ .
\]
As $W(R^\flat)\to W(R^\flat) / [J]$ is surjective, there exist $n\in N$ and $w\in W(R^\flat)$ with $a=w+n$. After multiplying $a$ by $p^2$, we may assume that $w\in p^2 W(R^\flat)$. Consider the element
\[
z=\varphi^1(w)-w=n-\varphi^1(n)\in W\cap N = [J]\ .
\]
Writing $z=\sum_{i\geq 0} [z_i] p^i$, we have $z_i\in J$ for all $i\geq 0$. Observe that
\[
a = w + n = w + z + \varphi^1(z) + (\varphi^1)^2(z) + \ldots + (\varphi^1)^k(z) + \ldots\ ,
\]
because the sum on the right-hand side makes sense by the previous lemma.

Let $AH(t)$ be the Artin-Hasse exponential series:
\[
AH(t)=\exp\left(t+\frac{t^p}{p}+\frac{t^{p^2}}{p^2}+\dots\right)\ .
\]
It is a classical result that $AH(t)\in\Z_p\powerseries{t}$.  Since $z_i\in J$, Lemma \ref{Wtopology} ensures that $AH([z_i])$ converges to an element in $1+I$. Note that $\log AH([z_i])$ does not exist in $W(R^\flat)$, but it does make sense in $A_\cris(R)$.

\begin{lemma} For any semiperfect ring $R$ and any element $z\in J=\ker(R^\flat\to R)$, we have
\[
\log AH([z]) = \sum_{n\geq 0} (\varphi^1)^n ([z])\ .
\]
\end{lemma}

\begin{proof} This is immediate by reduction to the universal case $R=\F_p[z^{1/p^\infty}]/z$, for which we recall that $A_\cris(R)$ is $p$-torsion free.
\end{proof}

Applying Lemma \ref{Wtopology}, the product
\[
\xi = \exp(w)\prod_{i\geq 0} AH([z_i])^{p^i}
\]
converges to an element of $1+I$; here we use that $w\in p^2 W(R^\flat)$ to ensure convergence of the exponential even for $p=2$. Then $\log\xi$ converges in $A_\cris(R)$ to $a=w+n$.

\begin{lemma} There is an equality $\varphi(\xi)=\pm \xi^p$, and in fact $\varphi(\xi) = \xi^p$ if $p\neq 2$.
\end{lemma}

\begin{proof} We have $\log(\varphi(\xi)/\xi^p)=\varphi(a)-pa=0$.   By Lemma \ref{gammainjective}, $\varphi(\xi)=\pm \xi^p$.
\end{proof}

Let $r$ be the image of $\xi$ in $R^\flat$, so that $r\in 1+J$. Let $\zeta=\xi/[r]\in 1+pW(R^\flat)$; then $\varphi(\zeta)=\pm \zeta^p$. A simple induction argument on the highest power of $p$ to divide $\zeta-1$ shows that $\zeta = 1$ (showing a posteriori that, even for $p=2$, there is no sign ambiguity). Thus $\xi=[r]$, and $a=\log([r])$, as desired.

\subsection{A surjectivity result}

The proof of full-faithfulness for general $p$-divisible groups $G$ and $H$ will require a certain surjectivity result that we will formulate in this subsection.

Let $R$ be a semiperfect ring and let $G$ be a $p$-divisible group over $R$. Consider the $R$-algebra $S$ which represents the Tate module $TG$, and the $R$-algebra $\tilde{S}$ which represents $\tilde{G}\times_G \hat{G}$. We recall that $\tilde{S}$ is a relative perfect flat $R$-algebra, and that $S=\tilde{S}/I$ is a flat $R$-algebra, where $I\subset \tilde{S}$ is the finitely generated ideal coming as pullback from the zero section $\{e\}\subset \hat{G}$.

In particular, we see that $S$ is semiperfect (and f-semiperfect if $R$ is f-semiperfect). Now note that there is a universal homomorphism $\Q_p/\Z_p\to G$ over $S$. It induces a map
\[
A_\cris(S) = \MM(\Q_p/\Z_p)(A_\cris(S))\to \MM(G)(A_\cris(R))\otimes_{A_\cris(R)} A_\cris(S)\ ,
\]
i.e. gives an element of $\MM(G)(A_\cris(R))\otimes_{A_\cris(R)} A_\cris(S)$. In turn, it induces a map
\[
A_\cris(S)^\ast = \Hom_{A_\cris(R)}(A_\cris(S),A_\cris(R))\to \MM(G)(A_\cris(R))\ .
\]

\begin{Theorem}\label{SurjResult} The cokernel of the map
\[
A_\cris(S)^\ast\to \MM(G)(A_\cris(R))
\]
is killed by $p^2$.
\end{Theorem}

\begin{rmk} This theorem can be regarded as the existence of many morphisms $\Q_p/\Z_p\to G$ over certain $R$-algebras. Note that if $R=\F_p$ and $G=\mu_{p^\infty}$, then no such morphisms exist over noetherian rings: One has to go to big rings to get these morphisms. Unfortunately, our proof will rely heavily on the theory of perfectoid spaces, and it would be desirable to have a more elementary argument. We should also note that it is probably possible to improve on the result; for example, for $p\neq 2$, our proof gives $p$ in place of $p^2$.
\end{rmk}

First we show that is enough to prove Theorem \ref{SurjResult} in the case that $R$ is a perfect ring.

\begin{prop} \label{perfectenough} Suppose that Theorem \ref{SurjResult} holds whenever $R$ is a perfect ring. Then it holds unconditionally.
\end{prop}

\begin{proof} Take any semiperfect ring $R$, and let $R^\flat$ be its inverse perfection, with $J=\ker(R^\flat\to R)$, as usual. Fix a lift $G^\prime$ of $G$ to $R^\flat$. Let $S^\prime$ be the $R^\flat$-algebra which represents $TG^\prime$. Since $G$ is the base change of $G^\prime$ along $R^\flat\to R$, we have $S=S^\prime\otimes_{R^\flat} R$. The map $S^\prime\to S$ is therefore a surjection with kernel $JS^\prime$.

\begin{lemma}\label{ASAS} There is an isomorphism of $A_\cris(R)$-algebras
\[
A_\cris(S)\isom A_\cris(S^\prime) \hat{\otimes}_{W(R^{\flat})} A_\cris(R)\ .
\]
\end{lemma}

\begin{proof} Let $\tilde{S}^\prime$ represent $\tilde{G}^\prime\times_{G^\prime} \hat{G}^\prime$. Then $\tilde{S}^\prime$ is relatively perfect over $R^\flat$, thus perfect. Let $I^\prime\subset \tilde{S}^\prime$ be the kernel of $\tilde{S}^\prime\to S^\prime$. As $I^\prime$ is topologically nilpotent, one finds that
\[
S^\flat = (S^\prime)^\flat = \tilde{S}^\prime\ .
\]
Also note that the kernel of the surjection $\tilde{S}^\prime\to S$ is given by $I^\prime + J\tilde{S}^\prime$. Moreover, $I^\prime$ and $J\tilde{S}^\prime$ satisfy
\[
I^\prime\cap J\tilde{S}^\prime = I^\prime J\ .
\]
This follows from the various flatness assertions for $\tilde{S}^\prime$, $I^\prime$ and $S^\prime$ over $R^\prime$. In particular, \cite[Proposition 3.10]{BerthelotOgusNotes} implies that giving divided powers on $I^\prime + J\tilde{S}^\prime$ is equivalent to giving divided powers on $I^\prime$ and $J\tilde{S}^\prime$ separately.

Now observe that one obtains $A_\cris(S^\prime)$ by adding divided powers of $I^\prime$, and one gets $A_\cris(R)$ from $W(R^\flat)$ by adding divided powers of $J$ (and completing $p$-adically). It follows that
\[
A_\cris(S^\prime) \hat{\otimes}_{W(R^{\flat})} A_\cris(R)
\]
is the universal $p$-adically complete $W(\tilde{S}^\prime)$-algebra with divided powers for $I^\prime + J\tilde{S}^\prime$. But as $I^\prime + J\tilde{S}^\prime = \ker(\tilde{S}^\prime\to S)$, this is exactly $A_\cris(S)$, as desired.
\end{proof}

We now turn to the proof of Proposition \ref{perfectenough}. Since $R^\flat$ is perfect, $A_\cris(R^\flat)=W(R^\flat)$.  By hypothesis, the cokernel of
\[
\Hom_{W(R^\flat)}( A_\cris(S^\prime), W(R^\flat)) \to \MM(G^\prime)(W(R^\flat))
\]
is killed by $p^2$.  We base change along $W(R^\flat)\to A_\cris(R)$ to find that the cokernel of
\[
\Hom_{A_\cris(R)}( A_\cris(S^\prime)\hat{\otimes}_{W(R^\flat)} A_\cris(R), A_\cris(R)) \to \MM(G^\prime)\otimes_{W(R^\flat)} A_\cris(R) = \MM(G)(A_\cris(R))
\]
is also killed by $p^2$. By Lemma \ref{ASAS}, $A_\cris(S^\prime)\hat{\otimes}_{W(R^\flat)} A_\cris(R)\isom A_\cris(S)$.  Therefore Theorem \ref{SurjResult} holds for $R$.
\end{proof}

Next, we will explain a method to construct many morphisms $\Q_p/\Z_p\to G$. For this, we will need big rings. More precisely, let $C$ be a complete, algebraically closed nonarchimedean field extension of $\Q_p$, and let $(T,T^+)$ be a perfectoid affinoid $(C,\OO_C)$-algebra for which $\Spec T$ is connected and has no nontrivial finite \'etale covers.  Let $G$ be a $p$-divisible group over $T^+$.

\begin{lemma} \label{ispdivisible} The abelian group $G^\ad_\eta(T,T^+)=G(T^+)$ is $p$-divisible.
\end{lemma}

\begin{proof} Pulling back the multiplication by $p$ morphism $G^\ad_\eta\to G^\ad_\eta$ to $\Spa(T,T^+)$ gives a finite \'{e}tale cover of $\Spa(T,T^+)$, i.e. a finite \'{e}tale $T$-algebra by \cite{Sch}. By assumption, this admits a section, giving the result.
\end{proof}

Let $G^0(T^+)=\ker(G(T^+)\to G(T^+/p))$. Then we have the exact sequence
\[
0\to TG(T^+)\to TG(T^+/p) \to \ G^0(T^+) \to 0\ .
\]
Here the map $TG(T^+/p)\to G^0(T^+)$ takes a sequence in $\varprojlim G[p^n](T^+/p)\subset \tilde{G}(T^+/p)$, lifts it uniquely to $\tilde{G}(T^+)$, and projects it onto $G(T^+)$. The sequence is exact on the right because of Lemma \ref{ispdivisible}.

We may identify $TG(T^+/p)$ with $\Hom_{T^+/p}(\Q_p/\Z_p, G\otimes_{T^+} T^+/p)$. Passage to Dieudonn\'e modules gives a map $TG(T^+/p)\to \MM(G)(T^+)$. Using Lemma \ref{QLogvsLog} and Lemma \ref{ExplicitDieudonne}, we get a commutative diagram
\[\xymatrix{
0 \ar[r] & TG(T^+) \ar[r] \ar[d] & TG(T^+/p) \ar[r] \ar[d] & G^0(T^+) \ar[d]^{\log_G} \ar[r] & 0 \\
0 \ar[r] & (\Lie G^\vee)^\vee \ar[r] & \MM(G)(T^+) \ar[r] & \Lie G \ar[r] & 0\ .
}\]
Note that a priori we only know commutativity after composition with $\Lie G\to (\Lie G)[p^{-1}]$, but this map is injective. The cokernel of the logarithm map $G^0(T^+)\to \Lie G$ is killed by $p^{1+\epsilon}$ for any $\epsilon>0$, because $\exp_G$ converges on $p^{1+\epsilon}\Lie G$.

Moreover, we have the following Hodge-Tate decomposition over $(T,T^+)$.

\begin{prop}\label{HodgeTateSequence} The cokernel of $TG(T^+)\otimes_{\Z_p} T^+\to (\Lie G^\vee)^\vee$ is killed by $p^{1/(p-1)+\epsilon}$ for any $\epsilon>0$. More precisely, there is a natural sequence
\[
0\to (\Lie G)(1)\to TG(T^+)\otimes_{\Z_p} T^+\to (\Lie G^\vee)^\vee\to 0
\]
whose cohomology groups are killed by $p^{1/(p-1)+\epsilon}$ for any $\epsilon>0$. Here, $(1)$ denotes a Tate twist $\otimes_{\Z_p} \Z_p(1)$.
\end{prop}

\begin{rmk} Let $\alpha_G\from TG(T^+)\otimes_{\Z_p} T^+\to (\Lie G^\vee)^\vee$ denote the map defined previously. Then the first map
\[
(\Lie G)(1)\to TG(T^+)\otimes_{\Z_p} T^+
\]
is given by $\alpha_{G^\vee}^\vee(1)$.
\end{rmk}

\begin{proof} In the case that $(T,T^+) = (C,\OO_C)$, the statement is exactly \cite[Th\'eor\`eme II.1.1]{FarguesGenestierLafforgue}. In general, we work over the perfectoid space $X=\Spa(T,T^+)$. All three terms of this sequence are finite projective $T^+$-modules, and we get an associated sequence (not necessarily a complex, not necessarily exact)
\[
0\to \mathscr{F}_1\to \mathscr{F}_2\to \mathscr{F}_3\to 0
\]
of finite projective $\OO_X^+$-modules, whose global sections recover the original sequence. Here, we use that $X$ is an honest adic space, i.e. we make use of the sheaf property.

For all points $x\in X$, given by $\Spa(K,K^+)$, the associated sequence
\[
0\to \widehat{\mathscr{F}_{1,x}}\to \widehat{\mathscr{F}_{2,x}}\to \widehat{\mathscr{F}_{3,x}}\to 0
\]
of $\widehat{\OO_{X,x}^+} = K^+$-modules has cohomology killed by $p^{1/(p-1)+\epsilon}$ for any $\epsilon>0$. Indeed, using the $\epsilon$, we may replace $K^+$ by $\OO_K$, and then the result follows from \cite[Th\'eor\`eme II.1.1]{FarguesGenestierLafforgue}, by extending to a complete algebraic closure, noting that the rank of the Tate module was already maximal from the start by (the proof of) Lemma \ref{ispdivisible}.

Now, for any $n$, consider the sequence
\[
0\to \mathscr{F}_1/p^n\to \mathscr{F}_2/p^n\to \mathscr{F}_3/p^n\to 0\ .
\]
We claim that it is a complex, and that its cohomology groups are annihilated by $p^{2/(p-1)+\epsilon}$ for any $\epsilon>0$. It is enough to check this on stalks. But on stalks, it is immediate from the previous discussion.

But the global sections of $\mathscr{F}_i/p^n$ are almost isomorphic to $H^0(X,\mathscr{F}_i/p^n)$, by the results of \cite{Sch}; moreover, the higher $H^j(X,\mathscr{F}_i/p^n)$, $j>0$, are almost zero. It follows that the sequence
\[
0\to H^0(X,\mathscr{F}_1)/p^n\to H^0(X,\mathscr{F}_2)/p^n\to H^0(X,\mathscr{F}_3)/p^n\to 0
\]
is (almost) a complex, and the cohomology groups are killed by $p^{2/(p-1)+\epsilon}$. Upon passing to the inverse limit over $n$, we get the desired result, except for slightly worse constants. As we will not need the precise constants, we leave it to the reader to fill in the details showing that one can get $p^{1/(p-1)+\epsilon}$.
\end{proof}

This discussion implies the following proposition.

\begin{prop} \label{TGTp} The cokernel of $TG(T^+/p)\otimes_{\Z_p} T^+\to \MM(G)(T^+)$ is killed by $p^{2+\epsilon}$ for any $\epsilon>0$.
\end{prop}

\begin{rmk}\label{Countability} For convenience, we had assumed that $\Spec T$ has no nontrivial finite \'etale covers. To reach the conclusion, it would have been enough to start with some base perfectoid affinoid $(C,\OO_C)$-algebra $(T_0,T_0^+)$ and adjoin a countable number of finite \'etale covers (each of which gives a perfectoid affinoid algebra, by the almost purity theorem). Indeed, we only had to assume that the Tate module has enough sections over $T^+$, which gives countably many finite \'etale covers, and that we can lift countably many elements of $G^0(T^+)$ to $TG(T^+/p)$, each of which similarly amounts to a countable number of finite \'etale covers.
\end{rmk}

With these preparations, we can give the proof of Theorem \ref{SurjResult}.

\begin{proof} {\it (of Theorem \ref{SurjResult}.)} We have reduced to the case that $R$ is perfect. We can also assume that $R$ is local, and in particular $\Spec R$ is connected. Let $C$ be the completion of an algebraic closure of $\Q_p$. Let $T_0^+=W(R)\otimes_{\Z_p} \OO_C$, and let $T_0 = T_0^+[p^{-1}]$; then $(T_0,T_0^+)$ is a perfectoid affinoid $(C,\OO_C)$-algebra. Note that $\Spec T_0$ is still connected. Choose a $T_0$-algebra $\tilde{T}_0$ that is a direct limit of faithfully flat finite \'{e}tale $T_0$-algebras, and such that $\Spec \tilde{T}_0$ is connected and does not admit nontrivial finite \'{e}tale covers. Let $\tilde{T}_0^+\subset \tilde{T}_0$ be the integral closure of $T_0^+$, and define $(T,T^+)$ as the completion of $(\tilde{T}_0,\tilde{T}_0^+)$ (with respect to the $p$-adic topology on $\tilde{T}_0^+$).

Let $G$ be a $p$-divisible group over $R$, and fix a lift $H$ of $G\otimes_R T^+/p$ to $T^+$. Applying Proposition \ref{TGTp} to $H$ shows that the cokernel of
\begin{equation}\label{whatweknow}
TG(T^+/p)\otimes_{\Z_p} T^+\to \MM(H)(T^+)=\MM(G)\otimes_{W(R)} T^+
\end{equation}
is killed by $p^{2+\epsilon}$. In fact, as observed in Remark \ref{Countability}, it would have been enough to adjoin countably many finite \'etale covers to $T_0$, and so we assume in the following that $T$ is the completion of a direct limit of countably many faithfully flat finite \'etale $T_0$-algebras.

\begin{lemma} \label{faithfullyflat} The algebra $(T,T^+)$ is a perfectoid affinoid $(C,\OO_C)$-algebra, and $T^+$ is almost faithfully flat over $T_0^+$. Moreover, for any $\epsilon>0$, one can find $W(R)$-subalgebras $S_i^+\subset T^+$, such that $T^+$ is the filtered union of the $S_i^+$, and such that for each $i$, there are maps
\[
f_i\from T^+\to W(R)^{n_i}\ ,\ g_i\from W(R)^{n_i}\to S_i^+
\]
for which the composite
\[
S_i^+\to T^+\to W(R)^{n_i}\to S_i^+
\]
is multiplication by $p^\epsilon$.
\end{lemma}

\begin{proof} The first sentence is a direct consequence of the almost purity theorem in the form given in \cite{Sch}. In fact, one even knows that one can write $T^+$ as the completion of a union of $T_0^+$-subalgebras $T_i^+$ for which $T_i=T_i^+[p^{-1}]$ is finite \'{e}tale over $T_0$, and for which $T_i^+$ is an almost direct summand of $(T_0^+)^{n_i}$ for some $n_i\geq 1$. In particular, there are maps
\[
T_i^+\to (T_0^+)^{n_i}\to T_i^+
\]
whose composite is $p^\epsilon$. One can find a finite extension $K_i$ of $\Q_p$ and a finite \'etale $W(R)\otimes_{\Z_p} K_i$-algebra $S_i$ such that $T_i = S_i\otimes_{K_i} C$. Let $S_i^+\subset S_i$ be the integral closure of $W(R)$. Taking $K_i$ large enough, one can assume that there are maps
\[
S_i^+\to (W(R)\otimes_{\Z_p} \OO_{K_i})^{n_i}\to S_i^+
\]
whose composition is $p^\epsilon$. One can identify the middle term with $W(R)^{n_i^\prime}$ for some $n_i^\prime$. Therefore, it is enough to show that there is a map $T^+\to S_i^+$ such that $S_i^+\to T^+\to S_i^+$ is multiplication by $p^{2\epsilon}$.

Possibly enlarging $K_i$, one can assume that the cokernel of $S_i^+\otimes_{\OO_{K_i}} \OO_C\to T_i^+$ is killed by $p^\epsilon$; in that case, any chosen projection $\OO_C\to \OO_{K_i}$ gives a retraction $T_i^+\to p^{-\epsilon} S_i^+$. It remains to show that there is a retraction of $T_i^+\to T^+$ up to $p^\epsilon$.

For this, we use the countability statement. Write $T_1^+ = T_i^+$, and enumerate the other $T_j^+$ for $j\geq i$ as $T_1^+\subset T_2^+\subset \ldots$. For each $j\geq 1$, $T_j^+$ is almost a direct summand of $T_{j+1}^+$ (both being uniformly almost finite projective $T_0^+$-modules), so that one can choose retractions $T_{j+1}^+\to T_j^+$ up to $p^{\epsilon/2^j}$. By composition, we get maps $T_j^+\to T_1^+$ for all $j$, and multiplying by an appropriate $p$-power, they become a compatible system of retractions up to $p^\epsilon$. Taking their direct limit gives a retraction $T^+\to T_1^+$ up to $p^\epsilon$, as desired.
\end{proof}

In the following, we use duals to denote the $W(R)$-linear dual, so e.g.
\[
(T^+)^\ast = \Hom_{W(R)}(T^+,W(R))\ .
\]

\begin{lemma} \label{TGTpcok} The cokernel of the $W(R)$-linear map
\begin{equation}\label{TGTpcokeq}
TG(T^+/p)\hat{\otimes}_{\Z_p} (T^+)^\ast\to \MM(G)\ ,
\end{equation}
given as the composite of
\[
TG(T^+/p)\hat{\otimes}_{\Z_p} (T^+)^\ast\to (\MM(G)\otimes_{W(R)} T^+)\hat{\otimes}_{\Z_p} (T^+)^\ast
\]
with the contraction of the second and third factor, is killed by $p^2$.
\end{lemma}

\begin{proof} We argue modulo $p^k$ for any $k\geq 3$. We claim that it suffices to check that the cokernel of
\[
\gamma\from \left(TG(T^+/p)\otimes_{\Z_p} (T^+)^\ast\right) \otimes_{W(R)} T^+/p^k \to \MM(G)\otimes_{W(R)} T^+/p^k
\]
is killed by $p^{3-\epsilon}$ for some $\epsilon>0$. Indeed, let $X$ be the cokernel of \eqref{TGTpcokeq} (modulo $p^k$), and let $Y$ be the image of multiplication by $p^2$ on $X$. Assuming the result for $\gamma$, we find that $Y\otimes_{W(R)} T^+/p^k$ is killed by $p^{1-\epsilon}$ for some $\epsilon>0$. But then $Y\otimes_{W(R)} T_0^+/p^k$ is killed by $p^{1-\epsilon}$ for some $\epsilon>0$, as $T^+/p^k$ is almost faithfully flat over $T_0^+/p^k$. Now $Y\otimes_{W(R)} T_0^+/p^k = Y\otimes_{\Z_p} \OO_C/p^k = Y\otimes_{\F_p} \OO_C/p$, where $Y$ is an $\F_p$-vector space. Thus $Y\otimes_{\F_p} \OO_C/p$ is free. If it is killed by $p^{1-\epsilon}$, it follows that $Y=0$, as desired.

We recall that we want to show that the cokernel of
\[
\gamma\from \left(TG(T^+/p)\otimes_{\Z_p} (T^+)^\ast\right) \otimes_{W(R)} T^+/p^k \to \MM(G)\otimes_{W(R)} T^+/p^k
\]
is killed by $p^{3-\epsilon}$ for some $\epsilon>0$. Take any $t\in TG(T^+/p)$. Write $\beta\from TG(T^+/p)\to \MM(G)\otimes_{W(R)} T^+/p^k$ for the canonical map. We claim that $p^\epsilon \beta(t)$ lies in the image of $\gamma$. This will give the result, as the image of $\gamma$ is then a $T^+/p^k$-submodule of $\MM(G)\otimes_{W(R)} T^+/p^k$ containing $p^\epsilon \beta(TG(T^+/p))$, giving the result by \eqref{whatweknow}.

Now modulo $p^k$, $\beta(t)\in \MM(G)\otimes_{W(R)} T^+/p^k$ will lie in $\MM(G)\otimes_{W(R)} S_i^+/p^k$ for $i$ large enough, where we use the algebras $S_i^+$ constructed in Lemma \ref{faithfullyflat}.

Note that the maps constructed in Lemma \ref{faithfullyflat} give rise to an element $x\in (T^+)^\ast\otimes_{W(R)} S_i^+$. Then
\[
t\otimes x\in TG(T^+/p)\otimes_{\Z_p} (T^+)^\ast \otimes_{W(R)} S_i^+\left(\to (TG(T^+/p)\otimes_{\Z_p} (T^+)^\ast) \otimes_{W(R)} T^+\right)
\]
maps under $\gamma$ to the image of
\[
\beta(t)\otimes x\in \left(\MM(G)\otimes_{W(R)} S_i^+/p^k\right)\otimes_{\Z_p} \left((T^+)^\ast\otimes_{W(R)} S_i^+\right)
\]
under the evaluation map
\[
\left(\MM(G)\otimes_{W(R)} S_i^+/p^k\right)\otimes_{\Z_p} \left((T^+)^\ast\otimes_{W(R)} S_i^+\right)\to \MM(G)\otimes_{W(R)} S_i^+/p^k
\]
contracting the second and third factor. But we have chosen $x$ so that it acts through multiplication by $p^\epsilon$ on $S_i^+$. This means that the result will be $p^\epsilon\beta(t)$, as desired.
\end{proof}

We can now complete the proof of Theorem \ref{SurjResult}. Recall that $S$ is the $R$-algebra which represents $TG$. We want to prove that the cokernel of
\[
A_\cris(S)^\ast\to \MM(G)
\]
is killed by $p^2$. It is enough to prove this result modulo $p^k$ for any $k\geq 3$. Let $m\in p^2\MM(G)$.  By Lemma \ref{TGTpcok}, $x$ has a preimage in $TG(T^+/p)\otimes_{\Z_p} (T^+)^\ast$ modulo $p^k$. We write this preimage as $\sum_{i=1}^n a_i\otimes \lambda_i$, with $a_i\in TG(T^+/p)$, $\lambda_i\in (T^+)^\ast$. It is enough to prove that the image of $a_i\otimes \lambda_i$ in $\MM(G)$ is also in the image of $A_\cris(S)^\ast\to \MM(G)$. But $a_i\in TG(T^+/p)$ gives rise to a map $S\to T^+/p$ carrying the universal element of $TG(S)$ to $a_i\in TG(T^+/p)$. We get a morphism
\[
A_\cris(S)\to A_\cris(T^+/p)\to T^+\ ,
\]
using the usual map $\Theta: A_\cris(T^+/p)\to T^+$, coming from the fact that $T^+\to T^+/p$ is a PD thickening. By duality, this gives a map $(T^+)^\ast\to A_\cris(S)^\ast$, through which we can consider $\lambda_i$ as an element $\lambda_i^\prime\in A_\cris(S)^\ast$. It is now an easy diagram chase to check that $\lambda_i^\prime$ maps to the image of $a_i\otimes \lambda_i$ in $\MM(G)$. This completes the proof of Theorem \ref{SurjResult}.
\end{proof}

\subsection{The general case}

In this section, we will finish the proof of Theorem \ref{fullyfaithful}. This follows an idea we learnt from a paper of de Jong and Messing, \cite{deJongMessing} (see the proof of Prop. 1.2), namely in order to prove full faithfulness for morphisms $G\to H$, one uses base-change to the ring with universal homomorphisms $\Q_p/\Z_p\to G$ and $H\to \mu_{p^\infty}$, over which one applies the result for the special case of morphisms $\Q_p/\Z_p\to \mu_{p^\infty}$.

We need a basic lemma on Hopf algebras. Let $G$ and $H$ be finite locally free group schemes over any ring $R$, with coordinate rings $A_G$, $A_H$. The coordinate ring of $G\times H^\vee$ is $\Hom_R(A_H,A_G)$ ($R$-module homomorphisms), so that the latter has the structure of an $R$-algebra.

\begin{lemma} \label{HopfAlgLemma}
If $f,g\from G\to H$ are group homomorphisms, corresponding to elements $\alpha,\beta\in\Hom_R(A_H,A_G)$, then $f+g$ corresponds to $\alpha\beta$, the product being computed as multiplication in the natural Hopf algebra structure.
\end{lemma}

\begin{proof} Easy and left to the reader.
\end{proof}

Now let $G$ be a Barsotti-Tate group over an f-semiperfect ring $R$. Let $A_{G[p^n]}$ be the coordinate ring of $G[p^n]$. We had observed earlier that $TG$ is represented by the f-semiperfect ring
\[
A_G = \varinjlim A_{G[p^n]}\ ,
\]
where the direct limit is taken along the $p$-th power maps $G[p^{n+1}]\to G[p^n]$. We note that this is naturally a Hopf algebra. For instance, if $G=\Q_p/\Z_p$, then $A_{G[p^n]}$ is the algebra of functions $\Z_p/p^n\Z_p\to R$, and $A_G$ is the algebra of locally constant functions from $\Z_p$ into $R$. On the other hand, if $G=\mu_{p^\infty}$, then we may identify $A_{G[p^n]}$ with $R[T]/(T^{p^n}-1)$ and $A_G$ with $R[T^{1/p^\infty}]/(T-1)$.

On the other hand, we can consider
\[
A_G^\prime = \varprojlim A_{G[p^n]}\ ,
\]
where the inverse limit is taken along the homomorphisms corresponding to the immersions $G[p^n]\injects G[p^{n+1}]$. Then $A_G^\prime$ is a complete topological ring, which however is in general not adic, as can be observed in the example of $G=\Q_p/\Z_p$, where $A_G^\prime$ is the algebra of arbitrary functions from $\Q_p/\Z_p$ to $R$. However, $A_G^\prime$ is still a topological Hopf algebra.

Let $G^\vee$ be the Serre dual of $G$; then the coordinate ring of $G^\vee[p^n]$ may be identified with the $R$-module of $R$-linear maps $\Hom_R(A_{G[p^n]},R)$ endowed with the $R$-algebra structure induced by comultiplication in $A_{G[p^n]}$. The Tate module $TG^\vee$ is represented by
\[
A_{G^\vee}=\varinjlim \Hom_R(A_{G[p^n]},R)=\Hom_{R,\cont}(A_G^\prime,R)
\]
Here $\Hom_{R,\cont}$ means continuous $R$-module homomorphisms, and $R$ itself is given the discrete topology. Lemma \ref{HopfAlgLemma} has obvious extensions to the coordinate rings $A_G$ and $A_G^\prime$.

Now take two $p$-divisible groups $G$ and $H$ over $R$. Let $A_G$ and $A_{H^\vee}$ represent $TG$ and $TH^\vee$, respectively; these are f-semiperfect. Let
\[
S =A_G\otimes_R A_{H^\vee}\cong \Hom_{R,\cont}(A_H^\prime,A_G)\ ,
\]
so that $S$ is an f-semiperfect ring representing $TG\times TH^\vee$, i.e. there are universal morphisms $\Q_p/\Z_p\to G$, $\Q_p/\Z_p\to H^\vee$ over $S$. The latter also gives a universal morphism $H\to \mu_{p^\infty}$.

Passing to Dieudonn\'e modules, we find morphisms $\MM(\Q_p/\Z_p)(A_\cris(S))\to \MM(G)(A_\cris(S))$ and $\MM(H)(A_\cris(S))\to \MM(\mu_{p^\infty})(A_\cris(S))$ of Dieudonn\'{e} modules. Composing these with the base change to $S$ of a given morphism of Dieudonn\'{e} modules $f\from \MM(G)(A_\cris(R))\to \MM(H)(A_\cris(R))$ gives a morphism of Dieudonn\'{e} modules
\[
\beta_f\from \MM(\Q_p/\Z_p)(A_\cris(S))\to \MM(\mu_{p^\infty})(A_\cris(S))\ ,
\]
which in turn gives rise to a morphism
\[
\eta_f\in \Hom_S(\Q_p/\Z_p, \mu_{p^\infty})[p^{-1}]\ .
\]

\begin{prop}\label{Injectivity} If $\eta_f=0$, then $f$ is $p$-torsion.
\end{prop}

\begin{proof} Note that if $\eta_f = 0$, then $\beta_f$ is $p$-torsion, so it is enough to prove that if $\beta_f = 0$, then $p^4 f = 0$. At this point, we need Theorem \ref{SurjResult}. Assume that $\beta_f=0$ and consider the diagram
\[
\xymatrix{
&  \MM(G)(A_\cris(A_{H^\vee})) \ar[r]^{f_{A_{H^\vee}}} \ar[d] & \MM(H)(A_\cris(A_{H^\vee})) \ar[d] \ar[r]^h & \MM(\mu_{p^\infty})(A_\cris(A_{H^\vee}))\ar[d] \\
\MM(\Q_p/\Z_p)(A_\cris(S)) \ar@/_.3in/[rrr]_{\beta_f=0} \ar[r] & \MM(G)(A_\cris(S)) \ar[r]_{f_S} & \MM(H)(A_\cris(S)) \ar[r]_{h_S} & \MM(\mu_{p^\infty})(A_\cris(S))\ .
}
\]
Let $\alpha$ be the image of $1\in A_{\cris}(S)=\MM(\Q_p/\Z_p)(A_\cris(S))$ in $\MM(G)(A_{\cris}(S))$. Take any $v\in \MM(G)(A_\cris(A_{H^\vee}))$. Let $A_{\cris}(S)^*$ be the $A_{\cris}(A_{H^\vee})$-linear dual of $A_{\cris}(S)$. By Theorem \ref{SurjResult}, there exists $\lambda\in A_{\cris}(S)^*$ such that
\[
\lambda(\alpha)=p^2 v\in \MM(G)(A_\cris(A_{H^\vee}))\ .
\]
The element $\lambda$ induces a commutative diagram
\[
\xymatrix{
&  \MM(G)(A_\cris(A_{H^\vee})) \ar[r]^{f_{A_{H^\vee}}} & \MM(H)(A_\cris(A_{H^\vee})) \ar[r]^h & \MM(\mu_{p^\infty})(A_\cris(A_{H^\vee})) \\
\MM(\Q_p/\Z_p)(A_\cris(S)) \ar@/_.3in/[rrr]_{\beta_f=0} \ar[r] & \MM(G)(A_\cris(S)) \ar[r]_{f_S}\ar[u]^\lambda & \MM(H)(A_\cris(S)) \ar[r]_{h_S}\ar[u]^\lambda & \MM(\mu_{p^\infty})(A_\cris(S))\ar[u]^\lambda\ ,
}
\]
which shows that $(h\circ f_{A_{H^\vee}})(p^2 v) = 0$, i.e. $p^2 (h\circ f_{A_{H^\vee}}) = 0$.

Now, look at the commutative diagram
\[
\xymatrix{
\MM(G)(A_\cris(R)) \ar[r]^{p^2 f} \ar[d] & \MM(H)(A_\cris(R)) \ar[d] \ar[dr]^g & \\
\MM(G)(A_\cris(A_{H^\vee})) \ar@/_.3in/[rr]_{=0}\ar[r]_{p^2 f_{A_{H^\vee}}} & \MM(H)(A_\cris(A_{H^\vee})) \ar[r]_h & \MM(\mu_{p^\infty})(A_\cris(A_{H^\vee}))\ .
}
\]
We claim that the kernel of
\[
g\from \MM(H)(A_\cris(R))\to \MM(\mu_{p^\infty})(A_\cris(A_{H^\vee}))\cong A_\cris(A_{H^\vee})
\]
is killed by $p^2$. Indeed, use Theorem \ref{SurjResult} for $H^\vee$ to see that
\[
\Hom_{A_\cris(R)}(A_\cris(A_{H^\vee}),A_\cris(R))\to \Hom_{A_\cris(R)}(\MM(H)(A_\cris(R)),A_\cris(R))
\]
has cokernel killed by $p^2$. Applying $\Hom(-,A_\cris(R))$ shows that the kernel of $g$ is killed by $p^2$. As $g\circ (p^2f) = 0$, it follows that $p^4 f=0$, as desired.
\end{proof}

Now we observe that the morphism
\[
\eta_f\in \Hom_S(\Q_p/\Z_p, \mu_{p^\infty})[p^{-1}]
\]
over $S$ is equivalent to a family of elements $s_n\in S$, $n\in \Z$, with $s_{n+1}^p=s_n$ and $s_n=1$ for $n$ sufficiently negative. We have identified $S$ with $\Hom_{R,\cont}(A_H^\prime,A_G)$. Under this identification, we have a family of continuous $R$-linear maps $r_{f,n}\from A_H^\prime\to A_G$.

\begin{prop} \label{honestHopf} The continuous $R$-linear map $r_{f,n}\from A_H^\prime\to A_G$ is a morphism of Hopf algebras.
\end{prop}

\begin{proof} First we prove a lemma concerning the functoriality of $f\mapsto r_{f,n}$.

\begin{lemma} \label{psiGdiagram} Let $\psi_G\from G_1\to G_2$ and $\psi_H\from H_1\to H_2$ be two morphisms of Barsotti-Tate groups over $R$, and let $f_1\from \MM(G_1)\to\MM(H_1)$ and $f_2\from\MM(G_2)\to\MM(H_2)$ be morphisms of crystals.
Assume that the diagram
\[\xymatrix{
\MM(G_1) \ar[r]^{\MM(\psi_H)} \ar[d]_{f_1} & \MM(G_2) \ar[d]^{f_2} \\
\MM(H_1) \ar[r]_{\MM(\psi_G)} &              \MM(H_2)
}\]
commutes. Then for $n\in \Z$, the diagram
\[\xymatrix{
A_{H_2}^\prime \ar[r]^{\psi_H^\prime} \ar[d]_{r_{f_2,n}} & A_{H_1}^\prime \ar[d]^{r_{f_1,n}} \\
A_{G_2} \ar[r]_{\psi_G^\ast} &              A_{G_1}
}\]
commutes.
\end{lemma}

\begin{proof}
Let $S_1$ and $S_2$ represent the Tate modules $TG_1\times TH_1^\vee$ and $TG_2\times TH_2^\vee$, respectively. Also let $S_3$ represent $TG_1\times TH_2^\vee$. The morphisms
\[
\xymatrix{
& TG_1\times TH_2^\vee \ar[dl]_{1\times T\psi_{H^\vee}} \ar[dr]^{T\psi_G\times 1}&\\
TG_1\times TH_1^\vee & & TG_2\times TH_2^\vee
}
\]
induce $R$-algebra homomorphisms $S_1,S_2\to S_3$. The following lemma follows easily from the definitions.

\begin{lemma}\label{triangles}
The base change to $S_3$ of the universal morphisms $\Q_p/\Z_p\to (G_1)_{S_1}$ and $\Q_p/\Z_p\to (G_2)_{S_2}$ fit into a commutative diagram
\[
\xymatrix{
& \Q_p/\Z_p \ar[dl] \ar[dr] \\
(G_1)_{S_3} \ar[rr]_{(\psi_G)_{S_3}} & & (G_2)_{S_3}.
}
\]
Similarly, the base change to $S_3$ of the universal morphisms $(H_1)_{S_1}\to \mu_{p^\infty}$ and $(H_2)_{S_2}\to \mu_{p^\infty}$ fit into a commutative diagram
\[
\xymatrix{
(H_1)_{S_3} \ar[dr] \ar[rr]^{(\psi_H)_{S_3}} & & (H_2)_{S_3} \ar[dl] \\
& \mu_{p^\infty}. &
}
\]
\end{lemma}

For $i=1,2$, $f_i$ induces a morphism $\eta_{f_i}\in \Hom_{S_i}(\Q_p/\Z_p,\mu_{p^\infty})[p^{-1}]$. We claim that these morphisms agree upon base change to $S_3$. Consider the diagram

\[\xymatrix{
\MM(\Q_p/\Z_p)(A_\cris(S_3)) \ar@/_1in/[ddd]_{\beta_{f_1}|_{S_3}}\ar[r]^{=} \ar[d] & \MM(\Q_p/\Z_p)(A_\cris(S_3)) \ar@/^1in/[ddd]^{\beta_{f_2}|_{S_3}}\ar[d] \\
\MM(G_1)(A_\cris(S_3)) \ar[r] \ar[d] & \MM(G_2)(A_\cris(S_3)) \ar[d] \\
\MM(H_1)(A_\cris(S_3)) \ar[r] \ar[d] & \MM(H_2)(A_\cris(S_3)) \ar[d] \\
\MM(\mu_{p^\infty})(A_\cris(S_3)) \ar[r]_{=} & \MM(\mu_{p^\infty})(A_\cris(S_3))\ .
}\]
The center square in this diagram commutes by hypothesis. The top and bottom squares in the diagram commute by applying $\MM$ to the commutative triangles in Lemma \ref{triangles}. Thus $\beta_{f_1}|_{S_3}=\beta_{f_2}|_{S_3}$. But this implies that $\eta_{f_1}|_{S_3} = \eta_{f_2}|_{S_3}$.

The morphisms $\eta_{f_i}$ determine the sequences of $R$-linear maps $r_{f_i,n}\from A_{H_i}^\prime\to A_{G_i}$.  Recall that $S_3=\Hom_{R,\cont}(A_{H_2}^\prime,A_{G_1})$.  The fact that the $\eta_{f_i}$ coincide upon base change to $S_3$ means exactly that the desired diagram commutes.
\end{proof}

\begin{lemma} Let $f\from \MM(G)\to \MM(H)$ be a morphism of crystals over $R$.
\begin{altenumerate}
\item[{\rm (i)}] Let $f\times f\from \MM(G\times G)\to \MM(H\times H)$ be the obvious morphism. We have $r_{f\times f,n}=r_{f,n}\otimes r_{f,n}$ as $R$-linear maps $A_H^\prime \hat{\otimes}A_H^\prime \to A_G\otimes A_G$.
\item[{\rm (ii)}] We have $(r_{f,n})^\vee = r_{f^\vee,n}$ as $R$-linear maps $A_{G^\vee}^\prime\to A_{H^\vee}$.
\end{altenumerate}
\end{lemma}

\begin{proof} Easy and left to the reader.
\end{proof}

Finally, we can show that $r_{f,n}\from A_H^\prime\to A_G$ is a morphism of Hopf algebras. Consider the diagonal maps $\Delta_G\from G\to G\times G$ and $\Delta_H\from H\to H\times H$. Then $\MM(\Delta_G)\from \MM(G)\to \MM(G\times G)=\MM(G)\otimes\MM(G)$ is also the diagonal map, and similarly for $\MM(\Delta_H)$.  The diagram
\[\xymatrix{
\MM(G) \ar[rr]^{\MM(\Delta_G)} \ar[d]_{f} && \MM(G\times G) \ar[d]_{f\times f}\\
\MM(H) \ar[rr]_{\MM(\Delta_H)}  && \MM(H\times H)
}\]
commutes. By Lemma \ref{psiGdiagram}, the diagram
\[\xymatrix{
A_H^\prime\hat{\otimes} A_H^\prime \ar[d]_{r_{f,n}\hat{\otimes} r_{f,n}} \ar[rr]^{\Delta_H^\prime} && A_H^\prime \ar[d]^{r_{f,n}} \\
A_G\otimes A_G \ar[rr]_{\Delta_G^\ast} && A_G
}\]
also commutes. But $\Delta_G^\ast\from A_G\times A_G\to A_G$ is simply the multiplication map, and similarly for $\Delta_H^\prime$. Thus $r_{f,n}$ is an $R$-algebra homomorphism. By dualizing and using the same result for $\MM(H^\vee)\to \MM(G^\vee)$, one sees that $r_{f,n}$ is also compatible with comultiplication, as desired.
\end{proof}

We get that $r_{f,n}$ lies in
\[\begin{aligned}
\Hom_{R\text{-Hopf},\cont}(A_H^\prime,A_G)&=\varinjlim_m \Hom_{R\text{-Hopf}}(A_{H[p^m]},A_{G[p^m]})\\
&=\varinjlim_m \Hom(G[p^m],H[p^m])\ ,
\end{aligned}\]
where the last $\Hom$ is in the category of group schemes over $R$. Thus for $m>n$ large enough, $r_{f,n}$ is induced by a morphism $\psi_n^\prime\from G[p^m]\to H[p^m]$. Let us assume that $\eta_f\in \Hom_S(\Q_p/\Z_p,\mu_{p^\infty})$ is integral, which we can after multiplication by a power of $p$. Then the morphisms $r_{f,n}$ come from elements $s_n\in S$, $n\geq 0$, satisfying $s_{n+1}^p = s_n$, $s_0=1$. In particular, $s_n^{p^n}=1$. By Lemma \ref{HopfAlgLemma}, this means that $p^n\psi_n^\prime=0$, so that $\psi_n^\prime$ factors through a map $\psi_n\from G[p^n]\to H[p^n]$. Now the condition $s_{n+1}^p = s_n$ for all $n\geq 0$ means that the morphisms $\psi_n$ combine to a morphism $\psi\from G\to H$ of $p$-divisible groups.

In summary, we have constructed a map
\[
\Hom_R(\MM(G),\MM(H))[p^{-1}]\to \Hom_R(G,H)[p^{-1}]\ ,
\]
that is injective by Proposition \ref{Injectivity}. On the other hand, the composition
\[
\Hom_R(G,H)[p^{-1}]\to \Hom_R(\MM(G),\MM(H))[p^{-1}]\to \Hom_R(G,H)[p^{-1}]
\]
is the identity by construction, finishing the proof of Theorem \ref{fullyfaithful}.

\section{On $p$-divisible groups over $\OO_C$}

In this section, we fix a complete algebraically closed extension $C$ of $\Q_p$, and will classify $p$-divisible groups over $\OO_C$. Fix a map $\bar{\mathbb{F}}_p\rightarrow \OO_C/p$. Also, in this section we will write
\[
B_\cris^+ = B_\cris^+(\OO_C/p)\ ,
\]
which comes with a natural map $\Theta\from B_\cris^+\to C$.

\subsection{From $p$-divisible groups to vector bundles}

First, we will relate some of the constructions that occured so far to the theory of vector bundles over the Fargues-Fontaine curve. Let $P$ be the graded $\Q_p$-algebra
\[
P=\bigoplus_{d\geq 0} (B_\cris^+)^{\varphi=p^d}\ ,
\]
and let $X=\Proj P$.  Then $X$ is a curve, in the sense that it is a connected, separated, regular Noetherian scheme of dimension 1, cf.~\cite{FarguesFontaine}, Th\'eor\`eme 10.2. Moreover, $X$ is also complete in the sense that there is a homomorphism $\deg\from \Div(X)\to\Z$ which is surjective, nonnegative on effective divisors, and zero on principal divisors. From here, one defines the rank and degree of any coherent sheaf on $X$, and one gets the following result.

\begin{prop}\label{RkDeg} Let $f\from \mathscr{F}\to \mathscr{E}$ be an injective map of coherent sheaves on $X$ of the same rank and degree. Then $f$ is an isomorphism. $\hfill \Box$
\end{prop}

The curve $X$ has a natural point $\infty\in X$, coming from the map $\Theta\from B_\cris^+\to C$. We denote by $i_\infty: \{\infty\}\rightarrow X$ the inclusion.

Let $G_0$ be a $p$-divisible group over $\OO_C/p$, and let $M$ denote the covariant Dieudonn\'{e} module of $G_0$, which is a finite projective $A_\cris$-module. We define the associated vector bundle over the Fargues-Fontaine curve $X$, as the vector bundle $\mathscr{E}(G_0)$ associated to the graded $P$-module
\[
\bigoplus_{d\geq 0}(M[p^{-1}])^{\varphi=p^{d+1}}\ .
\]
Before going on, let us observe that our full faithfulness result over $\OO_C/p$ recovers the rational comparison isomorphism for $p$-divisible groups.

\begin{Cor} Let $G$ be a $p$-divisible group over $\OO_C$, and let $T=T(G)(\OO_C)$. The sequence
\[
0\to T[p^{-1}]\to \tilde{G}(\OO_C)\to \Lie G\otimes C\to 0
\]
is exact, and gets identified with the sequence
\[
0\to T[p^{-1}]\to (M[p^{-1}])^{\varphi = p}\to \Lie G\otimes C\to 0\ .
\]
\end{Cor}

\begin{proof} To check exactness on the right, use that $C$ is algebraically closed, to deduce that multiplication by $p$ on $G(\OO_C)$ is surjective. Then note that
\[
\tilde{G}(\OO_C) = \tilde{G}(\OO_C/p) = \Hom_{\OO_C/p}(\Q_p/\Z_p,G)[p^{-1}] = (M[p^{-1}])^{\varphi = p}\ .
\]
\end{proof}

\begin{rmk} Note that $\OO_C/p$ satisfies the hypothesis of Corollary \ref{Intfullyfaithful}, so that one also gets the integral version of the comparison theorem.
\end{rmk}

\begin{Theorem}\label{FactsFarguesFontaine}
\begin{altenumerate}
\item[{\rm (i)}] Let $G$ be a $p$-divisible group over $\OO_C$. Then there exists a $p$-divisible group $H$ over $\bar{\mathbb{F}}_p$ and a quasi-isogeny
\[
\rho\from H\otimes_{\bar{\mathbb{F}}_p} \OO_C/p\rightarrow G\otimes_{\OO_C} \OO_C/p\ .
\]
\item[{\rm (ii)}] The functor $G_0\mapsto \mathscr{E}(G_0)$ from $p$-divisible groups over $\OO_C/p$ up to isogeny to vector bundles over $X$ is fully faithful. The essential image is given by the vector bundles all of whose slopes are between $0$ and $1$.
\end{altenumerate}
\end{Theorem}

\begin{rmk} The first part of this theorem says that all $p$-divisible groups over $\OO_C$ are isotrivial in the sense defined by Fargues, \cite{FarguesGenestierLafforgue}. The same result is true when $C$ is only required to be a perfectoid field of characteristic $0$ with algebraically closed residue field. In fact, the same proof applies.
\end{rmk}

\begin{proof} As $\OO_C/p$ is f-semiperfect, Theorem \ref{fullyfaithful} reduces the result to the analogous statement for Dieudonn\'{e} modules over $B_\cris^+$. Then both results are due to Fargues--Fontaine: The first part is Theorem 7.24 in \cite{FarguesFontaineDurham}, and the full faithfulness in the second part is Theorem 7.15 combined with Proposition 7.17 in \cite{FarguesFontaineDurham}. By Theorem 7.14 in \cite{FarguesFontaineDurham}, any vector bundle on $X$ is isomorphic to the vector bundle associated to some isocrystal over $\bar{\FF}_p$. All those with slopes between $0$ and $1$ are clearly in the essential image; on the other hand, any vector bundle in the image is the vector bundle associated to some $p$-divisible group $H$ over $\bar{\FF}_p$ by part (i), and thus has slopes between $0$ and $1$, giving part (ii).
\end{proof}

We will also need the following statements, cf. also \cite{FarguesFontaineDurham}, Section 6.

\begin{prop}\label{PDivGroupGivesModVectBund} Let $G$ be a $p$-divisible group over $\OO_C$. Let $\mathscr{E} = \mathscr{E}(G_0)$ for $G_0 = G\otimes_{\OO_C} \OO_C/p$ and $\mathscr{F} = \OO_X\otimes_{\Z_p} T$ be associated vector bundles over $X$. Here and in the following, we write $T = T(G)(\OO_C)$.
\begin{altenumerate}
\item[{\rm (i)}] There is a natural exact sequence of coherent sheaves over $X$,
\[
0\to \mathscr{F}\to \mathscr{E}\to i_{\infty\ast} (\Lie G\otimes C)\to 0\ .
\]
The global sections of this map give the logarithm sequence
\[
0\to T[p^{-1}]\to \tilde{G}(\OO_C)\to \Lie G\otimes C\to 0\ .
\]

\item[{\rm (ii)}] Under the identification
\[
i_\infty^\ast \mathscr{E} = M(G)\otimes_{\OO_C} C\ ,
\]
the adjunction morphism $\mathscr{E}\to i_{\infty\ast} i_\infty^\ast \mathscr{E}$ induces on global sections the quasi-logarithm morphism
\[
\tilde{G}(\OO_C)\to M(G)\otimes C\ .
\]
When restricted to $T[\frac 1p]$, it induces a surjective map
\[
\alpha_G\from T\otimes C\to (\Lie G^\vee\otimes C)^\vee\subset M(G)\otimes C\ .
\]
\item[{\rm (iii)}] The sequence
\[
0\to \Lie G\otimes C(1)\buildrel{\alpha_{G^\vee}^\vee(1)}\over\longrightarrow T\otimes C\buildrel\alpha_G\over\longrightarrow (\Lie G^\vee\otimes C)^\vee\to 0
\]
is exact.
\end{altenumerate}
\end{prop}

In other words, a $p$-divisible group of height $h$ over $\OO_C$ gives rise to a modification $\mathscr{E}$ of the trivial vector bundle $\mathscr{F} = \mathcal{O}_X^h$ along $\infty\in X$.

We note that $\alpha_G$ can be described more elementary as follows. One has the map
\[
\Q_p/\Z_p\otimes_{\Z_p} T\to G\ ,
\]
which is defined as a map of $p$-divisible groups over $\OO_C$, and induces by Cartier duality a map
\[
G^\vee\to \mu_{p^\infty}\otimes_{\Z_p} T^\vee\ .
\]
On Lie algebras, this gives a map $\Lie G^\vee\otimes C\to T^\vee\otimes C$. Dualizing again, we get the map
\[
T\otimes C\buildrel\alpha_G\over\to (\Lie G^\vee\otimes C)^\vee\ .
\]

\begin{proof} First of all, all identifications of maps are immediate from our preparations on $p$-divisible groups. In particular, the natural map $\mathscr{F}\to \mathscr{E}$ induced from $T\to M^{\varphi = p}$ factors through
\[
\mathscr{F}^\prime = \ker(\mathscr{E}\to i_{\infty\ast} (\Lie G\otimes C))\ .
\]
Moreover, the induced map $\mathscr{F}\to \mathscr{F}^\prime$ is injective. Indeed, it suffices to check this for $\mathscr{F}\to \mathscr{E}$. Consider the diagram
\[\xymatrix{
&\mathscr{F}\ar[r]\ar@{=}[d]&\mathscr{E}\ar[d]\ar[r]&i_{\infty\ast} (\Lie G\otimes C)\ar[d]&\\
0\ar[r]&\mathscr{F}\ar[r]&\mathscr{E}^\prime\ar[r]&i_{\infty\ast} (T\otimes C(-1))\ar[r]&0
}\]
On the lower line, $\mathscr{E}^\prime = T\otimes_{\Z_p} \OO_X(1)$, and the sequence comes from the $p$-divisible group $T(-1)\otimes_{\Z_p} \mu_{p^\infty}$. By a direct computation for $\mu_{p^\infty}$, the lower line is exact. In particular, $\mathscr{F}\injects \mathscr{E}$. But $\mathscr{F}$ and $\mathscr{F}^\prime$ are vector bundles of the same rank $h$ and degree $0$, thus $\mathscr{F}\cong \mathscr{F}^\prime$ by Proposition \ref{RkDeg}.

Part (iii) is taken from \cite{FarguesGenestierLafforgue}, Appendix C to Chapter 2. One could also reprove it using the machinery employed here. In fact, the surjectivity of $\alpha_G$ follows directly from part (i) after $i_\infty^\ast$. For part (iii), it only remains to see that the composite $\alpha_G\circ \alpha^\vee_{G^\vee}(1)$ is $0$. We have a composite map
\[
f_G\from \mathscr{E}_{G^\vee}^\vee\to \mathscr{F}_{G^\vee}^\vee\cong \mathscr{F}_G\to \mathscr{E}_G\ ,
\]
where we use subscripts to denote the $p$-divisible group with respect to which we construct the vector bundles. Also, we use a trivialization $\Z_p(1)\cong \Z_p$ to identify the vector bundles $\mathscr{F}_{G^\vee}^\vee$ and $\mathscr{F}_G$. We have to see that $i_\infty^\ast f_G$ is $0$. For this, it is enough to show that
\[
f_G\from \mathscr{E}_{G^\vee}^\vee\to \mathscr{E}_G\cong \mathscr{E}_{G^\vee}^\vee\otimes_{\OO_X} \OO_X(1)
\]
is multiplication by $t\in H^0(X,\OO_X(1))$, where $t=\log([\epsilon])$ comes from the element $\epsilon\in \OO_C^\flat$ induced from the chosen trivialization $\Z_p(1)\cong \Z_p$. For this, one can replace $G$ by $T\otimes_{\Z_p} \Q_p/\Z_p$, where it is immediate.
\end{proof}

\subsection{Classification of $p$-divisible groups}

The main theorem of this section is the following result.

\begin{Theorem}\label{ClassificationOverOC} There is an equivalence of categories between the category of $p$-divisible groups over $\OO_C$ and the category of free $\Z_p$-modules $T$ of finite rank together with a $C$-subvectorspace $W$ of $T\otimes C(-1)$.
\end{Theorem}

The functor is defined in the following way. To $G$, one associates the Tate module $T=T(G)$, together with $W=\Lie G\otimes C$, embedded into $T\otimes C$ via
\[
\alpha_{G^\vee}^\vee: \Lie G\otimes C\to T\otimes C(-1)\ .
\]

\begin{proof} We start by proving that the functor is fully faithful. This was already observed by Fargues, \cite{FarguesPDivGroups}. Let
\[
G^\prime = T(G)(-1)\otimes_{\Z_p} \mu_{p^\infty}
\]
be a multiplicative $p$-divisible group equipped with a map $G\to G^\prime$ inducing an isomorphism $T(G)\cong T(G^\prime)$. We also get a map $\Lie G\to \Lie G^\prime = T(G)(-1)\otimes \OO_C$. It is easy to check that this map is identified with $\alpha_{G^\vee}^\vee$ (after inverting $p$).

The resulting adic spaces fit into a commutative diagram
\[\xymatrix{
0\ar[r]&G^\ad_\eta[p^\infty]\ar[d]^{\cong}\ar[r]&G^\ad_\eta\ar[d]\ar[r]&\Lie G\otimes \Ga\ar[d]\\
0\ar[r]&G^{\prime\ad}_\eta[p^\infty]\ar[r]&G^{\prime\ad}_\eta\ar[r]&\Lie G^\prime\otimes \Ga
}\]
We claim that the right square is cartesian. Indeed, if $x\in G^{\prime\ad}_\eta(R,R^+)$ maps to $\Lie G\otimes \Ga$, then for $n$ sufficiently large, $p^nx\in G^\ad_\eta(R,R^+)$, as on a small neighborhood of $0$, $\log_{G^\prime}$ is an isomorphism. One gets $x\in G^\ad_\eta(R,R^+)$, as
\[\xymatrix{
G^\ad_\eta\ar[r]^{p^n}\ar[d] & G^\ad_\eta\ar[d]\\
G^{\prime\ad}_\eta\ar[r]^{p^n} & G^{\prime\ad}_\eta
}\]
is cartesian. It follows that we can reconstruct $G^\ad_\eta$ from $(T,W) = (T(G),\Lie G\otimes C)$. But then we can also reconstruct
\[
G = \bigsqcup_{Y\subset G^\ad_\eta} \Spf H^0(Y,\OO_Y^+)\ ,
\]
where $Y$ runs through the connected components of $G^\ad_\eta$.

It remains to prove that the functor is essentially surjective. For this, let us first assume that $C$ is spherically complete, and that the norm map $|\cdot|: C\to \R_{\geq 0}$ is surjective.

So, assume given $(T,W)$. We define $G^\prime$ as before, and then $G^\ad_\eta\subset G^{\prime\ad}_\eta$ as the fibre product. We have to see that
\[
\bigsqcup_{Y\subset G^\ad_\eta} \Spf H^0(Y,\OO_Y^+)
\]
defines a $p$-divisible group. One checks directly that (the rigid-analytic space corresponding to) $G^\ad_\eta$ is a $p$-divisible rigid-analytic group in the sense of \cite{FarguesPDivGroups}:

\begin{defn} A $p$-divisible rigid-analytic group over $C$ is a commutative smooth rigid-analytic group $G$ over $C$ such that $p\from G\to G$ is finite locally free and faithfully flat, and $p\from G\to G$ is topologically nilpotent.
\end{defn}

A basic fact is that these decompose into connected components.

\begin{lemma}{\cite[Corollaire 9]{FarguesPDivGroups}} Let $G$ be a $p$-divisible rigid-analytic group over $C$. Then $G=G^0\times \pi_0(G)$, where $G^0$ is a connected $p$-divisible rigid-analytic group over $C$, and $\pi_0(G)\cong (\Q_p/\Z_p)^r$ for some integer $r$.
\end{lemma}

By Th\'eor\`eme 15 of \cite{FarguesPDivGroups}, it is enough to show that the connected component $(G^\ad_\eta)^0$ of the identity is isomorphic to a $d$-dimensional open unit ball $\mathring{\B}^d$: That theorem verifies that if $(G^\ad_\eta)^0\cong \mathring{\B}^d$, then the formal group induced on
\[
\Spf H^0((G^\ad_\eta)^0,\OO_{(G^\ad_\eta)^0}^+)\cong \Spf \OO_C\powerseries{T_1,\ldots,T_d}
\]
is $p$-divisible, i.e. the kernel of multiplication by $p$ is finite and locally free. Moreover, we know the following.

\begin{prop}{\cite[Proposition 14, Lemma 13]{FarguesPDivGroups}} One can write $(G^\ad_\eta)^0$ as an increasing union of connected affinoid subgroups $U_n\subset (G^\ad_\eta)^0$, such that $U_n\cong \B^d$ for all $n\geq 0$.
\end{prop}

\begin{proof} The subgroups $U_n$ can be defined as the (connected component of) the intersection of $(G^\ad_\eta)^0$ with increasing closed balls in $G^{\prime\ad}_\eta\cong \mathring{\B}^h$.

For the convenience of the reader, we recall the crucial argument in showing that $U_n\cong \B^d$. Write $U_n = \Spa(A,A^+)$. Using \cite{BoschGuentzerRemmert}, we know that $A^+$ is topologically of finite type over $\OO_C$, and it has the structure of a topological Hopf algebra. Let $\overline{A} = A^+\otimes_{\OO_C} k$, where $k$ is the residue field of $\OO_C$. Then $\Spec \overline{A}$ is a connected commutative affine reduced group scheme over $k$ such that multiplication by $p$ is nilpotent. As it is reduced and a group scheme, it is smooth; it follows that it is an extension of a torus by a unipotent group. As multiplication by $p$ is nilpotent, the torus part has to be trivial. It follows that $\overline{A}\cong k[T_1,\ldots,T_d]$, and hence $A^+\cong \OO_C\tatealgebra{T_1,\ldots,T_d}$, and thus $U_n\cong \B^d$.
\end{proof}

Now we finish the proof by using the following proposition (which is the only place where we use our assumptions on $C$).

\begin{prop} Assume that $K$ is a spherically complete nonarchimedean field for which the norm map $|\cdot|: K\to \R_{\geq 0}$ is surjective. Let $X$ be an adic space over $\Spa(K,\OO_K)$ with a point $0\in X$ which can be written as an increasing union $X=\bigcup X_n$ of quasicompact open subsets $0\in X_0\subset X_1\subset \ldots$. Moreover, we assume the following conditions.
\begin{altenumerate}
\item[{\rm (i)}] For all $n\geq 0$, the inclusion $X_n\subset X_{n+1}$ is a strict open embedding.
\item[{\rm (ii)}] For each $n\geq 0$, there is an isomorphism $X_n\cong \mathbb{B}^d$, carrying $0\in X_n$ to $0\in \mathbb{B}^d$. Let
\[
\OO_C\tatealgebra{T_{n,1},\ldots,T_{n,d}}\cong A_n := H^0(X_n,\OO_{X_n}^+)\ .
\]
\item[{\rm (iii)}] Let $I_n\subset A_n$ be the kernel of evaluation at $0$, and let $M_n = I_n/I_n^2\cong \OO_K^d$. The transition maps $M_{n+1}\to M_n$ are injective, and induce isomorphisms $M_{n+1}\otimes K\cong M_n\otimes K$. We require that $M=\bigcap_n M_n$ satisfies $M\otimes K = M_n\otimes K$ for one (and hence every) $n$.
\end{altenumerate}

Then $X\cong \mathring{\B}^d$.
\end{prop}

\begin{rmk} Condition (iii) says intuitively that the radii of the balls stay bounded, and follows for $X=(G^\ad_\eta)^0$ directly from the comparison with $G^{\prime\ad}_\eta\cong \mathring{\mathbb{B}}^h$.
\end{rmk}

\begin{proof} The key ingredient is the following lemma.

\begin{lemma}\label{SphCompleteLemma} Let $V$ be a finite-dimensional $K$-vector space, and let $\Lambda_0\supset \Lambda_1\supset\ldots$ be a sequence of $\OO_K$-lattices in $V$. Assume that $\Lambda = \bigcap_n \Lambda_n$ satisfies $\Lambda\otimes K = V$. Then $\Lambda$ is an $\OO_K$-lattice, and $R^1\varprojlim \Lambda_n = 0$.
\end{lemma}

\begin{proof} Filtering $V$ by subvectorspaces, one reduces to the case $V=K$. In that case, $\Lambda$ is a fractional ideal of $K$, and as the norm map is surjective, we can assume that either $\Lambda=\OO_K$ or $\Lambda=\gothm$. But if $\Lambda=\gothm$, then $\gothm\subset \Lambda_n$ for all $n$, hence $\OO_K\subset \Lambda_n$ as $\Lambda_n$ is a principal ideal. It follows that $\OO_K\subset \bigcap_n \Lambda_n$, which is a contradiction. Therefore $\Lambda=\OO_K$ is an $\OO_K$-lattice. It is easy to check that as $K$ is spherically complete, $R^1\varprojlim \Lambda_n = 0$.
\end{proof}

In the situation of the proposition, it follows that $M\cong \OO_K^d$. Similarly, one deduces that $\varprojlim I_n^k/I_n^{k+1}\cong \Sym^k M$ for all $k\geq 0$ with vanishing $R^1\varprojlim$. By induction, we get $\varprojlim A_n/I_n^k\cong \Sym^{\bullet} M / I^k$ for all $k$ with vanishing $R^1\varprojlim$, where $I\subset \Sym^{\bullet} M$ denotes the augmentation ideal. Therefore $\varprojlim_{n,k} A_n/I_n^k\cong \OO_C\powerseries{T_1,\ldots,T_d}$. Let $\hat{A}_n$ denote the $I_n$-adic completion of $A_n$. To conclude that
\[
H^0(X,\OO_X^+) = \OO_C\powerseries{T_1,\ldots,T_d}\ ,
\]
it remains to show that $\varprojlim_n \hat{A}_n = \varprojlim A_n$. But as $X_n\subset X_{n+1}$ is a strict closed embedding, the map $A_{n+1}\rightarrow A_n$ factors over $\hat{A}_{n+1}$, giving the result.
\end{proof}

This finishes the proof of Theorem \ref{ClassificationOverOC} for $C$ spherically complete, with surjective norm map. The general case will be deduced later.
\end{proof}

Let us add the following proposition.

\begin{prop}\begin{altenumerate}
\item[{\rm (i)}] Let $G$ be a $p$-divisible group over $\OO_C$ corresponding to $(T,W)$. Then the dual $p$-divisible group $G^\vee$ corresponds to $(T^\vee(1),W^\bot)$, where $T^\vee(1)$ denotes the Tate twist by one of the dual $\Z_p$-module, and $W^\bot\subset T^\vee\otimes C$ denotes the orthogonal complement of $W\subset T\otimes C(-1)$.
\item[{\rm (ii)}] Let
\[
0\to G_1\to G_2\to G_3\to 0
\]
be an exact sequences of $p$-divisible groups over $\OO_C$. Let $(T_i,W_i)$ for $i=1,2,3$, be the associated linear algebra data. Then
\[
0\to T_1\to T_2\to T_3\to 0
\]
and
\[
0\to W_1\to W_2\to W_3\to 0
\]
are exact.
\end{altenumerate}
\end{prop}

\begin{rmk} The converse to (ii) is false. Indeed, let $G_1$ be $\Q_p/\Z_p$, let $G_2$ be connected of dimension $1$ and height $2$, and let $G_3$ be $\mu_{p^\infty}$. Then there are maps $G_1\to G_2\to G_3$ whose composition is $0$, and such that on Tate modules $0\to T_1\to T_2\to T_3\to 0$ and Lie algebras $0\to W_1\to W_2\to W_3\to 0$, one gets exact sequences. However, the sequence $0\to G_1\to G_2\to G_3\to 0$ is not exact, as one checks easily on the special fibre.
\end{rmk}

\begin{proof} For (i), use the Hodge--Tate sequence. Part (ii) is clear.
\end{proof}

\section{Rapoport--Zink spaces}

\subsection{The space of level $0$, and the period morphism}

Let $H$ be a $p$-divisible group over a perfect field $k$ of characteristic $p$, of height $h$ and dimension $d$.

\begin{defn} \label{deformation} Let $R\in \Nilp_{W(k)}$. A {\em deformation} of $H$ to $R$ is a pair $(G,\rho)$, where $G$ is a $p$-divisible group over $R$ and
\[
\rho\from H\otimes_k R/p \to G\otimes_R R/p
\]
is a quasi-isogeny. Let $\Def_H$ be the associated functor on $\Nilp_{W(k)}$, taking $R$ to the set of isomorphism classes of deformations $(G,\rho)$ of $H$ to $R$.
\end{defn}

Recall the following theorem of Rapoport--Zink, \cite{RZ}.

\begin{Theorem} The functor $\Def_H$ is representable by a formal scheme $\mathcal{M}$ over $\Spf W(k)$, which locally admits a finitely generated ideal of definition. Moreover, every irreducible component of the reduced subscheme is proper.
\end{Theorem}

Let $M(H)$ denote the Dieudonn\'{e} module of $H$, a free $W(k)$-module of rank $h$. Assume now that $R$ is a $p$-adically complete $W(k)$-algebra equipped with the $p$-adic topology. Then via Grothendieck--Messing theory, $(G,\rho)$ gives rise to a surjection of locally free $R[\frac 1p]$-modules $M(H)\otimes R[\frac 1p]\to \Lie G[\frac 1p]$, which depends on $(G,\rho)$ only up to isogeny. We get the induced period map
\[
\pi\from \mathcal{M}^\ad_\eta\to \Flag\ ,
\]
which we call the crystalline period map. Here, $\Flag$ is the flag variety parametrizing $d$-dimensional quotients of the $h$-dimensional $W(k)[\frac 1p]$-vector space $M(H)[\frac 1p]$; we consider $\Flag$ as an adic space over $\Spa(W(k)[\frac 1p],W(k))$.

Moreover, we will consider the following variant of $\Def_H$. Consider the functor $\Def_H^\isog$ that associates to a $p$-adically complete flat $W(k)$-algebra $R$ equipped with the $p$-adic topology the set of deformations $(G,\rho)$ of $H$ to $R$, modulo quasi-isogeny over $R$. Observe that giving a quasi-isogeny over $R$ is strictly stronger than giving a compatible system of quasi-isogenies over the quotients $R/p^n$, so that $\Def_H^\isog$ cannot be defined on $\Nilp_{W(k)}$. Using the usual procedure, one gets a sheaf $(\Def_H^\isog)^\ad_\eta$ on the category of complete affinoid $(W(k)[\frac 1p],W(k))$-algebras. We note that $\pi$ factors over a map, still denoted $\pi$,
\[
\pi\from (\Def_H^\isog)^\ad_\eta\to \Flag\ .
\]

Our first goal is to prove the following result, which is essentially contained in \cite{deJong}.

\begin{prop} The sheaf $(\Def_H^\isog)^\ad_\eta$ is representable by an adic space, which is identified with an open subspace $U\subset \Flag$ under $\pi$.
\end{prop}

We remark that if $d=1$, i.e. we are working in the Lubin-Tate case, then $(\Def_H^\isog)^\ad_\eta = \Flag$, the map $\pi$ is called the Gross -- Hopkins period map, and the result follows easily from the theory in \cite{GrossHopkins}.

\begin{proof} Let $U\subset \Flag$ be the image of $\pi\from \mathcal{M}^\ad_\eta\to \Flag$. As $\pi$ is an \'{e}tale morphism of adic spaces locally of finite type, $U$ is open. We claim that
\[
\pi\from (\Def_H^\isog)^\ad_\eta\to U
\]
is an isomorphism. First we check that for any complete affinoid $(W(k)[\frac 1p],W(k))$-algebra $(R,R^+)$, the map
\[
(\Def_H^\isog)^\ad_\eta(R,R^+)\to U(R,R^+)
\]
is injective. So, let two elements of $(\Def_H^\isog)^\ad_\eta(R,R^+)$ be given which map to the same element of $U(R,R^+)$. Locally, they are given by $p$-divisible groups $G_i$, $i=1,2$, over an open and bounded subring $R_0\subset R^+$, and quasi-isogenies
\[
\rho_i\from H\otimes_{\bar{\mathbb{F}}_p} R_0/p\to G_i\otimes_{R_0} R_0/p\ .
\]
In particular, we get a quasi-isogeny from $G_1\otimes_{R_0} R_0/p$ to $G_2\otimes_{R_0} R_0/p$. By Grothendieck--Messing theory, the obstruction to lifting this to a quasi-isogeny from $G_1$ to $G_2$ over $R_0$ is the compatibility with the Hodge filtration. By assumption, we have this compatibility, so that the quasi-isogeny lifts, thereby showing that the two elements of $(\Def_H^\isog)^\ad_\eta$ agree.

To show surjectivity, it is enough to prove the following lemma.

\begin{lemma} The map $\pi\from \mathcal{M}^\ad_\eta\to U$ admits local sections.
\end{lemma}

\begin{proof} It is enough to prove that for any complete nonarchimedean extension $K$ of $\mathbb{Q}_p$ with an open and bounded valuation subring $K^+\subset K$, the map
\[
\mathcal{M}^\ad_\eta(K,K^+)\to U(K,K^+)
\]
is surjective. Indeed, this shows that any point $u\in U$ with values in $(K,K^+)$ can be lifted to a $(K,K^+)$-valued point of $\mathcal{M}^\ad_\eta$. As $\pi$ is \'{e}tale, one checks that this section extends to a small neighborhood of $u$, cf. proof of Lemma 2.2.8 of \cite{HuberBook}.

However, all irreducible components of the reduced subscheme of $\mathcal{M}$ are proper, so that $\mathcal{M}^\ad_\eta$ is partially proper. It follows that we may assume that $K^+ = \OO_K$. Let $x\in U(K,\OO_K)$. As $x$ is in the image of $\pi$, there is some finite Galois extension $L$ of $K$ and a point $\tilde{x}\in \mathcal{M}^\ad_\eta(L,\OO_L)$ lifting $x$. This is given by a $p$-divisible group $G$ over $\OO_L$, and a quasi-isogeny $\rho$ over $\OO_L/p$. On the Tate module $T(G)=T(G)(\OO_{\hat{\bar{K}}})$, we get an action of $\Gal(\bar{K}/L)$. Moreover, on the rational Tate module $V(G) = T(G)[\frac 1p]$, we get an action of $\Gal(\bar{K}/K)$. Indeed, for any $\sigma\in \Gal(\bar{K}/K)$, the point $\sigma(\tilde{x})$ projects to $x$ via $\pi$, meaning that there is a unique quasi-isogeny between $\sigma^\ast(G)$ and $G$ compatible with the isogenies $\sigma^\ast(\rho)$ and $\rho$ on $\OO_L/p$. This induces a canonical identification between $\sigma^\ast(V(G))$ and $V(G)$, whereby we get the desired action of $\Gal(\bar{K}/K)$ on $V(G)$.

But now the profinite group $\Gal(\bar{K}/K)$ acts on the finite-dimensional $\Q_p$-vector space $V(G)$. It follows that there is an invariant $\Z_p$-lattice $T\subset V(G)$. This corresponds to a $p$-divisible group $G^\prime$ over $\OO_L$, quasi-isogenous to $G$, with corresponding quasi-isogeny $\rho^\prime$ modulo $p$. As $T$ is fixed by $\Gal(\bar{K}/K)$, it follows that $(G^\prime,\rho^\prime)$ descends to $\OO_K$, and thus defines an element of $\mathcal{M}(K,\OO_K)$, projecting to $x$, as desired.
\end{proof}
\end{proof}

\subsection{The image of the period morphism}

\begin{Theorem}\label{ImagePeriodMorphism} Fix a $p$-divisible group $H$ over a perfect field $k$ of dimension $d$ and height $h$, and let $\mathcal{M}$ be the associated Rapoport--Zink space, with period map
\[
\pi\from \mathcal{M}^\ad_\eta\to \Flag\ .
\]
Let $C$ be an algebraically closed complete extension of $W(k)[p^{-1}]$, and let $x\in \Flag(C,\OO_C)$ be a point corresponding to a $d$-dimensional quotient $M(H)\otimes C\to W$. Let $\mathscr{E} = \mathscr{E}(H)$ be the vector bundle over $X$ associated to $H$, and consider the modification of vector bundles
\[
0\to \mathscr{F}\to \mathscr{E}\to i_{\infty\ast} W\to 0
\]
corresponding to the quotient
\[
\mathscr{E}\to i_{\infty\ast} i_\infty^\ast \mathscr{E} = i_{\infty\ast} (M(H)\otimes C)\to i_{\infty\ast} W\ .
\]
Then $x$ is in the image of $\pi$ if and only if $\mathscr{F}$ is isomorphic to $\OO_X^h$.
\end{Theorem}

\begin{proof} It is enough to prove this in the case that $C$ is spherically complete with surjective norm map, as $\pi$ is locally of finite type. Indeed, it is clear that $\mathscr{F}$ is trivial if $x$ is in the image of $\pi$, by Proposition \ref{PDivGroupGivesModVectBund}. For the converse, if $\mathscr{F}$ is trivial over $C$, it is still trivial over a big extension $C^\prime$ of $C$; if the result is true over $C^\prime$, we find that the fibre $\pi^{-1}(x)$ is an adic space locally of finite type over $\Spa(C,\OO_C)$, which has a $(C^\prime,\OO_{C^\prime})$-valued point, and in particular is nonempty. It follows that it has a $(C,\OO_C)$-valued point.

So, assume that $\mathscr{F}$ is trivial and choose a $\Z_p$-lattice $T$ in $H^0(X,\mathscr{F})\cong \Q_p^h$. Then the modification of vector bundles gives an injection $W\to T\otimes C(-1)$, and hence a $p$-divisible group $G$ over $\OO_C$, such that the modification
\[
0\to \mathscr{F}\to \mathscr{E}\to i_{\infty\ast} W\to 0
\]
of vector bundles is the one induced from $G$ via Proposition \ref{PDivGroupGivesModVectBund}. By Theorem \ref{FactsFarguesFontaine} (ii), the corresponding identification of $\mathscr{E}(G\otimes_{\OO_C} \OO_C/p)$ with $\mathscr{E}(H)=\mathscr{E}(H\otimes_{\bar{\mathbb{F}}_p} \OO_C/p)$ comes from a quasi-isogeny
\[
\rho\from H\otimes_{\bar{\mathbb{F}}_p} \OO_C/p\to G\otimes_{\OO_C} \OO_C/p\ ,
\]
giving a point $\tilde{x}\in \mathcal{M}(\OO_C)=\mathcal{M}^\ad_\eta(C,\OO_C)$ with $\pi(\tilde{x}) = x$.
\end{proof}

\begin{proof} {\it (of Theorem \ref{ClassificationOverOC} for general $C$.)} We can now show that the functor is essentially surjective in general. Take any $(T,W)$, and construct the associated modification of vector bundles
\[
0\to \mathscr{F}\to \mathscr{E}\to i_{\infty\ast} W\to 0
\]
over $X$, where $\mathscr{F} = T\otimes_{\Z_p} \OO_X$: Define $\mathscr{E}$ by the pullback diagram
\[\xymatrix{
0\ar[r]&\mathscr{F}\ar[r]\ar@{=}[d]&\mathscr{E}\ar@{^(->}[d]\ar[r]&i_{\infty\ast} W\ar@{^(->}[d]\ar[r]&0\\
0\ar[r]&\mathscr{F}\ar[r]&\mathscr{E}^\prime\ar[r]&i_{\infty\ast} (T\otimes C(-1))\ar[r]&0
}\]
On the lower line, $\mathscr{E}^\prime = T\otimes_{\Z_p} \OO_X(1)$, and the sequence comes from the $p$-divisible group $T(-1)\otimes_{\Z_p} \mu_{p^\infty}$.

\begin{prop} There is a $p$-divisible group $H$ over $\bar{\mathbb{F}}_p$ such that $\mathscr{E}\cong \mathscr{E}(H)$.
\end{prop}

\begin{proof} By Theorem \ref{FactsFarguesFontaine}, we only have to prove that the slopes of $\mathscr{E}$ are between $0$ and $1$. Assume that $\mathscr{E}$ contains an irreducible direct summand $\OO_X(\lambda)$. Then there is a nonzero morphism $\OO_X\to \OO_X(\lambda)$, as otherwise, $\OO_X(\lambda)$ would inject into the cokernel of $\mathscr{F}\to \mathscr{E}$, which however is a skyscraper sheaf. This implies that $\lambda\geq 0$. On the other hand, as $\mathscr{E}$ injects into $\mathscr{E}^\prime$, it follows that there is a nonzero morphism $\OO_X(\lambda)\to \OO_X(1)$, so that $\lambda\leq 1$.
\end{proof}

Fixing an identification $\mathscr{E} = \mathscr{E}(H)$, the modification of vector bundles
\[
0\to \mathscr{F}\to \mathscr{E}(H)\to i_{\infty\ast} W\to 0
\]
corresponds to a quotient $M(H)\otimes C\to W$, i.e. a point $x\in \Flag$. As $\mathscr{F}$ is trivial by construction, we can use the preceding theorem to lift $x$ to a point $\tilde{x}$ in the corresponding Rapoport--Zink space, giving a $p$-divisible $G^\prime$ over $\OO_C$ with a quasi-isogeny
\[
\rho\from H\otimes_{\bar{\mathbb{F}}_p} \OO_C/p\to G^\prime\otimes_{\OO_C} \OO_C/p\ .
\]
The modification of vector bundles associated to $G^\prime$ is given by
\[
0\to \mathscr{F}\to \mathscr{E}(H)\to i_{\infty\ast} W\to 0\ .
\]
In particular, $T(G^\prime)\subset H^0(X,\mathscr{F})\cong T[\frac 1p]$ is a $\Z_p$-lattice. After replacing $G^\prime$ by a quasi-isogeneous $G$, we may assume that $T(G) = T$, which gives the desired $p$-divisible group.
\end{proof}

We recall that Faltings, \cite{FaltingsPeriodDomains}, proved Theorem \ref{ImagePeriodMorphism} by reduction to discretely valued fields, where he used a result of Breuil, \cite{Breuil}, and Kisin, \cite{Kisin}, to conclude. We can reverse the argument and deduce their result. We stress that contrary to the methods of Breuil and Kisin, which are based on integral $p$-adic Hodge theory, our argument only needs rational $p$-adic Hodge theory.

\begin{Cor} Let $K$ be a discretely valued complete nonarchimedean extension of $\Q_p$ with perfect residue field $k$. Then the category of $p$-divisible groups over $\OO_K$ is equivalent to the category of $\Gal(\bar{K}/K)$-stable $\Z_p$-lattices in crystalline representations of $\Gal(\bar{K}/K)$ with Hodge-Tate weights contained in $\{0,1\}$.
\end{Cor}

\begin{proof} By Tate's theorem, \cite{TatePDivGroups}, the functor from $p$-divisible groups to lattices in crystalline representations is fully faithful. Let $T$ be a $\Gal(\bar{K}/K)$-stable $\Z_p$-lattice in a crystalline representation of $\Gal(\bar{K}/K)$ with Hodge-Tate weights in $\{0,1\}$. We get an associated filtered isocrystal $M$ over $W(k)[p^{-1}]$, whose slopes are between $0$ and $1$ (by weak admissibility). Therefore, $M=M(H)$ for a $p$-divisible group $H$ over $k$. Let $\mathcal{M}_H$ be the associated Rapoport--Zink space, with period map
\[
\pi\from (\mathcal{M}_H)^\ad_\eta\to \Flag\ .
\]
The filtration on $M$ determines a point $x\in \Flag(K,\OO_K)$, which lies in the image of $\pi$ by Theorem \ref{ImagePeriodMorphism} (noting that $\mathscr{F}\cong T\otimes_{\Z_p} \OO_X$ canonically). One has a canonical identification between $\pi^{-1}(x)(C,\OO_C)$ and the set of $\Z_p$-lattices in $H^0(X,\mathscr{F})\cong T[p^{-1}]$, where $C = \hat{\bar{K}}$. As $T\subset T[p^{-1}]$ is stable under $\Gal(\bar{K}/K)$, this gives a point $\tilde{x}\in \pi^{-1}(x)(K,\OO_K)$, which in turn gives rise to the desired $p$-divisible group $G$ over $\OO_K$.
\end{proof}

\subsection{Rapoport--Zink spaces at infinite level}

In this section, we will define Rapoport--Zink spaces at infinite level, and explain how they are related to the classical Rapoport--Zink spaces at finite level. Fix a perfect field $k$, and a $p$-divisible group $H$ over $k$ of dimension $d$ and height $h$. Recall the definition of Rapoport--Zink spaces of finite level.

\begin{defn} Consider the functor $\mathcal{M}_n$ on complete affinoid $(W(k)[\frac 1p],W(k))$-algebras, sending $(R,R^+)$ to the set of triples $(G,\rho,\alpha)$, where $(G,\rho)\in \mathcal{M}^\ad_\eta(R,R^+)$, and
\[
\alpha\from (\Z_p/p^n\Z_p)^h\to G[p^n]^\ad_\eta(R,R^+)
\]
is a morphism of $\Z_p/p^n\Z_p$-modules such that for all $x=\Spa(K,K^+)\in \Spa(R,R^+)$, the induced map
\[
\alpha(x)\from (\Z_p/p^n\Z_p)^h\to G[p^n]^\ad_\eta(K,K^+)
\]
is an isomorphism.
\end{defn}

We recall the following result from \cite{RZ}.

\begin{prop} The functor $\mathcal{M}_n$ is representable by an adic space over $\Spa(W(k)[\frac 1p],W(k))$, which is an open and closed subset of the $h$-fold fibre product $(G[p^n]^\ad_\eta)^h$ of $G[p^n]^\ad_\eta$ over $\mathcal{M}^\ad_\eta$.
\end{prop}

\begin{defn} Consider the functor $\mathcal{M}_\infty$ on complete affinoid $(W(k)[\frac 1p],W(k))$-algebras, sending $(R,R^+)$ to the set of triples $(G,\rho,\alpha)$, where $(G,\rho)\in \mathcal{M}^\ad_\eta(R,R^+)$, and
\[
\alpha\from \Z_p^h\to T(G)^\ad_\eta(R,R^+)
\]
is a morphism of $\Z_p$-modules such that for all $x=\Spa(K,K^+)\in \Spa(R,R^+)$, the induced map
\[
\alpha(x)\from \Z_p^h\to T(G)^\ad_\eta(K,K^+)
\]
is an isomorphism.
\end{defn}

The main theorem of this section is the following.

\begin{Theorem}\label{MInftyExists} The functor $\mathcal{M}_\infty$ is representable by an adic space over $\Spa(W(k)[\frac 1p],W(k))$. Moreover, $\mathcal{M}_\infty$ is preperfectoid, and
\[
\mathcal{M}_\infty\sim \varprojlim_n \mathcal{M}_n\ ,
\]
where $\mathcal{M}_n$ denotes the Rapoport--Zink space of level $n$.
\end{Theorem}

\begin{proof} We start by proving that $\mathcal{M}_\infty$ is representable. In fact, we have the adic space $T(G)^\ad_\eta$ over $\mathcal{M}^\ad_\eta$, and a morphism
\[
\mathcal{M}_\infty\to (T(G)^\ad_\eta)^h\ ,
\]
where $(T(G)^\ad_\eta)^h$ denotes the $h$-fold fibre product of $T(G)^\ad_\eta$ over $\mathcal{M}^\ad_\eta$. We claim that there is a cartesian square
\[\xymatrix{
\mathcal{M}_\infty\ar[r]\ar[d]& (T(G)^\ad_\eta)^h\ar[d]\\
\mathcal{M}_1\ar[r]& (G[p]^\ad_\eta)^h\ .
}\]
In fact, recall that the morphism $\mathcal{M}_1\to (G[p]^\ad_\eta)^h$ is an open and closed embedding. In particular, one can check whether $\mathcal{M}_\infty\to (G[p]^\ad_\eta)^h$ factors over $\mathcal{M}_1$ on points, where it follows directly from the definition. Moreover, it is clear that if $\alpha(x)\mod p\from \F_p^h\to G[p]^\ad_\eta(K,K^+)$ is an isomorphism, then so is $\alpha(x)\from \Z_p^h\to T(G)^\ad_\eta(K,K^+)$. This proves that the square is cartesian, so that $\mathcal{M}_\infty$ is representable.

Moreover, $\mathcal{M}_\infty\sim \varprojlim_n \mathcal{M}_n$ follows from $(T(G)^\ad_\eta)^h\sim \varprojlim (G[p^n]^\ad_\eta)^h$, which is Proposition \ref{TateInvLimit}.

It remains to prove that $\mathcal{M}_\infty$ is preperfectoid. For this, we start with a different description of $\mathcal{M}_\infty$. The universal cover $\tilde{H}$ of $H$ lifts uniquely to $W(k)$, and we still denote by $\tilde{H}$ this lift. It gives rise to an adic space $\tilde{H}^\ad_\eta$ over $\Spa(W(k)[\frac 1p],W(k))$, equipped with a quasi-logarithm map
\[
\qlog\from \tilde{H}^\ad_\eta\to M(H)\otimes \Ga\ .
\]
Moreover, $\tilde{H}^\ad_\eta$ is preperfectoid, cf. Example \ref{ExmpBall} and Corollary \ref{tG-G0} in the connected case. The general case is easily deduced.

\begin{defn} \label{Minfprimedef} Consider the functor $\mathcal{M}_\infty^\prime$ on complete affinoid $(W(k)[\frac 1p],W(k))$-algebras, sending $(R,R^+)$ to the set of those $h$-tuples
\[
(s_1,\ldots,s_h)\in \tilde{H}^\ad_\eta(R,R^+)
\]
for which the following conditions are satisfied:
\begin{altenumerate}
\item[{\rm (i)}] The matrix $(\qlog(s_1),\ldots,\qlog(s_h))\in (M(H)\otimes R)^h$ is of rank exactly $h-d$; let $M(H)\otimes R\to W$ be the induced finite projective quotient of rank $d$.
\item[{\rm (ii)}] For all geometric points $x=\Spa(C,\OO_C)\to \Spa(R,R^+)$, the sequence
\[
0\to \Q_p^h\xrightarrow{(s_1,\ldots,s_h)} \tilde{H}^\ad_\eta(C,\OO_C)\to W\otimes_R C\to 0
\]
is exact.
\end{altenumerate}
\end{defn}

We remark that the condition on the rank is saying more precisely that all minors of size $\geq h-d+1$ vanish, and that the minors of size $h-d$ generate the unit ideal.

\begin{lemma}\label{AltDescription} There is a natural isomorphism of functors $\mathcal{M}_\infty\cong \mathcal{M}_\infty^\prime$. Moreover, $\mathcal{M}_\infty^\prime\subset (\tilde{H}^\ad_\eta)^h$ is a locally closed subfunctor.
\end{lemma}

\begin{proof} We start by constructing a map $\mathcal{M}_\infty\to \mathcal{M}_\infty^\prime$. So, let $(R,R^+)$ be a complete affinoid $(W(k)[\frac 1p],W(k))$-algebra, and $(G,\rho,\alpha)\in \mathcal{M}_\infty(R,R^+)$, where $(G,\rho)$ is defined over some open and bounded subring $R_0\subset R^+$. Using $\rho$, we get an identification $\tilde{H}_{R_0}\cong \tilde{G}$. Therefore, we get a map
\[
\alpha\from \Z_p^h\to \tilde{G}^\ad_\eta(R,R^+) = \tilde{H}^\ad_\eta(R,R^+)\ ,
\]
i.e. $h$ sections $s_1,\ldots,s_h\in \tilde{H}^\ad_\eta(R,R^+)$. We have to check that they satisfy conditions (i) and (ii) above. Condition (ii) is clearly satisfied, when $W=\Lie G\otimes R$. For condition (i), note that we have a commutative diagram
\[\xymatrix{
\tilde{H}^\ad_\eta(R,R^+)\ar[r]\ar[d]^{\qlog} & G^\ad_\eta(R,R^+)\ar[d]^{\log}\\
M(H)\otimes R\ar[r] & \Lie G\otimes R\ .
}\]
This implies that $\qlog(s_i)\in \ker(M(H)\otimes R\to \Lie G\otimes R)$, which is an $h-d$-dimensional projective $R$-submodule. Therefore, all minors of size $\geq h-d+1$ vanish. To check that the minors of size $h-d$ generate the unit ideal, it suffices to check for all $x=\Spa(K,K^+)\in \Spa(R,R^+)$, the sections $\qlog(s_1),\ldots,\qlog(s_h)\in M(H)\otimes K$ generate an $h-d$-dimensional subspace. We can assume that $K=C$ is algebraically closed, and that $K^+=\OO_C$. In that case, this follows from the surjectivity of
\[
T(G)\otimes C\to (\Lie G^\vee)^\vee\otimes C\ ,
\]
cf. Proposition \ref{PDivGroupGivesModVectBund} (iii). Note that these considerations also show that the quotient of $M(H)\otimes R$ by the submodule generated by $\qlog(s_1),\ldots,\qlog(s_h)$ is exactly $W=\Lie G\otimes R$.

Now we will construct the inverse functor $\mathcal{M}_\infty^\prime\to \mathcal{M}_\infty$. First, observe that $\mathcal{M}_\infty^\prime\to \Flag$ factors over $U\subset \Flag$, where $U$ denotes the image of the period morphism. As $U\subset \Flag$ is partially proper, it is enough to check this on geometric points $\Spa(C,\OO_C)$, i.e. those of rank $1$. We have the vector bundle $\mathscr{E} = \mathscr{E}(H)$ on the Fargues--Fontaine curve $X$. Let
\[
\mathscr{F} = \ker(\mathscr{E}\to i_{\infty\ast} W)\ .
\]
Then $H^0(X,\mathscr{F})\cong \Q_p^h$ by condition (ii). We claim that $\mathscr{F}\cong \OO_X^h$. As $\mathscr{F}$ is of rank $h$ and degree $0$, if $\mathscr{F}\not\cong \OO_X^h$, then there exists some $\lambda>0$ such that $\OO_X(\lambda)\hookrightarrow \mathscr{F}$. But $\dim_{\Q_p} H^0(X,\OO_X(\lambda))=\infty$ for $\lambda>0$, contradiction. By Theorem \ref{ImagePeriodMorphism}, this implies the desired factorization.

Next, consider the functor $\mathcal{F}$ defined similarly to $\mathcal{M}_\infty^\prime$, but with condition (ii) weakened to the condition that the quotient $W$ of $M(H)\otimes R$ defines a point of $U\subset \Flag$. Then $\mathcal{F}$ is clearly locally closed in $(\tilde{H}^\ad_\eta)^h$. Moreover, there is a natural morphism $\mathcal{F}\to U$, so that locally on $\Spa(R,R^+)\subset \mathcal{F}$, one has a deformation $(G^\prime,\rho^\prime)$ of $H$ to an open and bounded $W(k)$-subalgebra $R_0\subset R^+$. However, $(G^\prime,\rho^\prime)$ is only well-defined up to quasi-isogeny over $R_0$; i.e. we have a section
\[
(G^\prime,\rho^\prime)\in (\Def_H^\isog)^\ad_\eta(\mathcal{F})\ .
\]
Also observe that the exact sequence (all of whose objects depend only on $G^\prime$ up to quasi-isogeny)
\[
0\to V(G^\prime)^\ad_\eta\to \tilde{G}^{\prime\ad}_\eta\cong \tilde{H}^\ad_\eta\to \Lie G^\prime\otimes \Ga
\]
implies that $s_1,\ldots,s_h\in \tilde{H}^\ad_\eta(\mathcal{F})\cong \tilde{G}^{\prime\ad}_\eta(\mathcal{F})$ lie in the subset $V(G^\prime)^\ad_\eta(\mathcal{F})$.

Let $\mathcal{F}^\prime\subset \mathcal{F}$ be the subfunctor such that for all points $x=\Spa(K,K^+)\in \Spa(R,R^+)$, the map $\Q_p^h\to \tilde{H}^\ad_\eta(K,K^+)$ is injective. Clearly, $\mathcal{M}_\infty^\prime\to \mathcal{F}$ factors over $\mathcal{F}^\prime$.

\begin{prop} The map $\mathcal{F}^\prime\to \mathcal{F}$ is an open embedding. Moreover, there is a unique $p$-divisible group $G$ over $\mathcal{F}^\prime$ quasi-isogenous to $G^\prime$, such that for all $x=\Spa(K,K^+)\in \mathcal{F}^\prime$, the map
\[
T(G)^\ad_\eta(K,K^+)\to V(G^\prime)^\ad_\eta(K,K^+)
\]
is an isomorphism onto $\bigoplus_{i=1}^h \Z_p s_i(x)$.
\end{prop}

\begin{proof} Let $\Spa(R,R^+)\subset \mathcal{F}$ be an open subset, over which there exists $(G^\prime,\rho^\prime)$ over an open and bounded $W(k)$-subalgebra $R_0\subset R^+$. Moreover, let $x=\Spa(K,K^+)\in \Spa(R,R^+)$ be a point such that $s_1(x),\ldots,s_h(x)\in V(G^\prime)^\ad_\eta(K,K^+)$ are linearly independent. Without loss of generality, we assume that the $\Z_p$-lattice they generate contains $T(G^\prime)^\ad_\eta(K,K^+)$. Then this lattice inside $V(G^\prime)^\ad_\eta(K,K^+)$ corresponds to a quotient $G^\prime(x)\to G(x)$ of $G^\prime(x)$ by a finite locally free group over $K$, and hence over $K^+$, as $K^+$ is a valuation ring.

\begin{lemma} The category of finite locally free groups over $K^+$ is equivalent to the $2$-categorical direct limit of the categories of finite locally free group schemes over $\OO_X^+(U)$, $x\in U$.
\end{lemma}

\begin{proof} We have to show that the categories of finite locally free groups over $K^+$ and $\OO_{X,x}^+$ are equivalent. The same statement for $K$ and $\OO_{X,x}$ follows from the fact that the categories of finite \'{e}tale covers are equivalent, as was observed e.g. in \cite[Lemma 7.5 (i)]{Sch}. As
\[
\OO_{X,x}^+ = \{f\in \OO_{X,x}\mid f(x)\in K^+\}\ ,
\]
where $f(x)\in K$ is the evaluation of $f$, one readily deduces that the additional integral structure of giving a finite locally free group over $K^+$ (resp. $\OO_{X,x}^+$) inside a given finite locally group over $K$ (resp. $\OO_{X,x}$) is the same.
\end{proof}

It follows that after replacing $X$ by an open neighborhood of $x$, one can extend $G^\prime(x)\to G(x)$ to a quotient $G^\prime\to G$ over $X$. It follows that
\[
\bigoplus_{i=1}^h \Z_p s_i(x)\to V(G)^\ad_\eta(K,K^+)
\]
is an isomorphism onto $T(G)^\ad_\eta(K,K^+)$. We claim that the same is true on a small neighborhood of $x$. As $T(G)^\ad_\eta\subset V(G)^\ad_\eta$ is an open subset, one checks directly that after replacing $X$ by a small neighborhood of $x$, the sections $s_i\in V(G)^\ad_\eta(R,R^+)$ lie in $T(G)^\ad_\eta(R,R^+)\subset V(G)^\ad_\eta(R,R^+)$. One gets sections
\[
\bar{s}_1,\ldots,\bar{s}_h\in G[p]^\ad_\eta(R,R^+)\ ,
\]
and after passage to an open subset, one may assume that they are linearly independent over $\F_p$ (as the corresponding locus in $(G[p]^\ad_\eta)^h$ is open and closed). This shows that in an open neighborhood of $x$, $G$ is as desired. It is easy to check that a group $G$ with the desired property is unique: It suffices to do this locally. At points, it is clear, and then the previous lemma extends this uniqueness to a small neighborhood.

In particular, it follows that in this open neighborhood, $s_1,\ldots,s_h$ are $\Q_p$-linearly independent, whereby $\mathcal{F}^\prime\subset \mathcal{F}$ is open. Moreover, locally on $\mathcal{F}^\prime$, we have constructed the desired group $G$, and by uniqueness, they glue.
\end{proof}

The proposition constructs a deformation $(G,\rho)$ of $H$ to $\mathcal{F}^\prime$. Moreover, we have
\[
s_1,\ldots,s_h\in V(G)^\ad_\eta(\mathcal{F}^\prime)\ .
\]
In fact, they lie in the subset $T(G)^\ad_\eta(\mathcal{F}^\prime)$. As $T(G)^\ad_\eta\subset V(G)^\ad_\eta$ is an open embedding, this can be checked on points, where it is clear. This means that $s_1,\ldots,s_h$ give a map
\[
\alpha\from \Z_p^h\to T(G)^\ad_\eta(\mathcal{F}^\prime)\ ,
\]
which is clearly an isomorphism at every point. This gives the inverse functor $\mathcal{M}_\infty^\prime\to \mathcal{F}^\prime\to \mathcal{M}_\infty$, as desired. In fact, $\mathcal{M}_\infty^\prime\cong \mathcal{F}^\prime\cong \mathcal{M}_\infty$.
\end{proof}

Now we can finish the proof that $\mathcal{M}_\infty$ is preperfectoid. Indeed, $\mathcal{M}_\infty\subset (\tilde{H}^\ad_\eta)^h$ is locally closed, and $\tilde{H}^\ad_\eta$ is preperfectoid. Thus, by Proposition \ref{LocClosedPreperfectoid}, $\mathcal{M}_\infty$ is preperfectoid.
\end{proof}

Finally, we give a description of $\mathcal{M}_\infty$ purely in terms of $p$-adic Hodge theory, on the category of perfectoid algebras. So, let us fix a perfectoid field $K$ of characteristic $0$.

\begin{prop} The functor $\mathcal{M}_\infty$ on perfectoid affinoid $(K,\OO_K)$-algebras is the sheafification of the functor sending $(R,R^+)$ to the set of $h$-tuples
\[
p_1,\ldots,p_h\in (M(H)\otimes_{W(k)} B_\cris^+(R^+/p))^{\varphi=p}\ ,
\]
for which the following conditions are satisfied.
\begin{altenumerate}
\item[{\rm (i)}] The matrix $(\Theta(p_1),\ldots,\Theta(p_h))\in (M(H)\otimes R)^h$ is of rank exactly $h-d$; let $M(H)\otimes R\to W$ be the induced quotient.
\item[{\rm (ii)}] For all geometric points $x=\Spa(C,\OO_C)\to \Spa(R,R^+)$, the sequence
\[
0\to \Q_p^h\xrightarrow{(p_1(x),\ldots,p_h(x))}(M(H)\otimes_{W(k)} B_\cris^+(\OO_C/p))^{\varphi=p}\to W\otimes_R C\to 0
\]
is exact.
\end{altenumerate}
\end{prop}

\begin{proof} We have
\[
\tilde{H}^\ad_\eta(R,R^+) = \tilde{H}(R^+) = \tilde{H}(R^+/p) = (M(H)\otimes_{W(k)} B_\cris^+(R^+/p))^{\varphi=p}\ ,
\]
by Theorem \ref{fullyfaithful}. It follows that the datum of $p_1,\ldots,p_h$ is equivalent to the datum of $h$ sections
\[
s_1,\ldots,s_h\in \tilde{H}^\ad_\eta(R,R^+)\ ,
\]
and the conditions (i) and (ii) clearly correspond.
\end{proof}

\begin{Cor}\label{CPoints} Let $C$ be an algebraically closed complete nonarchimedean extension of $\Q_p$, and let $X$ be the associated Fargues--Fontaine curve. Let $\mathscr{F}=\OO_X^h$ and $\mathscr{E} = \mathscr{E}(H)$ be associated vector bundles on $X$. Then $\mathcal{M}_\infty(C,\OO_C)$ is given by the set of morphisms $f\from \mathscr{F}\to \mathscr{E}$ that give rise to a modification
\[
0\to \mathscr{F}\buildrel f\over \to \mathscr{E}\to i_{\infty\ast} W\to 0\ ,
\]
where $W$ is some $C$-vector space.
\end{Cor}

\subsection{The Lubin-Tate space at infinite level}
In this section, $H$ is a connected $p$-divisible group over $k$ of dimension 1 and height $h$.  We assume that $k$ is algebraically closed, so that $H$ only depends on $h$.  Let $\mathcal{M}_\infty^h$ be the associated infinite-level Rapoport-Zink space.   We show here that $\mathcal{M}_\infty^h$ is cut out from $(\tilde{H}^{\ad}_{\eta})^h$ by a single determinant condition.

Let us begin with the case of $h=1$, so that $H=\mu_{p^\infty,k}$.  Then $H$ lifts uniquely to the $p$-divisible group $\mu_{p^\infty}$ over $W(k)$.   We have $M(\mu_{p^\infty,k})=\Lie\mu_{p^\infty}=W(k)$.   The quasi-logarithm map reduces to the usual logarithm map, which sits inside the exact sequence
\[
\xymatrix{
0\ar[r]&
V(\mu_{p^\infty})_{\eta}^{\ad} \ar[r] &
\tilde{\mu}_{p^\infty,\eta}^{\ad} \ar[r]^{\log} &
\mathbf{G}_a
}
\]
of sheaves of $\Q_p$-vector spaces on $\CAff_{\Spa(W(k)[1/p],W(k))}$.  Now let $(R,R^+)$ be a complete affinoid $(W(k)[\frac{1}{p}],W(k))$-algebra.  By Lemma \ref{AltDescription}, $\mathcal{M}^1_\infty(R,R^+)$ is the set of sections $s\in \tilde{\mu}_{p^\infty}^{\ad}(R,R^+)$ which satisfy both $\log(s)=0$ and the condition that $s(x)\neq 0$ for all points $x=\Spa(K,K^+)\in\Spa(R,R^+)$.  This means that
\[
\mathcal{M}^1_\infty = V(\mu_{p^\infty})^{\ad}_{\eta}\backslash\set{0}\ .
\]
Either from this description, or from the description of $\mathcal{M}^1_\infty$ as classifying triples $(G,\rho,\alpha)$, it follows that
\[
\mathcal{M}^1_\infty = \bigsqcup_{n\in\Z} \Spa(L,\OO_L)\ ,
\]
where $L$ is the completion of $W(k)[\frac 1p](\mu_{p^\infty})$. Indeed, the decomposition by $n\in \Z$ comes from the height of $\rho$. Assume that $\rho$ is of height $0$; then $G=\mu_{p^\infty}$ canonically, by rigidity of multiplicative $p$-divisible groups. Now $\alpha\from \Z_p\to T\mu_{p^\infty}$ amounts to the choice of a compatible system of $p$-power roots of unity, whence the result.

Now return to the general case, so that $H$ has dimension 1 and height $h$.  We construct a determinant morphism from the $h$-fold product of $\tilde{H}^{\ad}_{\eta}$ into $\tilde{\mu}_{p^\infty,\eta}^{\ad}$. Let $\wedge^h M(H)$ be the top exterior power of the $W(k)$-module $M(H)$;  then $\wedge^h M(H)$ is free of rank $1$ with Frobenius slope $1$, so that $\wedge^h M(H)\isom M(\mu_{p^\infty,k})$.

 Write $\tilde{H}$ for the universal cover of $H$, considered over the base $W(k)$. By Corollary \ref{tG-G0}, $\tilde{H}^h$ is representable by a formal scheme $\Spf R$, where $R=W(k)\powerseries{X_1^{1/p^\infty},\dots,X_h^{1/p^\infty}}$. In particular $R/p$ is an inverse limit of f-semiperfect rings. The projection maps $\tilde{H}^h\to\tilde{H}$ give $h$ canonical elements $s_1,\dots,s_h\in \tilde{H}(R/p)$. Appealing to Theorem \ref{fullyfaithful}, the Dieudonn\'e module functor gives us isomorphisms
\[
\alpha_H\from \tilde{H}(S)\to (M(H)\otimes_{W(k)} B_{\cris}^+(S))^{\phi=1}
\]
and
\[
\alpha_{\mu_{p^\infty}}\from \tilde{\mu}_{p^\infty}(S)\to (M(\mu_{p^\infty,k})\otimes B_{\cris}^+(S))^{\phi=1}
\]
for any f-semiperfect ring $S$. We get an element $\alpha_{\mu_{p^\infty}}^{-1}(\alpha_H(s_1)\wedge\dots\wedge\alpha_H(s_h))$ of $\tilde{\mu}_{p^\infty}(S)$. As $R/p$ is an inverse limit of f-semiperfect rings, we get a morphism of formal schemes $\det\from \tilde{H}^h\to \tilde{\mu}_{p^\infty}$ over $\Spec k$, and then by rigidity also over $\Spf W(k)$. Passing to generic fibres, we get a morphism of adic spaces
\[
\det\from (\tilde{H}^{\ad}_{\eta})^h\to\tilde{\mu}_{p^\infty,\eta}^{\ad}
\]
which is $\Q_p$-alternating when considered as a map between sheaves of $\Q_p$-vector spaces on $\CAff_{\Spa(W(k)[1/p],W(k))}$.  This morphism makes the diagram
\[
\xymatrix{
(\tilde{H}^{\ad}_{\eta})^h \ar[r]^{\det} \ar[d]_{\qlog} & \tilde{\mu}_{p^\infty,\eta}^{\ad} \ar[d]^{\log} \\
\left( M(H)\otimes \mathbf{G}_a\right)^h \ar[r]_{\det} & M(\mu_{p^\infty,k})\otimes\mathbf{G}_a
}
\]
commute.

\begin{Theorem} There is a cartesian diagram
\[
\label{LTdiagram}
\xymatrix{
\mathcal{M}^h_{\infty} \ar[r]^{\det}\ar[d]  & \mathcal{M}^1_\infty \ar[d] \\
(\tilde{H}^{\ad}_{\eta})^h \ar[r]_{\det} & \tilde{\mu}_{p^\infty,\eta}^{\ad} .
}
\]
\end{Theorem}

\begin{proof} First, we have to prove that $\det\from \mathcal{M}^h_\infty\to \tilde{\mu}_{p^\infty,\eta}^\ad$ factors over $\mathcal{M}^1_\infty$. For this, suppose that the $h$-tuple $s_1,\dots,s_h$ represents a section of $\mathcal{M}_\infty^h$ over an affinoid algebra $(R,R^+)$. Since the matrix $(\qlog(s_1),\dots,\qlog(s_h))$ has rank $h-1$, we have $\det(\qlog(s_1),\ldots,\qlog(s_h))=0$, i.e. $\log(\det(s_1,\ldots,s_h))=0$. It remains to show that $\det(s_1,\ldots,s_h)(x)\neq 0$ at all points $x\in \Spa(R,R^+)$. For this, it is enough to check on geometric points $\Spa(C,\OO_C)$, where it is readily deduced from Corollary \ref{CPoints}, noting that $\bigwedge^h\mathscr{F}\injects \bigwedge^h \mathscr{E}$.

It remains to see that the square is cartesian. For this, let $s_1,\dots,s_h$ represent a section of the fibre product over $(R,R^+)$.  This means that $\det(s_1,\ldots,s_h)$ is a section of $V(\mu_{p^\infty})^{\ad}_{\eta}\backslash\set{0}$. We have $\det(\qlog(s_1),\ldots,\qlog(s_h))=\log(\det(s_1,\ldots,s_h))=0$, meaning that the matrix $(\qlog(s_1),\dots,\qlog(s_h))$ has rank at most $h-1$. Therefore to show that $s_1,\dots,s_h$ lies in $\mathcal{M}_\infty^h$, we only need to show that this matrix has rank exactly $h-1$, and that condition (ii) in Definition \ref{Minfprimedef} is satisfied. This can be checked on geometric points $\Spa(C,\OO_C)$.

The tuple $s_1,\dots,s_h$ corresponds to a map $\mathscr{F}\to\mathscr{E}$ of vector bundles on the Fargues-Fontaine curve $X$, with $\mathscr{F}=\mathcal{O}_X^h$ and $\mathscr{E}=\mathscr{E}(H)$. Since the section $s_1\wedge \dots \wedge s_h$ of $\mathscr{E}(\mu_{p^\infty})$ vanishes only at $\infty$, we have an exact sequence
\[ 0\to\mathscr{F}\to\mathscr{E}\to i_{\infty\ast}W\to 0\ ,\]
where $W$ is a $C$-vector space equal to the quotient of $M(H)\otimes_{W(k)} C$ by the span of $s_1,\dots,s_h$. Counting degrees, we get $\dim W = 1$, verifying the condition on the rank. Taking global sections, one gets the exact sequence
\[
0\to \Q_p^h\xrightarrow{(s_1,\ldots,s_h)} \tilde{H}^\ad_\eta(C,\OO_C)\to W\otimes_R C\to 0\ ,
\]
verifying condition (ii) in Definition \ref{Minfprimedef}.
\end{proof}

\begin{rmk} The space $\mathcal{M}^1_\infty=\bigsqcup_{\Z}\Spa(L,\OO_L)$ has an obvious integral model, namely $\hat{\mathcal{M}}^1_\infty=\bigsqcup_{\Z}\Spf\OO_L$.  This suggests defining an integral model $\hat{\mathcal{M}}^h_\infty$ as the fibre product of $\tilde{H}_{\eta}^h$ and $\hat{\mathcal{M}}^1_\infty$ over $\tilde{\mu}_{p^\infty}$.  On the other hand we have the integral models $\hat{\mathcal{M}}^h_n$ of finite level, due to Drinfeld. It is proved in \cite{WeinsteinSemistableModels} that $\hat{\mathcal{M}}$ is the inverse limit $\varprojlim \hat{\mathcal{M}}_n$ in the category of formal schemes over $\Spf\Z_p$.
\end{rmk}

\subsection{EL structures}
\label{ELstructures}

In this section, we generalize the previous results to Rapoport--Zink spaces of EL type. It would not be problematic to consider cases of PEL type. However, the group theory becomes more involved, and the present theory makes it possible to consider analogues of Rapoport--Zink spaces which are not of PEL type, so that we will leave this discussion to future work.

Let us fix a semisimple $\Q_p$-algebra $B$, a finite nondegenerate $B$-module $V$, and let $\G=\GL_B(V)$. Fix a maximal order $\OO_B\subset B$ and an $\OO_B$-stable lattice $\Lambda\subset V$. Further, fix a conjugacy class of cocharacters $\mu\from \Gm\to \G_{\bar{\mathbb{Q}}_p}$, such that in the corresponding decomposition of $V_{\bar{\mathbb{Q}}_p}$ into weight spaces, only weights $0$ and $1$ occur; write
\[
V_{\bar{\mathbb{Q}}_p} = V_0\oplus V_1
\]
for the corresponding weight decomposition, and set $d=\dim V_0$, $h=\dim V$, so that $h-d=\dim V_1$. Moreover, fix a $p$-divisible group $H$ of dimension $d$ and height $h$ over $k$ with action $\OO_B\to \End(H)$, such that $M(H)\otimes_{W(k)} W(k)[p^{-1}]\cong V\otimes_{\Q_p} W(k)[p^{-1}]$ as $B\otimes_{\Q_p} W(k)[p^{-1}]$-modules.

In the following, we write $\mathcal{D} = (B,V,\tilde{H},\mu)$ for the rational data, and $\mathcal{D}^\int = (\OO_B,\Lambda,H,\mu)$ for the integral data. Moreover, let $E$ be the reflex field, which is the field of definition of the conjugacy class of cocharacters $\mu$. Set $\breve{E} = E\cdot W(k)$.

\begin{defn} The Rapoport--Zink space $\mathcal{M}_{\mathcal{D}^\int}$ of EL type associated to $\mathcal{D}^\int$ is the functor on $\Nilp_{\OO_{\breve{E}}}$ sending $R$ to the set of isomorphism classes of $(G,\rho)$, where $G/R$ is a $p$-divisible group with action of $\OO_B$ satisfying the determinant condition, cf. \cite[3.23 a)]{RZ}, and
\[
\rho\from H\otimes_k R/p\to G\otimes_R R/p
\]
is an $\OO_B$-linear quasi-isogeny.
\end{defn}

Then Rapoport--Zink prove representability of $\mathcal{M}_{\mathcal{D}^\int}$.

\begin{Theorem} The functor $\mathcal{M}_{\mathcal{D}^\int}$ is representable by a formal scheme, which locally admits a finitely generated ideal of definition.
\end{Theorem}

On the generic fibre, we have the Grothendieck--Messing period morphism
\[
\pi_\GM\from (\mathcal{M}_{\mathcal{D}^\int})^\ad_\eta\to \Flag_\GM\ ,
\]
where the flag variety $\Flag_{\GM}$ parametrizes $B$-equivariant quotients $W$ of $M(H)\otimes_{W(k)} R$ which are finite projective $R$-modules, and locally on $R$ isomorphic to $V_0\otimes_{\Q_p} R$ as $B\otimes_{\Q_p} R$-modules.

For any $n\geq 1$, one also defines the cover $\mathcal{M}_{\mathcal{D}^\int,n}$ of its generic fibre $(\mathcal{M}_{\mathcal{D}^\int})^\ad_\eta$, as parametrizing $\OO_B$-linear maps $\Lambda/p^n\to G[p^n]^\ad_\eta$ which are isomorphisms at every point.

We have the following definition of $\mathcal{M}_{\mathcal{D}^\int,\infty}$.

\begin{defn} Consider the functor $\mathcal{M}_{\mathcal{D}^\int,\infty}$ on complete affinoid $(\breve{E},\OO_{\breve{E}})$-algebras, sending $(R,R^+)$ to the set of triples $(G,\rho,\alpha)$, where $(G,\rho)\in (\mathcal{M}_{\mathcal{D}^\int})^\ad_\eta(R,R^+)$, and
\[
\alpha\from \Lambda\to T(G)^\ad_\eta(R,R^+)
\]
is a morphism of $\OO_B$-modules such that for all $x=\Spa(K,K^+)\in \Spa(R,R^+)$, the induced map
\[
\alpha(x)\from \Lambda\to T(G)^\ad_\eta(K,K^+)
\]
is an isomorphism.
\end{defn}

Our arguments immediately generalize to the following result.

\begin{Theorem}\label{MInftyExistsEL} The functor $\mathcal{M}_{\mathcal{D}^\int,\infty}$ is representable by an adic space over $\Spa(\breve{E},\OO_{\breve{E}})$. The space $\mathcal{M}_{\mathcal{D}^\int,\infty}$ is preperfectoid, and
\[
\mathcal{M}_{\mathcal{D}^\int,\infty}\sim \varprojlim_n \mathcal{M}_{\mathcal{D}^\int,n}\ .
\]
Moreover, one has the following alternative description of $\mathcal{M}_{\mathcal{D}^\int,\infty} = \mathcal{M}_{\mathcal{D},\infty}$, which depends only on the rational data $\mathcal{D}$. The sheaf $\mathcal{M}_{\mathcal{D},\infty}$ is the sheafification of the functor sending a complete affinoid $(\breve{E},\OO_{\breve{E}})$-algebra $(R,R^+)$ to the set of maps of $B$-modules
\[
V\to \tilde{H}^\ad_\eta(R,R^+)
\]
for which the following conditions are satisfied.
\begin{altenumerate}
\item[{\rm (i)}] The quotient $W$ of $M(H)\otimes_{W(k)} R$ by the image of $V\otimes R$ is a finite projective $R$-module, which locally on $R$ is isomorphic to $V_0\otimes_{\Q_p} R$ as $B\otimes_{\Q_p} R$-module.
\item[{\rm (ii)}] For any geometric point $x=\Spa(C,\OO_C)\to \Spa(R,R^+)$, the sequence
\[
0\to V\to \tilde{H}^\ad_\eta(C,\OO_C)\to W\otimes_R C\to 0
\]
is exact.
\end{altenumerate}
\end{Theorem}

In particular, this gives a description of the image of the period morphism in the style of Theorem \ref{ImagePeriodMorphism} in all EL cases. Finally, let us observe that there are natural group actions on $\mathcal{M}_{\mathcal{D},\infty}$. Namely, there is an action of $\G(\Q_p)$ on $\mathcal{M}_{\mathcal{D},\infty}$, through its action on $V$. Let $J$ denote the group of $B$-linear self-quasiisogenies of $H$. Then $J(\Q_p)$ acts on $\mathcal{M}_{\mathcal{D},\infty}$ through its action on $\tilde{H}$. Recall that this action of $J(\Q_p)$ is in fact defined on every level $\mathcal{M}_{\mathcal{D}^\int,n}$, and already on $\mathcal{M}_{\mathcal{D}^\int}$, contrary to the action of $\G(\Q_p)$, which exists only in the inverse limit.

\begin{prop}\label{GroupActions} The action of $\G(\Q_p)\times J(\Q_p)$ on $\mathcal{M}_{\mathcal{D},\infty}$ is continuous, i.e. for any qcqs open $U=\Spa(R,R^+)\subset \mathcal{M}_{\mathcal{D},\infty}$, the stabilizer of $U$ in $\G(\Q_p)\times J(\Q_p)$ is open, and there is a base for the topology of $\mathcal{M}_{\mathcal{D},\infty}$ given by open affinoids $U = \Spa(R,R^+)$ for which the stabilizer of $U$ in $\G(\Q_p)\times J(\Q_p)$ acts continuously on $R$.

Moreover, the Grothendieck--Messing period map
\[
\pi_\GM\from \mathcal{M}_{\mathcal{D},\infty}\to \Flag_\GM
\]
is $\G(\Q_p)\times J(\Q_p)$-equivariant for the trivial action of $\G(\Q_p)$ on $\Flag_\GM$, and the action of $J(\Q_p)$ on $\Flag_\GM$ induced from the action of $J(\Q_p)$ on $M(H)\otimes_{W(k)} W(k)[p^{-1}]$.
\end{prop}

\begin{proof} The first part is a direct consequence of the continuity of the action of $J(\Q_p)$ on each $\mathcal{M}_{\mathcal{D}^\int,n}$, and the relation
\[
\mathcal{M}_{\mathcal{D},\infty}\sim \varprojlim \mathcal{M}_{\mathcal{D}^\int,n}\ .
\]
The equivariance of the Grothendieck--Messing period map is obvious from the definitions.
\end{proof}

\section{Duality of Rapoport--Zink spaces}

In this section, we deduce a general duality isomorphism between basic Rapoport--Zink spaces at infinite level, in the EL case. We keep the notation from the previous subsection concerning the EL case.

\subsection{The Hodge--Tate period morphism}

First, we observe that in addition to the usual Grothendieck--Messing period morphism, there is a second morphism at infinite level, coming from the Hodge--Tate map. Let $\Flag_\HT$ be the adic space over $\Spa(E,\OO_E)$ parametrizing $B$-equivariant quotients $W_1$ of $V\otimes_{\Q_p} R$ that are finite projective $R$-modules, and locally on $R$ isomorphic to $V_1\otimes_{\Q_p} R$ as $B\otimes_{\Q_p} R$-module. Note that $\G$ acts naturally on $\Flag_\HT$.

\begin{prop} There is a Hodge--Tate period map
\[
\pi_\HT\from \mathcal{M}_{\mathcal{D},\infty}\to \Flag_\HT\ ,
\]
sending an $(R,R^+)$-valued point of $\mathcal{M}_{\mathcal{D},\infty}$, given by a map $V\to \tilde{H}^\ad_\eta(R,R^+)$, to the quotient of $V\otimes_{\Q_p} R$ given as the image of the map
\[
V\otimes_{\Q_p} R\to M(H)\otimes_{W(k)} R\ .
\]
The Hodge--Tate period map is $\G(\Q_p)\times J(\Q_p)$-equivariant for the canonical action of $\G(\Q_p)$ on $\Flag_\HT$, and the trivial action of $J(\Q_p)$ on $\Flag_\HT$.
\end{prop}

\begin{proof} Locally, the point of $\mathcal{M}_{\mathcal{D},\infty}$ corresponds to a $p$-divisible group $G$ over an open and bounded subring $R_0\subset R^+$. Then we have an exact sequence
\[
0\to (\Lie G^\vee)^\vee\otimes R\to M(H)\otimes_{W(k)} R\to \Lie G\otimes R\to 0\ .
\]
Moreover, we get a map $V\otimes_{\Q_p} R\to (\Lie G^\vee)^\vee\otimes R$. We claim that it is surjective. Indeed, it suffices to check this on geometric points, where it follows from Proposition \ref{PDivGroupGivesModVectBund} (iii). As $M(H)\otimes_{W(k)} R\cong V\otimes_{\Q_p} R$ and $\Lie G\otimes R\cong V_0\otimes_{\Q_p} R$ locally, it follows that $(\Lie G^\vee)^\vee\otimes R\cong V_1\otimes_{\Q_p} R$ locally, as desired.
\end{proof}

\begin{rmk} We caution the reader that unlike $\pi_\GM$, $\pi_\HT$ is not in general \'etale. Also, the fibres are not $J(\Q_p)$-torsors in general. However, in the basic case to be considered in the next section, these statements are true, and follow by identification from the similar statements for $\pi_\GM$ by duality.
\end{rmk}

\subsection{The duality isomorphism}

Assume now that $k$ is algebraically closed. Recall the following proposition.

\begin{prop}\label{propbasic} The $\sigma$-conjugacy class corresponding to $H$ is basic if and only if $\check{B} = \End_B^\circ(H)$ (where $\End^\circ = \End\otimes \Q$) satisfies
\[
\check{B}\otimes_{\Q_p} W(k)[p^{-1}] = \End_{B\otimes W(k)}(M(H)[p^{-1}])\ ,
\]
where the right-hand side denotes the $B\otimes W(k)$-linear endomorphisms of the $B\otimes W(k)$-module $M(H)[p^{-1}]$, not necessarily compatible with $F$ and $V$.
\end{prop}

In the following, assume that the conditions of the proposition are satisfied. In particular, it follows that we have a canonical identification
\[\begin{aligned}
\check{B}\otimes_{\Q_p} W(k)[p^{-1}] &= \End_{B\otimes W(k)}(M(H)[p^{-1}])\\
& = \End_{B\otimes W(k)}(V\otimes_{\Q_p} W(k)[p^{-1}]) =  \End_B(V)\otimes_{\Q_p} W(k)[p^{-1}]\ .
\end{aligned}\]
We define dual EL data as follows. First, $\check{B}$ is defined as in the proposition, with $\OO_{\check{B}} = \End_{\OO_B}(H)$. Moreover, $\check{V} = \check{B}$ and $\check{\Lambda} = \OO_{\check{B}}$, with the left action by $\check{B}$, resp. $\OO_{\check{B}}$. It follows that $\check{\G}\cong J$ is the group of $B$-linear self-quasiisogenies of $H$, where $g\in \check{\G} = J\subset \check{B}$ acts on $\check{V}$ by multiplication with $g^{-1}$ from the right.

Next, we set $\check{H} = \Lambda^\ast\otimes_{\OO_B} H$, where $\Lambda^\ast = \Hom_{\OO_B}(\Lambda,\OO_B)$. To make sense of this equation, interpret it for example in terms of Dieudonn\'e modules, so that
\[
M(\check{H}) = \Lambda^\ast\otimes_{\OO_B} M(H)\ .
\]
As $\OO_{\check{B}}$ acts naturally on $H$, there is an induced action of $\OO_{\check{B}}$ on $\check{H}$. One has an identification
\[\begin{aligned}
M(\check{H})[p^{-1}] = V^\ast\otimes_B M(H) &= (V^\ast\otimes_B V)\otimes_{\Q_p} W(k)[p^{-1}]\\
& = \End_B(V)\otimes_{\Q_p} W(k)[p^{-1}] = \check{B}\otimes_{\Q_p} W(k)[p^{-1}]\ ,
\end{aligned}\]
as desired. It remains to define the conjugacy class of cocharacters $\check{\mu}\from \Gm\to \check{\G}$. Equivalently, we have to give the corresponding weight decomposition of $\check{V}$. We require that this is given by the decomposition
\[\begin{aligned}
\check{V}\otimes_{\Q_p} &W(k)[p^{-1}] = \End_{B}(V)\otimes_{\Q_p} W(k)[p^{-1}]\\
&= (\Hom_{B\otimes W(k)}(V_0,V)\otimes_{\Q_p} W(k)[p^{-1}])\oplus (\Hom_{B\otimes W(k)}(V_1,V)\otimes_{\Q_p} W(k)[p^{-1}])\ ,
\end{aligned}\]
where the first summand is by definition of weight $0$, and the second summand of weight $1$. (We note that we would have to extend the scalars further to make sense of this equation, as $V_0$ and $V_1$ are not defined over the base field.)

Let us denote by $\check{\Flag}_\GM$ and $\check{\Flag}_\HT$ the flag varieties for the dual EL data; we consider all spaces as being defined over $\Spa(\breve{E},\OO_{\breve{E}})$ in the following.

\begin{prop} There are natural actions of $\G$ on $\Flag_\HT$ and $\check{\Flag}_\GM$, and a canonical $\G$-equivariant isomorphism $\Flag_\HT\cong \check{\Flag}_\GM$. Similarly, there are natural actions of $\check{\G}$ on $\check{\Flag}_\HT$ and $\Flag_\GM$, and a canonical $\check{\G}$-equivariant isomorphism $\check{\Flag}_\HT\cong \Flag_\GM$.
\end{prop}

\begin{proof} Recall that $\Flag_\HT$ parametrizes $B$-equivariant quotients $V\otimes R\to W_1$ which are locally of the form $V_1\otimes R$. On the dual side, $\check{\Flag}_\GM$ parametrizes $\End_B(V)$-equivariant quotients
\[
M(\check{H})\otimes_{W(k)} R = \End_B(V)\otimes_{\Q_p} R\to \check{W}
\]
which are locally of the form $\Hom_B(V_0,V)\otimes_{\Q_p} R$. Starting with $0\to W_0\to V\otimes R\to W_1\to 0$, and looking at $\Hom_{B\otimes R}(-,V\otimes R)$ gives a sequence
\[
0\to \Hom_{B\otimes R}(W_1,V\otimes R)\to \End_B(V)\otimes R\to \Hom_{B\otimes R}(W_0,V\otimes R)\to 0\ ,
\]
where $\check{W} = \Hom_B(W_0,V)$ is locally of the form $\Hom_B(V_0,V)\otimes R$. It is easy to see that this gives an isomorphism of flag varieties, under which the canonical action of $\G$ on $\Flag_\HT$ gets identified with an action of $\G$ on $\check{\Flag}_\GM$.

In the other case, recall that $\Flag_\GM$ parametrizes $B$-equivariant quotients $M(H)\otimes R\to W$ which are locally of the form $V_0\otimes R$. On the dual side, $\check{\Flag}_\HT$ parametrizes $\End_B(V)$-equivariant quotients
\[
\check{V}\otimes_{\Q_p} R = \End_B(V)\otimes_{\Q_p} R\to \check{W}_1
\]
which are locally of the form $\Hom_B(V_1,V)\otimes_{\Q_p} R$. Again, starting with
\[
0\to W_1\to M(H)\otimes R\to W\to 0
\]
and looking at $\Hom_{B\otimes R}(-,V\otimes R)$ produces a quotient $\check{W}_1 = \Hom_{B\otimes R}(W_1,V\otimes R)$ of $\End_B(V)\otimes R$ which is locally of the form $\Hom_B(V_1,V)\otimes R$, as desired.
\end{proof}

Let $\mathcal{M}_{\mathcal{D},\infty}$ be the Rapoport--Zink space at infinite level for our original EL data, and let $\M_{\check{\mathcal{D}},\infty}$ denote the Rapoport--Zink space at infinite level for the dual EL data.

\begin{Theorem} There is a natural $\G(\Q_p)\times \check{\G}(\Q_p)$-equivariant isomorphism
\[
\mathcal{M}_{\mathcal{D},\infty}\cong \M_{\check{\mathcal{D}},\infty}\ ,
\]
under which $\pi_\GM\from \mathcal{M}_{\mathcal{D},\infty}\to \Flag_\GM$ gets identified with $\check{\pi}_\HT\from \M_{\check{\mathcal{D}},\infty}\to \check{\Flag}_\HT$, and vice versa.
\end{Theorem}

\begin{proof} Let $(R,R^+)$ be any complete affinoid $(\breve{E},\OO_{\breve{E}})$-algebra. Then recall that $\mathcal{M}_{\mathcal{D},\infty}$ is the sheafification of the functor sending $(R,R^+)$ to the set of $B$-linear maps
\[
s\from V\to \tilde{H}^\ad_\eta(R,R^+)
\]
for which the induced quotient $M(H)\otimes R\to W$ by the image of $V\otimes R$ is locally of the form $V_0\otimes R$, and which give exact sequences at every geometric point. Similarly, $\M_{\check{\mathcal{D}},\infty}$ is the sheafification of the functor sending $(R,R^+)$ to the set of $\End_B(V)$-linear maps
\[
\check{s}\from \End_B^\circ(H)\to V^\ast\otimes_B \tilde{H}^\ad_\eta(R,R^+)
\]
for which the induced quotient $M(\check{H})\otimes_{W(k)} R=\End_B(V)\otimes_{\Q_p} R\to \check{W}$ by the image of $\End_{\OO_B}(H)\otimes R$ is locally of the form $\End_B(V_0,V)\otimes R$, and which give exact sequences at every geometric point. Let us first check that $s$ and $\check{s}$ correspond. Indeed, define
\[
(\check{s}(f))(v) = f(s(v))\ ,
\]
where $f\in \End_B^\circ(H)$ and $v\in V$. One readily checks that this defines a bijective correspondence between $s$ and $\check{s}$.

Now consider the map
\[
\End_B^\circ(H)\otimes R\to M(\check{H})\otimes_{W(k)} R = \End_B(V)\otimes_{\Q_p} R
\]
induced by $\check{s}$. We claim that it is the same as $\Hom_{B\otimes R}(-,V\otimes R)$ applied to the map
\[
V\otimes_{\Q_p} R\to M(H)\otimes R
\]
induced from $s$; this is a direct verification. From here, one deduces that $W$ is locally of the form $V_0\otimes R$ if and only if $\check{W}$ is locally of the form $\Hom_B(V_0,V)\otimes R$, as well as the compatibility with period maps. Moreover, the exactness conditions at geometric points correspond.
\end{proof}

\subsection{Equivariant covers of Drinfeld's upper half-space}
\label{equivariantcovers}

Fix a finite extension $F$ of $\Q_p$, and let $\Omega\subset \mathbf{P}^{n-1}$ denote Drinfeld's upper half-space, where we consider everything as an adic space over $\Spa(C,\OO_C)$, where our field of scalars $C$ is a complete algebraically closed extension of $F$. As an application of the duality of Rapoport--Zink spaces, we prove the following result.

\begin{Theorem}\label{splittingtheorem} Let $\tilde{\Omega}\to \Omega$ be a finite \'etale $\GL_n(F)$-equivariant morphism of adic spaces. There exists an $m\geq 0$ such that $\mathcal{M}^\Dr_m\to\Omega$ factors through $\tilde{\Omega}$.
\end{Theorem}

\begin{proof} Let $\hat{\mathcal{M}}^\LT_\infty$ be (the strong completion of the base-change to $C$ of) the Lubin-Tate space of height $n$ for $F$ at infinite level, and let $\hat{\mathcal{M}}^\Dr_\infty$ be (the strong completion of the base-change to $C$ of) the Drinfeld space of height $n$ for $F$ at infinite level. Then both $\hat{\mathcal{M}}^\LT_\infty$ and $\hat{\mathcal{M}}^\Dr_\infty$ are perfectoid spaces over $\Spa(C,\OO_C)$. By duality, we have the $\GL_n(F)\times D^\times$-equivariant isomorphism $\hat{\mathcal{M}}^\LT_\infty\cong \hat{\mathcal{M}}^\Dr_\infty$, where $D$ is the division algebra over $F$ of invariant $\frac 1n$. (We note that the $\GL_n(F)$-action which comes from the general duality statement differs from the standard $\GL_n(F)$-action on the Drinfeld tower by $g\mapsto (g^{-1})^t$. This does not change any of the statements to follow.)

Let $\tilde{\hat{\M}}^\Dr_\infty$ be the fibre product $\tilde{\Omega}\times_{\Omega} \hat{\M}^\Dr_\infty$. Then $\tilde{\hat{\M}}^\Dr_\infty\to\hat{\M}^\Dr_\infty$ is a $\GL_n(F)$-equivariant finite \'etale morphism of perfectoid spaces over $\Spa(C,\OO_C)$. By the duality isomorphism, it gives rise to a $\GL_n(F)$-equivariant finite \'etale morphism $\tilde{\hat{\M}}^\LT_\infty\to \hat{\M}^\LT_\infty$.

\begin{lemma} The category of adic spaces finite \'etale over $\mathbf{P}^{n-1}$ is equivalent to the $\GL_n(F)$-equivariant spaces finite \'etale over $\hat{\M}^\LT_\infty$.
\end{lemma}

\begin{proof} We can work locally over $\mathbf{P}^{n-1}$. Let us write $\mathcal{M}^\LT$ for the generic fibre of the Lubin -- Tate space for $F$ at level $0$. As $\pi\from \mathcal{M}^\LT\to \mathbf{P}^{n-1}$ admits local sections, we can assume that we work over an open affinoid subset $U\subset \mathbf{P}^{n-1}$ which admits a lifting $U\to \mathcal{M}^\LT$. Let $U_\infty^\prime\subset \hat{\M}^\LT_\infty$ be the preimage of $U\subset \mathbf{P}^{n-1}$, and let $U_\infty\subset \hat{\M}^\LT_\infty$ denote the preimage of $U\subset \M^\LT$, so that $U_\infty\subset U_\infty^\prime$. Then $U_\infty\to U$ is a $\GL_n(\OO_F)$-equivariant map, and $U_\infty^\prime\to U$ is a $\GL_n(F)$-equivariant map.

The action of $\GL_n(F)$ on $U_\infty^\prime$ gives a disjoint decomposition
\[
U_\infty^\prime = \bigsqcup_{g\in \GL_n(F)/\GL_n(\OO_F)} gU_\infty\ .
\]
In particular, the category of $\GL_n(F)$-equivariant covers of $U_\infty^\prime$ is equivalent to the category of $\GL_n(\OO_F)$-equivariant covers of $U_\infty$, and we can restrict attention to the $\GL_n(\OO_F)$-equivariant map $U_\infty\to U$.

With the obvious definition of $U_m\subset \mathcal{M}^\LT_m$, we have $U_\infty\sim \varprojlim U_m$. In particular, a finite \'etale morphism of $V_\infty\to U_\infty$ comes via pullback from a finite \'etale morphism $V_m\to U_m$ for some $m$. We get an induced morphism
\[
\Gamma(V_m,\mathcal{O}_{V_m})\to \Gamma(V_\infty,\mathcal{O}_{V_\infty})\ ,
\]
where we also recall that all spaces are affinoid. We claim that some open subgroup of $\GL_n(\OO_F)$ fixes the image of this morphism point-wise. As $\Gamma(V_m,\mathcal{O}_{V_m})$ is finite over $\Gamma(U,\mathcal{O}_U)$, it is enough to check that any element $x\in \Gamma(V_m,\mathcal{O}_{V_m})$ is fixed by an open subgroup. But $x$ satisfies a monic polynomial equation $P(x)=0$ with coefficients in $\Gamma(U,\mathcal{O}_U^+)$, and finite \'etale in characteristic $0$. Any element $g\in \GL_n(\OO_F)$ will map $x$ to another solution $gx$ of $P(gx)=0$. But by Hensel's lemma, if $gx$ close enough $p$-adically to $x$, then $gx=x$. But as $\GL_n(\OO_F)$ acts continuously, we get the claim.

After enlarging $m$, we can assume that $\ker(\GL_n(\OO_F)\to \GL_n(\OO_F/p^m))$ acts trivially on $\Gamma(V_m,\mathcal{O}_{V_m})$. Then $V_m$ is $\GL_n(\OO_F/p^m)$-equivariant finite \'etale over $U_m$, and by usual finite \'etale Galois descent, this gives rise to a finite \'etale $V\to U$. It is easy to see that this gives the desired equivalence of categories.
\end{proof}

But now we recall that by rigid GAGA, finite \'etale maps to $\mathbf{P}^{n-1}$ are algebraic, and all of these split. It follows that the same is true for $\tilde{\hat{\M}}^\LT_\infty\to \hat{\M}^\LT_\infty$, and then for $\tilde{\hat{\M}}^\Dr_\infty\to \hat{\M}^\Dr_\infty$. In particular, there is a $\GL_n(F)$-equivariant morphism
\[
\hat{\M}^\Dr_\infty\to \tilde{\hat{\M}}^\Dr_\infty\to \tilde{\Omega}\ .
\]
It remains to prove that this morphism factors over $\M^\Dr_m$ for some $m$. As $\GL_n(F)$ permutes the connected components of $\M^\Dr$ transitively, it is enough to check this for one connected component; so fix a connected component $(\M^\Dr)^0\subset \M^\Dr$ of the Drinfeld space at level $0$, and denote by similar symbols its preimage. Observe that $\Omega$ modulo the action of $\GL_n(F)$ is quasicompact; in particular, it suffices to check the desired factorization property locally over $\Omega$, so let $U\subset \Omega$ be some open affinoid subset, and let $\tilde{U}\subset \tilde{\Omega}$ and $U_\infty\subset (\hat{\M}^\Dr_\infty)^0$ be the preimages. We get a map
\[
\Gamma(\tilde{U},\OO_{\tilde{U}})\to \Gamma(U_\infty,\OO_{U_\infty})\ ,
\]
where the right-hand side is the completion of $\varinjlim \Gamma(U_m,\OO_{U_m})$, with the obvious definition of $U_m\subset (\M^\Dr_m)^0$. But finite \'etale algebras over $\Gamma(U_\infty,\OO_{U_\infty})$ and $\varinjlim \Gamma(U_m,\OO_{U_m})$ are equivalent, so that we get a map
\[
\Gamma(\tilde{U},\OO_{\tilde{U}})\to \varinjlim \Gamma(U_m,\OO_{U_m})\ ,
\]
i.e. a map $\Gamma(\tilde{U},\OO_{\tilde{U}})\to \Gamma(U_m,\OO_{U_m})$ for some $m$, giving the desired factorization $U_\infty\to U_m\to \tilde{U}$.
\end{proof}

\bibliographystyle{amsalpha}
\bibliography{Moduli}

\end{document}